\documentclass[12pt, oneside]{article}   	
\usepackage{geometry}                		
\usepackage{amsthm}
\usepackage{amsmath}
\usepackage{mathtools}
\usepackage{graphicx}				
							 	
\usepackage{amssymb}
\usepackage[all]{xy}
\usepackage{makeidx}
\usepackage{nameref}
\usepackage{xcolor}
\usepackage{tikz-cd}
\usetikzlibrary{arrows,decorations.pathreplacing,angles,quotes,bending}
\usepackage{mathscinet}
\usepackage{enumerate}
\usepackage{float}
\usepackage{enumitem}
\usepackage[font=scriptsize]{caption}
\makeindex
\theoremstyle{plain}
\makeatletter
\newcommand{\newreptheorem}[2]{\newtheorem*{rep@#1}{\rep@title}\newenvironment{rep#1}[1]{\def\rep@title{#2 \ref*{##1}}\begin{rep@#1}}{\end{rep@#1}}}
\makeatother

\newtheorem{theorem}{Theorem}[section]
\newreptheorem{theorem}{Theorem}
\newtheorem{lemma}[theorem]{Lemma}
\newtheorem{corollary}[theorem]{Corollary}
\newtheorem{proposition}[theorem]{Proposition}
\theoremstyle{definition}
\newtheorem{definition}[theorem]{Definition}
\newtheorem{setup}[theorem]{Setup}

\title{The complete classification of isotopy classes of degree three symplectic curves in $\mathbb{CP}^2$ via a novel algebraic theory of braid monodromy}

\author{Amitesh Datta}
\begin{document}
\maketitle

\abstract{We develop a new algebraic theory of positive braids and conjugacy classes in the braid group $B_3$. We use our theory to establish a complete classification of isotopy classes of degree three symplectic curves in $\mathbb{CP}^2$ with only $A_n$-singularities for $n\geq 1$ (an $A_n$-singularity is locally modelled by the equation $z^2 = w^n$) independent of Gromov's theory of pseudoholomorphic curves. We show that if $C$ and $C'$ are degree three symplectic curves in $\mathbb{CP}^2$ with the same numbers of $A_n$-singularities for each $n\geq 1$, then $C$ is isotopic to $C'$. Furthermore, our theory furnishes a single method of proof that independently establishes and unifies several fundamental classification results on degree three symplectic curves in $\mathbb{CP}^2$. In particular, we prove using our theory: (1) there is a unique isotopy class of degree three smooth symplectic curves in $\mathbb{CP}^2$ (a result due to Sikorav), (2) the number of nodes is a complete invariant of the isotopy class of a degree three nodal symplectic curve in $\mathbb{CP}^2$ (the case of irreducible nodal curves is due to Shevchishin and the case of reducible nodal curves is due to Golla-Starkston), and (3) there is a unique isotopy class of degree three cuspidal symplectic curves in $\mathbb{CP}^2$ (a generalization of a result due to Ohta-Ono). The present work represents the first step toward resolving the symplectic isotopy conjecture using purely algebraic techniques in the theory of braid groups. Finally, we independently establish a complete classification of genus one Lefschetz fibrations over $\mathbb{S}^2$ (a result due to Moishezon-Livne).}

\tableofcontents

\section{Introduction}
In this paper, we develop a new algebraic theory of the braid group $B_3$ in order to establish a complete classification of isotopy classes of degree three symplectic curves in $\mathbb{CP}^2$ with only $A_n$-singularities, and isomorphism classes of genus one Lefschetz fibrations over $\mathbb{S}^2$. An $A_n$-singularity is locally modelled by the equation $z^2 = w^n$ (which includes nodal and cuspidal singularities).

Firstly, the theory unifies and independently establishes results due to the authors Sikorav~\cite{sikoravsmoothdegreethree} (smooth symplectic curves), Shevchishin~\cite{shevchishinnodes} (irreducible nodal symplectic curves), Golla-Starkston~\cite{gollastarkstonreducible} (reducible nodal symplectic curves), Ohta-Ono~\cite{ohtaonocusp} (cuspidal symplectic curves with a single cusp), and Moishezon-Livne~\cite{moishezongenusonelefschetzfibrations} (genus one Lefschetz fibrations). The novelty in our methods is that we avoid Gromov's theory of pseudoholomorphic curves, which was the main tool in the original proofs of these results. Instead, we use our new theory in order to classify braid monodromy factorizations of degree three symplectic curves in $\mathbb{CP}^2$. Moreover, we establish all of these results with a single method of proof.

We also use our theory to establish complete constraints on the possible singularities of a degree three symplectic curve in $\mathbb{CP}^2$ with only $A_n$-singularities. The classical approach to establishing such constraints on singularities of curves in algebraic geometry is based on the classical B{\'e}zout theorem (for algebraic curves). The approach using our theory of braid monodromy is independent of B{\'e}zout's theorem and any generalization to symplectic curves based on Gromov's theory of pseudoholomorphic curves. The present work is a purely algebraic proof of the complete classification of isotopy classes of degree three symplectic curves in $\mathbb{CP}^2$ with only $A_n$-singularities.

\subsection{Summary of the main results}

We will work in the setting of simple Hurwitz curves in $\mathbb{CP}^2$ since every symplectic curve in $\mathbb{CP}^2$ with only $A_n$-singularities is symplectically isotopic to a simple Hurwitz curve. We briefly recall the relevant definitions (of Hurwitz curves, isotopies of Hurwitz curves, and braid monodromy) before summarizing the main results of the paper and elaborating on the details of our algebraic theory of the braid group $B_3$. We refer the reader to Section~\ref{sbackgroundmotivation} for a more detailed discussion of the background. Let $\mathbb{CP}^n$ denote complex projective $n$-space. Let $\pi:\mathbb{CP}^2\setminus \left[0:0:1\right]\to \mathbb{CP}^1$ be the projection defined in homogeneous coordinates by the rule $\pi\left[z_0:z_1:z_2\right] = \left[z_0:z_1\right]$. 

A \textit{Hurwitz curve} in $\mathbb{CP}^2$ is a two-dimensional oriented real submanifold $C\subseteq \mathbb{CP}^2$ with isolated singularities such that (1) the set $S$ of singularities in $C$ and nondegenerate tangencies of $C$ to the fibers of $\pi$ is a finite set, (2) $\left[0:0:1\right]\not\in C$, and (3) $C$ intersects the fibers of $\pi$ positively except at the points of $S$. The \textit{degree of a Hurwitz curve} is the algebraic intersection number $\left[C\right]\cdot \left[L\right]>0$, where $L\subseteq \mathbb{CP}^2$ is a complex projective line. 

An \textit{$A_n$-singularity} of a Hurwitz curve $C\subseteq \mathbb{CP}^2$ for $n\geq 1$ is a singularity of $C$ locally modelled by the equation $z^2 = w^n$ in local complex coordinates $\left(z,w\right)$ near $\left(0,0\right)$, where the level sets of the $z$-coordinates are transverse to the fibers of $\pi$. For example, an $A_1$-singularity of $C$ is a nondegenerate tangency of $C$ to a fiber of $\pi$, an $A_2$-singularity is a \textit{nodal singularity}, an $A_3$-singularity is a \textit{cuspidal singularity}, and an $A_4$-singularity is a \textit{tacnodal singularity}. (We note that our terminology differs slightly from the literature - in the literature, an $A_n$-singularity is locally modelled by the equation $z^2=w^{n+1}$ - but our terminology is convenient for our purposes.)

An \textit{isotopy of Hurwitz curves} is a one-parameter family $\{C_t\}_{t\in \left[0,1\right]}$ of Hurwitz curves such that the singularities of $C_t$ are the same as the singularities of $C_0$ for each $t\in \left[0,1\right]$. A Hurwitz curve $C\subseteq \mathbb{CP}^2$ is a \textit{simple Hurwitz curve} if the singularities of $C$ are all $A_n$-singularities.

The first main result of the paper is the following complete resolution of the isotopy problem for degree three simple Hurwitz curves in $\mathbb{CP}^2$.

\begin{theorem}
\label{tmainHurwitzcurve}
If $C,C'\subseteq \mathbb{CP}^2$ are simple Hurwitz curves of degree three with the same number of $A_n$-singularities for each positive integer $n\geq 1$, then $C$ is isotopic to $C'$.
\end{theorem}

The second main result is the following complete classification of the possible singularities of degree three simple Hurwitz curves in $\mathbb{CP}^2$.

\begin{theorem}
\label{tmainsingularities}
Let $C\subseteq \mathbb{CP}^2$ be a degree three simple Hurwitz curve. Let $\nu_n$ be the number of $A_n$-singularities in $C$ for each positive integer $n\geq 1$. We have the singularity formula $6 = \sum_{n=1}^{6} n\nu_n$ and the following constraints on the singularities of $C$:
\begin{description}
\item[(i)] $C$ has at most one $A_3$-singularity (cuspidal singularity).
\item[(ii)] $C$ does not simultaneously have an $A_2$-singularity (nodal singularity) and an $A_3$-singularity (cuspidal singularity).
\item[(iii)] $C$ does not simultaneously have an $A_2$-singularity (nodal singularity) and an $A_4$-singularity (tacnodal singularity).
\item[(iv)] $C$ does not have an $A_n$-singularity for $n\geq 5$.
\end{description}
\end{theorem}

A reformulation of Theorem~\ref{tmainsingularities} is that if $C\subseteq \mathbb{CP}^2$ is a degree three simple Hurwitz curve, then either $C$ is smooth, a nodal curve with at most three nodes, a cuspidal curve with a single cusp, or a curve with a single tacnode and no other singularities. The following statement is a consequence of Theorem~\ref{tmainHurwitzcurve} and Theorem~\ref{tmainsingularities}.

\begin{theorem}
\label{tmainHurwitzalgebraic}
If $C\subseteq \mathbb{CP}^2$ is a degree three simple Hurwitz curve, then $C$ is isotopic to a projective algebraic curve.
\end{theorem}

If $C\subseteq \mathbb{CP}^2$ is a degree three simple Hurwitz curve, then Theorem~\ref{tmainsingularities} and Theorem~\ref{tmainHurwitzalgebraic} imply that $C$ is either isotopic to a smooth cubic curve (if $C$ is smooth), a nodal cubic curve (if $C$ has a single node), the transverse union of a line and an irreducible conic (if $C$ has two nodes), the transverse union of three lines without a common intersection (if $C$ has three nodes), a cuspidal cubic curve (if $C$ has a single cusp), or the union of an irreducible conic and a tangent line to the conic (if $C$ has a tacnode).

The symplectic isotopy conjecture posits that every smooth or nodal symplectic curve in $\mathbb{CP}^2$ is symplectically isotopic to a projective algebraic curve in $\mathbb{CP}^2$. Theorem~\ref{tmainHurwitzalgebraic} is equivalent to the following statement.

\begin{theorem}
\label{tmainsymplectic}
If $C\subseteq \mathbb{CP}^2$ is a symplectic curve of degree three, where the singularities of $C$ are all simple (e.g., nodal, cuspidal, or tacnodal singularities), then $C$ is symplectically isotopic to a projective algebraic curve.
\end{theorem}

The importance of the symplectic isotopy conjecture stems partially from the work of Auroux~\cite{aurouxsymplecticbranchedcover} and Auroux-Katzarkov~\cite{aurouxsymplecticinvariants}. The upshot of these works is that in principle, the classification of closed symplectic $4$-manifolds reduces to the isotopy problem for Hurwitz curves in $\mathbb{CP}^2$ with only nodal and cuspidal singularities. Furthermore, the classification of which closed symplectic $4$-manifolds are complex projective surfaces reduces to the symplectic isotopy problem (i.e., which symplectic curves with only nodal and cuspidal singularities are symplectically isotopic to a projective algebraic curve?).

We remark that Sikorav established Theorem~\ref{tmainsymplectic} in the special case of smooth curves~\cite{sikoravsmoothdegreethree}, Shevchishin established Theorem~\ref{tmainsymplectic} in the special case of irreducible nodal curves~\cite{shevchishinnodes}, Golla-Starkston established Theorem~\ref{tmainsymplectic} in the special case of reducible nodal curves~\cite{gollastarkstonreducible}, and Ohta-Ono established Theorem~\ref{tmainsymplectic} in the special case of cuspidal curves with a single cusp~\cite{ohtaonocusp}. The methods used by these authors are based on Gromov's theory of pseudoholomorphic curves.

We use the theory of braid monodromy of simple Hurwitz curves in $\mathbb{CP}^2$ in order to rephrase the main results in this paper as algebraic statements in the braid group $B_3$. Indeed, the braid monodromy of a degree $d$ simple Hurwitz curve $C\subseteq \mathbb{CP}^2$ is a factorization of $\Delta^2$ into powers of positive half-twists in the braid group $B_d$, where $\Delta$ is the Garside element in $B_d$. The number of factors in this factorization is equal to the number of singularities in $C$ and nondegenerate tangencies of $C$ to the fibers of $\pi$. The monodromy of $C$ with respect to an $A_n$-singularity in $C$ is the $n$th power of a positive half-twist in $B_d$. The braid monodromy of $C$ is well-defined up to Hurwitz equivalence, which is an equivalence relation on the set of factorizations of $\Delta^2$. Kulikov-Teicher~\cite{kulikovteicherbraidmonodromy} proved that the Hurwitz equivalence class of the braid monodromy of a Hurwitz curve $C$ with only nodal and cuspidal singularities is a complete invariant of the isotopy class of $C$. Kharlamov-Kulikov~\cite{kulikovkharlamovbraidmonodromy} extended this result to all simple Hurwitz curves in $\mathbb{CP}^2$. We refer the reader to Section~\ref{sbackgroundmotivation} for a more detailed discussion of the background.

In this paper, we classify factorizations of $\Delta^2$ into powers of positive half-twists in the braid group $B_3$ using novel algebraic techniques based on a theory of positive braids and conjugacy classes in the braid group $B_3$, which we develop in Section~\ref{sB3background}. Indeed, if \[B_3 = \left\langle \sigma_1,\sigma_2 : \sigma_1\sigma_2\sigma_1 = \sigma_2\sigma_1\sigma_2\right\rangle\] denotes the Artin presentation of the braid group $B_3$, then $\Delta = \sigma_1\sigma_2\sigma_1 = \sigma_2\sigma_1\sigma_2$, and the $n$th power of a positive half-twist is an element of the conjugacy class of $\sigma_1^n$. In Subsection~\ref{ssdualitypositivebraid}, we essentially develop a theory of equations in $B_3$, and in Subsection~\ref{sspowerpositivehalftwist}, we establish an explicit algebraic characterization of powers of positive half-twists in $B_3$. In Section~\ref{smain}, we use these results in order to classify factorizations of $\Delta^2$ into powers of positive half-twists in $B_3$. 

In a sequel (\cite{dattageneralclassificationHurwitzcurves}) to the present paper, we generalize our algebraic theory of the braid group $B_3$ to the braid groups $B_d$ for all positive integers $d\geq 1$. We use this algebraic theory in order to prove that every smooth symplectic curve in $\mathbb{CP}^2$ (without restriction on the degree) is isotopic to a projective algebraic curve, affirmatively establishing the symplectic isotopy conjecture in full generality. The present work represents the first step toward resolving the symplectic isotopy conjecture using the purely algebraic techniques of braid monodromy.

In this paper, we also handle the case of singular degree three simple Hurwitz curves in $\mathbb{CP}^2$. Indeed, an analog of Theorem~\ref{tmainsymplectic} is not true for symplectic curves of every degree, when there are singularities. In fact, Moishezon~\cite{moishezoncusps} has constructed examples of symplectic curves of degree at least eighteen in $\mathbb{CP}^2$ with nodal and cuspidal singularities, but which are not symplectically isotopic to a projective algebraic curve.

Finally, we establish a complete classification of genus one Lefschetz fibrations over the two-sphere $\mathbb{S}^2$. The following statement has already been established by Moishezon-Livne~\cite{moishezongenusonelefschetzfibrations} but we give a new proof in this paper using our theory on the braid group $B_3$.

\begin{theorem}
\label{tmainLefschetz}
The number of singular fibers is a complete invariant of the isomorphism class of a genus one Lefschetz fibration over $\mathbb{S}^2$. Furthermore, the number of singular fibers in a genus one Lefschetz fibration over $\mathbb{S}^2$ is divisible by $12$.
\end{theorem}

\subsection{Outline of the paper}
The outline of the paper is as follows. In Section~\ref{sbackgroundmotivation}, we review background on the theory of Hurwitz curves in $\mathbb{CP}^2$, braid monodromy factorizations, the Garside normal form in the braid groups, and positive half-twists in the braid groups. We also set up terminology and notation that we will use in the rest of the paper. 

In Section~\ref{sB3background}, we develop our theory on the algebraic structure of the braid group $B_3$. In Subsection~\ref{ssdualitypositivebraid}, we introduce our theory of duality of positive braids, which is a tool to study factorizations in the braid group $B_3$. In Subsection~\ref{sspowerpositivehalftwist}, we algebraically characterize the Garside normal forms of powers of positive half-twists in the braid group $B_3$.

In Section~\ref{smain}, we use the theory developed in Section~\ref{sB3background} in order to classify factorizations of $\Delta^2$ into powers of positive half-twists in the braid group $B_3$ up to Hurwitz equivalence. In Subsection~\ref{ssclassificationfactorizationspositivehalftwists}, we show that there is a unique Hurwitz equivalence class of factorizations of $\Delta^2$ into positive half-twists in the braid group $B_3$. In Subsection~\ref{ssclassificationfactorizationspowerspositivehalftwists}, we show that the numbers of factors of each type are a complete set of invariants of the Hurwitz equivalence class of a factorization of $\Delta^2$ into powers of positive half-twists in the braid group $B_3$. In Subsection~\ref{sssingularityconstraints}, we constrain the types of factors in a factorization of $\Delta^2$ into powers of positive half-twists in the braid group $B_3$. The combination of these statements in Section~\ref{smain} is equivalent to Theorem~\ref{tmainHurwitzcurve} and Theorem~\ref{tmainsingularities} via braid monodromy.

In Section~\ref{sLefschetzfibrations}, we show that there is a unique Hurwitz equivalence class of factorizations of the identity into positive Dehn twists in the mapping class group $\text{Mod}\left(\mathbb{T}^2\right) = \text{SL}_2\left(\mathbb{Z}\right)$, where $\mathbb{T}^2$ is the two-dimensional torus. The statement is equivalent to Theorem~\ref{tmainLefschetz}.

\subsection{Acknowledgements}
The author would like to express his gratitude to Peter Ozsv{\'a}th and Zolt{\'a}n Szab{\'o} for their interest in this work and helpful feedback on this manuscript. The author would also like to express his gratitude to Denis Auroux for helpful and informative email correspondence.

\section{Background and motivation}
\label{sbackgroundmotivation}
In this section, we review background on the theory of Hurwitz curves in $\mathbb{CP}^2$ and braid monodromy (Subsection~\ref{ssHurwitzcurvesbraidmonodromy}), the theory of factorizations in a group (Subsection~\ref{sstheoryfactorizations}), the Garside normal form in the braid groups (Subsection~\ref{ssGarsidebackground}), and positive half-twists in the braid groups (Subsection~\ref{ssdpositivehalftwist}). We will also set up notation and terminology that we will use in this paper. We refer the reader to Auroux~\cite{aurouxopenquestions} for a more detailed expository overview of the theory of Hurwitz curves and braid monodromy factorizations, and its connections to symplectic geometry. The material in this section is not new, and our contributions begin in Section~\ref{sB3background}.

\subsection{The theory of Hurwitz curves in $\mathbb{CP}^2$ and braid monodromy}
\label{ssHurwitzcurvesbraidmonodromy}

Let $\mathbb{CP}^n$ denote complex projective $n$-space. Let $\pi:\mathbb{CP}^2\setminus \left[0:0:1\right]\to \mathbb{CP}^1$ be the projection defined in homogeneous coordinates by the rule $\pi\left[z_0:z_1:z_2\right] = \left[z_0:z_1\right]$. We recall the definition of one of the main objects of study in this paper.

\begin{definition}
A (singular) closed oriented $2$-dimensional real submanifold $C\subseteq \mathbb{CP}^2$ with isolated singularities is a \textit{Hurwitz curve} if the following conditions on $C$ are satisfied:
\begin{description}
\item[(i)] The point $\left[0:0:1\right]\not\in C$.
\item[(ii)] The intersections of $C$ with the fibers of $\pi$ are transverse and positive, except at a finite set of points, which are precisely the singularities of $C$ and the nondegenerate tangencies of $C$ to the fibers of $\pi$. We refer to this finite set of exceptional points as the set of \textit{singular intersections of $C$}.
\end{description}
If $L\subseteq \mathbb{CP}^2$ is a complex projective line, then the homology class $\left[L\right]\in H_2\left(\mathbb{CP}^2\right)$ is a generator of the second homology group. The \textit{degree of the Hurwitz curve $C$} is the algebraic intersection number $d = \left[C\right]\cdot \left[L\right] > 0$. (The positivity of the algebraic intersection number follows from condition~\textbf{(ii)}.) We can write $\left[C\right] = d\left[L\right]$. 
\end{definition}

We recall the definition of an important class of Hurwitz curves.

\begin{definition}
A \textit{projective algebraic curve} $C\subseteq \mathbb{CP}^2$ is the zero set of a homogeneous polynomial $P\left(z_0,z_1,z_2\right)\in \mathbb{C}\left[z_0,z_1,z_2\right]$. The \textit{degree of the projective algebraic curve $C$} is the degree of the polynomial $P$.
\end{definition}

A projective algebraic curve is a Hurwitz curve with respect to a generic choice of pole $p\in \mathbb{CP}^2$ and linear projection $\mathbb{CP}^2\setminus \{p\}\to \mathbb{CP}^1$. In fact, the Hurwitz curves in $\mathbb{CP}^2$ behave like projective algebraic curves with respect to the projection $\pi:\mathbb{CP}^2\setminus \left[0:0:1\right]\to \mathbb{CP}^1$. In this paper, we will restrict our attention to a special class of Hurwitz curves, where the singularities are sufficiently simple. We recall the definition of these special kinds of singularities.

\begin{definition}
An \textit{$A_n$-singularity of a Hurwitz curve $C\subseteq \mathbb{CP}^2$} is a singularity of $C$ locally modelled by the equation $z^2 = w^n$ in local complex coordinates $\left(z,w\right)$ near $\left(0,0\right)$, where the level sets of the $z$-coordinate are transverse to the fibers of $\pi$. For example, an $A_1$-singularity of $C$ is a nondegenerate tangency of $C$ to a fiber of $\pi$, an $A_2$-singularity of $C$ is a \textit{nodal singularity} (locally modelled on the transverse intersection of two complex lines), an $A_3$-singularity of $C$ is a \textit{cuspidal singularity}, and an $A_4$-singularity of $C$ is a \textit{tacnodal singularity}.

If $C$ has an $A_2$-singularity and no $A_n$-singularities for $n\not\in \{1,2\}$, then $C$ is a \textit{nodal curve}. If $C$ has an $A_3$-singularity and no $A_n$-singularities for $n\not\in \{1,3\}$, then $C$ is a \textit{cuspidal curve}.
\end{definition}

Of course, an $A_1$-singularity is not a singularity of the Hurwitz curve in the usual sense, but we adopt this terminology as a matter of convenience. We now define the restricted class of Hurwitz curves that we will study in this paper.

\begin{definition}
\label{dsimpleHurwitzcurve}
A Hurwitz curve $C\subseteq \mathbb{CP}^2$ is a \textit{simple Hurwitz curve} if the following conditions are satisfied:
\begin{description}
\item[(i)] The singularities of $C$ are all $A_n$-singularities.
\item[(ii)] The singular intersections of $C$ have distinct images under the projection $\pi:\mathbb{CP}^2\setminus \left[0:0:1\right]\to \mathbb{CP}^1$. 
\end{description}
\end{definition}

We remark that condition~\textbf{(ii)} in Definition~\ref{dsimpleHurwitzcurve} is a matter of convenience rather than a constraint. Indeed, a Hurwitz curve $C\subseteq \mathbb{CP}^2$ can always be isotoped (through Hurwitz curves with the same singularities as $C$, in the sense of the following Definition~\ref{disotopyHurwitz}) so that condition~\textbf{(ii)} in Definition~\ref{dsimpleHurwitzcurve} is satisfied.

An important question is to characterize the extent to which the topology of the embedding of a simple Hurwitz curve $C\subseteq \mathbb{CP}^2$ differs from the topology of the embedding of a projective algebraic curve $C\subseteq \mathbb{CP}^2$. Of course, when we refer to ``the topology of an embedding", we are referring to the isotopy class of $C\subseteq \mathbb{CP}^2$. We recall the definition of an isotopy in the present context.

\begin{definition}
\label{disotopyHurwitz}
An \textit{isotopy of Hurwitz curves} between the Hurwitz curves $C_0,C_1\subseteq \mathbb{CP}^2$ is a one-parameter family $\{C_t\}_{t\in \left[0,1\right]}$, with the property that $C_t$ is a Hurwitz curve with the same singularities as $C_0$ (and $C_1$) for each $t\in \left[0,1\right]$. The \textit{isotopy problem for Hurwitz curves} in $\mathbb{CP}^2$ is the problem of classifying Hurwitz curves in $\mathbb{CP}^2$ up to isotopy.
\end{definition}

In this paper, we will establish a complete resolution of the isotopy problem for degree three simple Hurwitz curves in $\mathbb{CP}^2$. The results we establish for Hurwitz curves in $\mathbb{CP}^2$ are also applicable to symplectic curves in $\mathbb{CP}^2$. We recall the basic definitions from the symplectic geometry of $\mathbb{CP}^2$.

\begin{definition}
Let $\left(\mathbb{CP}^2,\omega\right)$ denote the complex projective plane with the Fubini-Study symplectic form $\omega$. A \textit{symplectic curve} in $\mathbb{CP}^2$ is a closed oriented two-dimensional real submanifold $C\subseteq \mathbb{CP}^2$ such that the restriction of $\omega$ to $C$ is the symplectic form on $C$ specified by its orientation. The \textit{degree of a symplectic curve $C$} is $d = \left[C\right]\cdot \left[L\right] > 0$, where $\left[L\right]\in H_2\left(\mathbb{CP}^2\right)$ is the homology class of a complex projective line $L\subseteq \mathbb{CP}^2$. (The positivity of the algebraic intersection number follows from the condition that $C$ is a symplectic curve.) 

A \textit{symplectic isotopy} between the symplectic curves $C_0,C_1\subseteq \mathbb{CP}^2$ is a one-parameter family $\{C_t\}_{t\in \left[0,1\right]}$, with the property that $C_t$ is a symplectic curve for each $t\in \left[0,1\right]$. The \textit{isotopy problem for symplectic curves} is the problem of classifying symplectic curves in $\mathbb{CP}^2$ up to isotopy.  
\end{definition}

If $\Sigma$ is an oriented surface and if $C$ can be written as the image of $\Sigma$ under a smooth map $f:\Sigma\to \mathbb{CP}^2$, then the condition that $C$ is a symplectic curve can be expressed as an inequality in terms of the derivative $df$ of $f$. If $C$ has degree three, then Theorem~\ref{tmainsymplectic} in the Introduction implies that there is an isotopy of smooth maps $\{f_t\}_{t\in \left[0,1\right]}$, where $f_t:\Sigma\to \mathbb{CP}^2$ satisfies the same inequality in terms of the derivative $df_t$, $f_0 = f$, and the image $f_1\left(\Sigma\right)$ satisfies a homogeneous polynomial equation.

In this paper, we will study isotopy classes of simple Hurwitz curves via the theory of braid monodromy, which we now review. Let $C\subseteq \mathbb{CP}^2$ be a simple Hurwitz curve of degree $d$ and let $q_1,q_2,\dots,q_k$ be the singular intersections of $C$, with $p_j = \pi\left(q_j\right)$ for $1\leq j\leq k$. The restriction of the projection $\pi:\mathbb{CP}^2\setminus \left[0:0:1\right]\to \mathbb{CP}^1$ to $C\subseteq \mathbb{CP}^2\setminus \left[0:0:1\right]$ has degree $d$ since the intersections $C\cap \pi^{-1}\left(p\right)$ of $C$ with the fibers of $\pi$ have cardinality $d$, except if $p\in \{p_1,p_2,\dots,p_k\}$. We will define the monodromy of the simple Hurwitz curve $C\subseteq \mathbb{CP}^2$ with respect to a simple closed curve in the base $\mathbb{CP}^1\setminus \{p_1,p_2,\dots,p_k\}$. Firstly, we recall the following interpretation of the braid group.

\begin{definition}
If $d\geq 1$ is a positive integer, then the \textit{unordered configuration space of $d$ points in the complex plane} $\mathbb{C}$ is \[\text{Conf}_d\left(\mathbb{C}\right) = \{\{z_1,z_2,\dots,z_d\}:z_i\in \mathbb{C}\text{ and }z_i\neq z_j\text{ for }i\neq j\}.\] The \textit{braid group $B_d$ on $d$ strands} is the fundamental group $B_d = \pi_1\left(\text{Conf}_d\left(\mathbb{C}\right)\right)$ of the unordered configuration space of $d$ points.
\end{definition}

We now define the monodromy of a simple Hurwitz curve $C\subseteq \mathbb{CP}^2$ with respect to a simple loop in the base $\mathbb{CP}^1$ that does not pass through the projections of the singular intersections of $C$.

\begin{definition}
\label{dmonodromyloop}
Let $C\subseteq \mathbb{CP}^2$ be a simple Hurwitz curve and let $q_1,q_2,\dots,q_k$ be the singular intersections of $C$, with $p_j=\pi\left(q_j\right)$ for $1\leq j\leq k$. Let us assume without loss of generality that $\left[0:1\right]\not\in \{p_1,p_2,\dots,p_k\}$. Let $\left[\alpha\right]\in \pi_1\left(\mathbb{CP}^1\setminus \{\left[0:1\right],p_1,p_2,\dots,p_k\}\right)$ be the homotopy class of a closed curve $\alpha$ in $\mathbb{CP}^1\setminus \{\left[0:1\right],p_1,p_2,\dots,p_k\}$. 

Let us fix a trivialization of the fiber bundle $\pi:\mathbb{CP}^2\setminus \left[0:0:1\right]\to \mathbb{CP}^1$ over the contractible open subset $\mathbb{CP}^1\setminus \left[0:1\right]\subseteq \mathbb{CP}^1$. The closed curve $\alpha$ lifts to a closed curve in $\text{Conf}_d\left(\mathbb{C}\right)$, by using the trivialization of the fiber bundle in order to identify the fibers of $\pi$ over $\mathbb{CP}^1\setminus \left[0:1\right]$ with $\mathbb{C}$. The \textit{monodromy of $C$ with respect to $\left[\alpha\right]$} is the corresponding element of $B_d\cong \pi_1\left(\text{Conf}_d\left(\mathbb{C}\right)\right)$.
\end{definition}

The monodromy of a degree $d$ simple Hurwitz curve $C\subseteq \mathbb{CP}^2$ can be interpreted as a map $\pi_1\left(\mathbb{CP}^1\setminus \{\left[0:1\right],p_1,p_2,\dots,p_k\}\right)\to B_d$, in the context of Definition~\ref{dmonodromyloop}, which is well-defined up to global conjugation in the braid group $B_d$ (corresponding to the choice of a trivialization of the fiber bundle over $\mathbb{CP}^1\setminus \left[0:1\right]$). We can encapsulate this data more concretely by defining the braid monodromy factorization of a simple Hurwitz curve. Firstly, we recall the definition of a factorization in a group and set up notation. 

\begin{definition}
Let $G$ be a group and let $g\in G$. A \textit{factorization} of $g$ is a $k$-tuple $\left(g_1,g_2,\dots,g_k\right)$ such that $g = \prod_{j=1}^{k} g_j$, where $g_j\in G$ for $1\leq j\leq k$. We refer to each $g_j$ as a \textit{factor} in the factorization and we refer to $k$ as the \textit{number of factors} in the factorization. The \textit{type of a factor} $g_j$ is the conjugacy class of $g_j$ in $G$. We use the notation $g\equiv \left(g_1,g_2,\dots,g_k\right)$ to denote that $\left(g_1,g_2,\dots,g_k\right)$ is a factorization of $g$.
\end{definition}

We now define the braid monodromy factorization of a simple Hurwitz curve in $\mathbb{CP}^2$. 

\begin{definition}
\label{dbraidmonodromy}
Let $C\subseteq \mathbb{CP}^2$ be a simple Hurwitz curve of degree $d$ and let $q_1,q_2,\dots,q_k$ be the singular intersections of $C$, with $p_j = \pi\left(q_j\right)$ for $1\leq j\leq k$. Let $\left[\gamma\right]\in \pi_1\left(\mathbb{CP}^1\setminus \{\left[0:1\right],p_1,p_2,\dots,p_k\}\right)$ be the homotopy class of the simple closed curve $\gamma$ in $\mathbb{CP}^1\setminus \{\left[0:1\right],p_1,p_2,\dots,p_k\}$ such that $\left[\gamma\right]$ is trivial in $\pi_1\left(\mathbb{CP}^1\setminus \{p_1,p_2,\dots,p_k\}\right)$. Let $\gamma_1,\gamma_2,\dots,\gamma_k$ be closed curves in $\mathbb{CP}^1\setminus \{\left[0:1\right],p_1,p_2,\dots,p_k\}$ such that $\{\left[\gamma_1\right],\left[\gamma_2\right],\dots,\left[\gamma_k\right]\}$ is a generating set of the free group $\pi_1\left(\mathbb{CP}^1\setminus \{\left[0:1\right],p_1,p_2,\dots,p_k\}\right)$ of rank $k$, and $\left[\gamma\right] = \left[\gamma_1\right]\left[\gamma_2\right]\cdots\left[\gamma_k\right]$. 

Let us fix a trivialization of the fiber bundle $\pi:\mathbb{CP}^2\setminus \left[0:0:1\right]\to \mathbb{CP}^1$ over the contractible open subset $\mathbb{CP}^1\setminus \left[0:1\right]\subseteq \mathbb{CP}^1$. Let $g\in B_d$ be the monodromy of $C$ with respect to the loop $\gamma$ and let $g_j$ be the monodromy of $C$ with respect to the loop $\gamma_j$. The \textit{braid monodromy factorization of the Hurwitz curve $C$ with respect to the generating set $\{\left[\gamma_1\right],\left[\gamma_2\right],\dots,\left[\gamma_k\right]\}$} is the factorization $g\equiv \left(g_1,g_2,\dots g_k\right)$. 
\end{definition}

In fact, there is a distinguished element $\Delta\in B_d$ (the Garside element) such that $\Delta^2$ generates the center of $B_d$. An important observation is that the monodromy of a simple Hurwitz curve $C\subseteq \mathbb{CP}^2$ with respect to $\left[\gamma\right]$ is equal to $\Delta^2$, so that we have $g = \Delta^2$ in Definition~\ref{dbraidmonodromy}. In particular, the braid monodromy factorization of $C$ with respect to a generating set $\{\left[\gamma_1\right],\left[\gamma_2\right],\dots,\left[\gamma_k\right]\}$ of $\pi_1\left(\mathbb{CP}^1\setminus \{\left[0:1\right],p_1,p_2,\dots,p_k\}\right)$ is a factorization $\Delta^2\equiv \left(g_1,g_2,\dots,g_k\right)$ into $k$ elements of $B_d$, where $g_j$ is the monodromy of $C$ with respect to $\left[\gamma_j\right]$.

Of course, the braid monodromy of $C$ with respect to another generating set $\{\left[\gamma_1'\right],\left[\gamma_2'\right],\dots,\left[\gamma_k'\right]\}$ of $\pi_1\left(\mathbb{CP}^1\setminus \{\left[0:1\right],p_1,p_2,\dots,p_k\}\right)$ is another factorization $\Delta^2 = g_1'g_2'\cdots g_k'$ into $k$ elements of $B_d$, where $g_j'$ is the monodromy of $C$ with respect to $\left[\gamma_j'\right]$. We can understand the relationship between braid monodromy factorizations of $C$ with respect to different generating sets via the notion of Hurwitz equivalence, which we now define.

\begin{definition}
Let $G$ be a group. A \textit{Hurwitz move} on a factorization $\left(g_1,g_2,\dots,g_k\right)$ is one of the following modifications of the $j$th and $\left(j+1\right)$st factors for some $1\leq j\leq k-1$: \begin{align*} \left(g_1,\dots,g_j,g_{j+1},\dots,g_k\right) &\sim \left(g_1,\dots,g_jg_{j+1}g_j^{-1},g_j,\dots,g_k\right) \\ \left(g_1,\dots,g_j,g_{j+1},\dots,g_k\right) &\sim \left(g_1,\dots,g_{j+1},g_{j+1}^{-1}g_jg_{j+1},\dots,g_k\right).\end{align*} If $g\in G$ is not central, then a pair of factorizations ${\cal F}$ and ${\cal F}'$ of $g$ are \textit{Hurwitz equivalent} if ${\cal F}'$ can be obtained from ${\cal F}$ by a finite sequence of Hurwitz moves.

The \textit{global conjugation move} by $h\in G$ on a factorization $\left(g_1,g_2,\dots,g_k\right)$ of a central element $g\in G$ is the factorization \[h\left(g_1,g_2,\dots,g_k\right)h^{-1} = \left(hg_1h^{-1},hg_2h^{-1},\dots,hg_kh^{-1}\right)\] of $g$. If $g\in G$ is central, then a pair of factorizations ${\cal F}$ and ${\cal F}'$ of $g$ are \textit{Hurwitz equivalent} if ${\cal F}'$ can be obtained from ${\cal F}$ by a finite sequence of Hurwitz moves and global conjugation moves.
\end{definition}

We observe that the relation of Hurwitz equivalence defines an equivalence relation on factorizations. Indeed, reflexivity and transitivity are evident, and symmetry is a consequence of the fact that the inverse of a Hurwitz move or global conjugation move is also a Hurwitz move or global conjugation move, respectively.

The braid monodromy factorization of a simple Hurwitz curve $C\subseteq \mathbb{CP}^2$ is independent of the choice of trivialization of the fiber bundle $\pi$ over $\mathbb{CP}^1\setminus \left[0:1\right]$ and generating set of $\pi_1\left(\mathbb{CP}^1\setminus \{\left[0:1\right],p_1,p_2,\dots,p_k\}\right)$, up to Hurwitz equivalence. Indeed, the braid group $B_k$ acts as mapping classes of $\mathbb{CP}^1\setminus \{\left[0:1\right],p_1,p_2,\dots,p_k\}$, and this induces a transitive action of $B_k$ on generating sets of $\pi_1\left(\mathbb{CP}^1\setminus \{\left[0:1\right],p_1,p_2,\dots,p_k\}\right)$. The Hurwitz moves applied to braid monodromy factorizations correspond to the actions of a generating set of $B_k$ on $\pi_1\left(\mathbb{CP}^1\setminus \{\left[0:1\right],p_1,p_2,\dots,p_k\}\right)$, and the global conjugation moves applied to factorizations correspond to different choices of a trivialization of the fiber bundle $\pi$ over $\mathbb{CP}^1\setminus \left[0:1\right]$.

\begin{definition}
Let $C\subseteq \mathbb{CP}^2$ be a simple Hurwitz curve of degree $d$ and let $q_1,q_2,\dots,q_k$ be the singular intersections of $C$, with $p_j=\pi\left(q_j\right)$ for $1\leq j\leq k$. The \textit{braid monodromy of $C$} is the Hurwitz equivalence class of the braid monodromy factorization of $C$ with respect to any generating set of $\pi_1\left(\mathbb{CP}^1\setminus \{\left[0:1\right],p_1,p_2,\dots,p_k\}\right)$. 
\end{definition}

Finally, we recall the following characterization of the factors in the braid monodromy of a degree $d$ simple Hurwitz curve $C\subseteq \mathbb{CP}^2$. Let $\gamma_j$ be a small simple closed curve encircling the point $p_j$ in Definition~\ref{dbraidmonodromy}, where $q_j$ is an $A_n$-singularity of $C$ for some positive integer $n\geq 1$. The monodromy $g_j$ of $C$ with respect to $\left[\gamma_j\right]$ is the $n$th power of a positive half-twist in the braid group $B_d$. In particular, the braid monodromy of $C$ is the Hurwitz equivalence class of a factorization of $\Delta^2$ into powers of positive half-twists in the braid group $B_d$.  

The motivation for studying the braid monodromy of a simple Hurwitz curve $C\subseteq \mathbb{CP}^2$ arises from the following result due to Kharlamov-Kulikov~\cite{kulikovkharlamovbraidmonodromy}. 

\begin{theorem}[Kharlamov-Kulikov]
\label{tKK}
If $C,C'\subseteq \mathbb{CP}^2$ are simple Hurwitz curves, then $C$ is isotopic to $C'$ if and only if the braid monodromy of $C$ is equal to the braid monodromy of $C'$.
\end{theorem}

(We remark that Kulikov-Teicher~\cite{kulikovteicherbraidmonodromy} proved that the braid monodromy is a complete invariant of the isotopy class of a simple Hurwitz curve with only $A_n$-singularities for $n\leq 3$ (i.e., nondegenerate tangencies, nodes and cusps), and for the most part in this paper, we only need this special case.)

Theorem~\ref{tKK} implies that the isotopy problem for degree $d$ simple Hurwitz curves in $\mathbb{CP}^2$ is equivalent to the problem of classifying factorizations of $\Delta^2$ into powers of positive half-twists in the braid group $B_d$ up to Hurwitz equivalence. In this paper, we will use this connection in order to establish a complete classification of isotopy classes of degree three simple Hurwitz curves in $\mathbb{CP}^2$ (Theorem~\ref{tmainHurwitzcurve} in the Introduction). We turn to a generalization of these results to all smooth curves (of any degree) in $\mathbb{CP}^2$ in the sequel~\cite{dattageneralclassificationHurwitzcurves}.

\subsection{The theory of factorizations and Hurwitz equivalence in a group}
\label{sstheoryfactorizations}

In this subsection, we review some elementary statements concerning the theory of factorizations and Hurwitz equivalence in a group that we will use in this paper.

\begin{proposition}
\label{pHurwitzequivalencefunctorial}
Let $G$ be a group and let ${\cal F} = \left(g_1,g_2,\dots,g_k\right)$ and ${\cal F}' = \left(g_1',g_2',\dots,g_k'\right)$ be factorizations in $G$. If ${\cal F}$ and ${\cal F}'$ are Hurwitz equivalent factorizations in $G$, and if $C\subseteq G$ is a conjugacy class of $G$, then the number of factors of type $C$ in ${\cal F}$ is equal to the number of factors of type $C$ in ${\cal F}'$. 

If $H$ is a group and if $\phi:G\to H$ is a group homomorphism, then define $\phi\left({\cal F}\right) = \left(\phi\left(g_1\right),\phi\left(g_2\right),\dots,\phi\left(g_k\right)\right)$ and $\phi\left({\cal F}'\right) = \left(\phi\left(g_1'\right),\phi\left(g_2'\right),\dots,\phi\left(g_k'\right)\right)$. If ${\cal F}$ and ${\cal F}'$ are Hurwitz equivalent factorizations in $G$, then $\phi\left({\cal F}\right)$ and $\phi\left({\cal F}'\right)$ are Hurwitz equivalent factorizations in $H$. (Hurwitz equivalence is functorial.)
\end{proposition}
\begin{proof}
The first statement is a consequence of the fact that Hurwitz moves and global conjugation moves do not change the number of factors of type $C$ if $C\subseteq G$ is a conjugacy class of $G$. The second statement is a consequence of the fact that group homomorphisms commute with Hurwitz moves.
\end{proof}

We remark that a global conjugation move applied to a factorization ${\cal F}$ may result in a factorization that is not equivalent to ${\cal F}$ by a finite sequence of Hurwitz moves alone. For example, the factorization $\left(g,g^{-1}\right)$ of the identity element $e\in G$ (which is central in $G$) into two factors is only Hurwitz equivalent to itself and the factorization $\left(g^{-1},g\right)$ of $e\in G$. However, there are other factorizations of $e$ that can be obtained from these ones by global conjugation if $\{g,g^{-1}\}$ is not a conjugacy class.

We record the following useful statement.

\begin{proposition}
\label{pglobalconjugationcommutesHurwitz}
A global conjugation move commutes with a Hurwitz move. In particular, a pair of factorizations ${\cal F}$ and ${\cal F}'$ of a central element are Hurwitz equivalent if and only ${\cal F}'$ can be obtained from $h^{-1}{\cal F}h$ by a finite sequence of Hurwitz moves for some $h\in G$. 
\end{proposition}
\begin{proof}
The first statement is a consequence of the fact that global conjugation is an (inner) automorphism of $G$. The second statement is a consequence of the first statement and the fact that the composition of global conjugation moves is itself a global conjugation move.
\end{proof}

If $G$ is a group, then the problem of classifying factorizations of elements of $G$ up to Hurwitz equivalence is known as the \textit{Hurwitz problem} in the group $G$. In general, the Hurwitz problem in a group is hard, or even undecidable. In fact, in the braid group $B_n$ for $n\geq 5$, it is an undecidable problem to classify factorizations of $\Delta^2$ up to Hurwitz equivalence, if there are no restrictions on the types of factors in the factorization~\cite{liberman2005Hurwitz}. However, in this paper, we are focussed on classifying factorizations of $\Delta^2$ in the braid group $B_3$ up to Hurwitz equivalence, where the factors are powers of positive half-twists. In Section~\ref{smain}, we will show that the numbers of factors of each type are a complete set of invariants of the Hurwitz equivalence class of a factorization of $\Delta^2$ into powers of positive half-twists in the braid group $B_3$, and this statement is equivalent to Theorem~\ref{tmainHurwitzcurve} in the Introduction. 

\subsection{The Garside normal form and the monoid $B_n^{+}$ of positive braids in the braid group $B_n$}
\label{ssGarsidebackground}
In this subsection, we review the Garside normal form and the monoid $B_n^{+}$ of positive braids in $B_n$. We will review notation, terminology and results in this section in the generality of the braid group $B_n$, although the rest of the paper will only concern the braid group $B_3$. We refer the reader to the classical reference~\cite{garside1969braid} for the original presentation of the ideas and proofs of the statements in this subsection. 

\begin{definition}
The \textit{Artin presentation} of the braid group $B_n$ is \[B_n = \left\langle  \sigma_1,\dots,\sigma_{n-1}:\sigma_i\sigma_{i+1}\sigma_i = \sigma_{i+1}\sigma_i\sigma_{i+1}, \sigma_i\sigma_j = \sigma_j\sigma_i\right\rangle,\] where $1\leq i\leq n-2$ in the first set of relations and $\left|i-j\right|>1$ in the second set of relations. If $1\leq i\leq n-1$, then the generator $\sigma_i$ is referred to as an \textit{Artin generator}. A \textit{positive braid} in $B_n$ is a braid in the submonoid of $B_n$ generated by $\sigma_1,\sigma_2,\dots,\sigma_{n-1}$. A \textit{product expansion} of a positive braid is a representation of the positive braid as a product of nonnegative powers of the Artin generators (we do not allow inverses of generators in a product expansion of a positive braid). The \textit{length} of a product expansion of a positive braid is the number of generators in the product expansion.
\end{definition}

The monoid of positive braids plays a fundamental role in the algebraic structure of $B_n$. Indeed, we recall the following result due to Garside~\cite{garside1969braid}.

\begin{theorem}[\cite{garside1969braid}]
\label{tmonoidembedding}
Let $B_n^{+}$ be the monoid with the Artin presentation \[B_n^{+} = \left\langle  \sigma_1,\dots,\sigma_{n-1}:\sigma_i\sigma_{i+1}\sigma_i = \sigma_{i+1}\sigma_i\sigma_{i+1}, \sigma_i\sigma_j = \sigma_j\sigma_i\right\rangle,\] where $1\leq i\leq n-2$ in the first set of relations and $\left|i-j\right|>1$ in the second set of relations. The natural map $B_n^{+}\to B_n$ is an embedding of monoids with image equal to the monoid of positive braids in $B_n$.
\end{theorem}

We will henceforth denote the monoid of positive braids in $B_n$ by $B_n^{+}$. Theorem~\ref{tmonoidembedding} states that if a pair of positive braids in $B_n$ are equal, then they are equal by a finite sequence of applications of the Artin relations in the monoid $B_n^{+}$. In particular, in such a finite sequence of applications of the Artin relations, we would not introduce inverses of generators at any stage. Theorem~\ref{tmonoidembedding} is the crucial ingredient in this paper for studying the combinatorics of product expansions in the monoid $B_3^{+}$ of positive braids in $B_3$. We discuss this point further at the beginning of Subsection~\ref{ssdualitypositivebraid}. In Subsection~\ref{ssdualitypositivebraid}, we use Theorem~\ref{tmonoidembedding} extensively in order to establish important technical results on the monoid $B_3^{+}$. 
 
A corollary of Theorem~\ref{tmonoidembedding} is that the length of a positive braid is well-defined, and independent of the choice of product expansion. Indeed, the number of generators match on both sides of each of the relations in the Artin presentation. In particular, the set of product expansions of a positive braid is a finite set, since all have the same length. 

We now recall Garside's fundamental theorem on the braid groups~\cite{garside1969braid} and we sketch its proof from~\cite{garside1969braid}.

\begin{theorem}[\cite{garside1969braid}]
\label{Garside}
The \textit{Garside element} of $B_n$ is \[\Delta = \left(\sigma_1\sigma_2\cdots \sigma_{n-1}\right)\left(\sigma_1\sigma_2\cdots \sigma_{n-2}\right)\cdots \left(\sigma_1\sigma_2\right)\sigma_1.\] If $g\in B_n$, then $g$ can be expressed uniquely as the product $\Delta^{k}\sigma$ where $k\in \mathbb{Z}$ and $\sigma\in B_n^{+}$ is a positive braid indivisible by $\Delta$ (i.e., $\sigma\neq \Delta\tau$ for any positive braid $\tau\in B_n^{+}$). We refer to this expression as the \textit{Garside normal form} of the braid $g$. We refer to $k$ as the \textit{Garside power} of the braid $g$.
\end{theorem}
\begin{proof}
Firstly, we establish the existence of the Garside normal form of $g$. The Garside element $\Delta\in B_n^{+}$ satisfies the fundamental equations $\Delta\sigma_i = \sigma_{n-i}\Delta$ for each $1\leq i\leq n-1$. Furthermore, $\Delta\sigma_i^{-1}\in B_n^{+}$ is a positive braid for each $1\leq i\leq n-1$. In particular, if $g\in B_n$, then there exists a positive integer $k\geq 0$ such that $\Delta^{k}g\in B_n^{+}$ is a positive braid. Therefore, we can write $g = \Delta^{k}\sigma$ for some integer $k\in \mathbb{Z}$ and some positive braid $\sigma\in B_n^{+}$. We can choose $\sigma\in B_n^{+}$ to be indivisible by $\Delta$ after possibly factoring out a power of $\Delta$ from $\sigma$. Indeed, the maximum power of $\Delta$ that divides $\sigma$ is at most the word length of $\sigma$ divided by $\frac{n\left(n-1\right)}{2}$, so this is always possible. 

Secondly, we establish the uniqueness of the Garside normal form of $g$. If $\Delta^{k}\sigma = g = \Delta^{k'}\sigma'$ for some $k,k'\in \mathbb{Z}$ and $\sigma,\sigma'\in B_n^{+}$ such that $\sigma,\sigma'$ are indivisible by $\Delta$, then we can either write $\Delta^{k-k'}\sigma = \sigma'$ if $k\geq k'$ or write $\Delta^{k'-k}\sigma' = \sigma$ if $k'\geq k$. In either case, we must have $k=k'$ since $\sigma,\sigma'$ are indivisible by $\Delta$. Of course, this also implies that $\sigma = \sigma'$. Therefore, the Garside normal form of $g$ is unique.
\end{proof}

An important fact is that the center of the braid group $B_n$ is the infinite cyclic group $\left\langle \Delta^2\right\rangle$, which adds further strength to the statement of Theorem~\ref{Garside}. In fact, in the topological interpretation of $B_n$ as the mapping class group of the $n$-times marked disk relative to its boundary, the square of the Garside element is a Dehn twist with respect to the boundary. In particular, conjugation by $\Delta^2$ is the identity. We can also establish this algebraically by recalling the identity $\Delta^{-1}\sigma_i\Delta = \sigma_{n-i}$ for each $1\leq i\leq n-1$. We introduce the following notation to describe conjugation by powers of the Garside element $\Delta$. We will use this notation extensively in this paper.

\begin{definition}
\label{dconjugationbyDelta}
Let $\sigma\in B_n^{+}$ be a positive braid. We write $\sigma_{\left(l\right)} = \Delta^{-l}\sigma\Delta^{l}$ for $l\in \mathbb{Z}$. We note that $\sigma_{\left(l\right)}$ depends only on the parity of $l$ since $\Delta^2$ is central in $B_n$. In particular, $\sigma_{\left(l\right)} = \sigma$ if $l$ is even and $\sigma_{\left(l\right)} = \sigma_{\left(1\right)}$ if $l$ is odd.
\end{definition}

We conclude this subsection by observing that \[B_3 = \left\langle \sigma_1,\sigma_2 : \sigma_1\sigma_2\sigma_1 = \sigma_2\sigma_1\sigma_2 \right\rangle\] is the Artin presentation of $B_3$, and $\Delta = \sigma_1\sigma_2\sigma_1 = \sigma_2\sigma_1\sigma_2\in B_3^{+}$ is the Garside element of $B_3$. We will work with this presentation of $B_3$ for the remainder of this paper.

\subsection{Positive half-twists in the braid group $B_n$}
\label{ssdpositivehalftwist}
In this subsection, we briefly review the well-known statement that if $e\geq 1$, then the set of $e$th powers of positive half-twists in the braid group $B_n$ constitutes a conjugacy class in $B_n$. We refer the reader to~\cite{farbmargalitprimer} for a more detailed discussion.

\begin{definition}
Let us denote the $n$-times marked disk by $\mathbb{D}_n$, and recall the interpretation of the braid group $B_n$ as the mapping class group of $\mathbb{D}_n$. A \textit{simple arc in the $n$-times marked disk} is an embedding $\gamma:\left[0,1\right]\to \mathbb{D}_n$ such that $\gamma\left(0\right)$ and $\gamma\left(1\right)$ are markings. (In the present context, we do not consider arcs with endpoints on the boundary.) Let $H_{\gamma}$ denote the positive half-twist with respect to $\gamma$. 

Let us fix an isomorphism between $B_n$ and the mapping class group of $\mathbb{D}_n$. A \textit{standard positive half-twist} is $\sigma_i$ for some $1\leq i\leq n-1$. A \textit{non-standard positive half-twist} is a positive half-twist that is not a standard positive half-twist. 
\end{definition}

We recall the following classical and important observation and a sketch of its proof.

\begin{lemma}
\label{lbraidgrouptransitiveaction}
The mapping class group of the $n$-times marked disk acts transitively on simple arcs in the $n$-times marked disk.
\end{lemma}
\begin{proof}
We recall the classification of oriented surfaces which states that a pair of oriented surfaces with the same number of marked points, the same number of boundary components, and the same Euler characteristic are homeomorphic via an orientation-preserving homeomorphism. We will show that the statement follows from the classification of oriented surfaces. Indeed, if $\gamma$ and $\gamma'$ are simple arcs in the $n$-times marked disk $\mathbb{D}_n$, then we can cut $\mathbb{D}_n$ along $\gamma$ and $\gamma'$ and obtain a pair of oriented surfaces $\Sigma_{\gamma}$ and $\Sigma_{\gamma'}$ with $n-2$ markings, two boundary components, and the same Euler characteristic. Indeed, one of the boundary components of each of these surfaces arises from the boundary of $\mathbb{D}_n$, and the other boundary component arises from the boundary of a tubular neighborhood of the relevant simple arc. We refer to the latter boundary component of $\Sigma_{\gamma}$ as $\partial \gamma$, and the latter boundary component of $\Sigma_{\gamma'}$ as $\partial \gamma'$. The surfaces $\Sigma_{\gamma}$ and $\Sigma_{\gamma'}$ have the same Euler characteristic, since $\gamma$ and $\gamma'$ are simple arcs. The classification of oriented surfaces implies that there is an orientation-preserving homeomorphism from $\Sigma_{\gamma}$ to $\Sigma_{\gamma'}$ that maps $\partial \gamma$ to $\partial \gamma'$. Therefore, there is a orientation-preserving homeomorphism of the $n$-times marked disk that maps $\gamma$ to $\gamma'$.
\end{proof}

We record the following consequence of Lemma~\ref{lbraidgrouptransitiveaction}, which is important for our purposes.

\begin{corollary}
\label{cpositivehalftwistconjugacy}
If $e\geq 1$, then the set of $e$th powers of positive half-twists in $B_n$ is the conjugacy class of the $e$th power $\sigma_1^{e}$ of the standard positive half-twist $\sigma_1$. 
\end{corollary}
\begin{proof}
Let $\gamma$ and $\gamma'$ be simple arcs in $\mathbb{D}_n$, and let $H_{\gamma}$ and $H_{\gamma'}$ be positive half-twists with respect to the simple arcs $\gamma$ and $\gamma'$ in $\mathbb{D}_n$, respectively. If $\phi$ is a mapping class of $\mathbb{D}_n$ such that $\phi\left(\gamma\right) = \gamma'$, then we have $H_{\gamma'} = \phi H_{\gamma} \phi^{-1}$. The statement follows from Lemma~\ref{lbraidgrouptransitiveaction} and the fact that the image of a simple arc in $\mathbb{D}_n$ under a mapping class is a simple arc in $\mathbb{D}_n$. 
\end{proof}

\section{The algebraic structure of the braid group $B_3$}
\label{sB3background}
In this section, we develop our theory on the algebraic structure of the braid group $B_3$. In Section~\ref{smain}, we will apply the results in this section in order to establish a complete classification of factorizations of $\Delta^2$ into powers of positive half-twists in the braid group $B_3$, thus establishing the main results in the Introduction. 

The outline of this section is as follows. In Subsection~\ref{ssdualitypositivebraid}, we will develop our theory of duality of positive braids in $B_3$. In Subsection~\ref{sspowerpositivehalftwist}, we will explicitly determine the Garside normal form of the power of a positive half-twist in $B_3$. We will also establish technical statements concerning such Garside normal forms that we will frequently apply in Section~\ref{smain}. 

\subsection{The theory of duality of positive braids in $B_3$}
\label{ssdualitypositivebraid}
Let us describe our theory of duality and outline the role that it plays in the study of the algebraic structure of the braid group $B_3$. We are interested in equations in the braid group $B_3$, specifically factorizations of $\Delta^2$ into powers of positive half-twists in $B_3$. The idea is that the Garside normal form in $B_3$ reduces the study of equations in $B_3$ to the study of equations in the monoid $B_3^{+}$ of positive braids. The goal of the theory of duality is to develop a theory of equations in the monoid $B_3^{+}$ of positive braids.

The first step in the reduction from $B_3$ to $B_3^{+}$ is understanding the Garside normal form of a product in terms of the Garside normal forms of the individual factors. If $g = \Delta^{k}\sigma$ and $h = \Delta^{l}\tau$ are the Garside normal forms of $g$ and $h$, respectively, then $gh = \Delta^{k+l}\sigma_{\left(l\right)}\tau$ is not necessarily the Garside normal form of $gh$, and this observation is the origin of our theory of duality of positive braids (the notation $\sigma_{\left(l\right)}$ is from Definition~\ref{dconjugationbyDelta}). The expression $\Delta^{k+l}\sigma_{\left(l\right)}\tau$ is the Garside normal form of $gh$ if and only if $\sigma_{\left(l\right)}\tau$ is indivisible by $\Delta$. Furthermore, in order to determine the Garside normal form of $gh$, we must determine the maximal power of $\Delta$ that divides the product $\sigma_{\left(l\right)}\tau$. If $\Delta^{\omega}$ is the maximal power of $\Delta$ that divides the product $\sigma_{\left(l\right)}\tau$, then we will see that we can write $\sigma = \sigma'\sigma''$ and $\tau = \tau''\tau'$ for positive braids $\sigma',\sigma'',\tau',\tau''\in B_3^{+}$ such that $\sigma''_{\left(l\right)}\tau'' = \Delta^{\omega}$ (Proposition~\ref{pdualbraidproduct}). In this case, the Garside normal form of $gh$ is $\Delta^{k+l+\omega}\sigma'_{\left(l+\omega\right)}\tau'$ (Proposition~\ref{pGarsidenormalformproduct}). We describe the equation $\sigma''_{\left(l\right)}\tau'' = \Delta^{\omega}$ by writing that $\sigma''_{\left(l\right)}$ is \textit{dual} to $\tau''$ (Definition~\ref{ddual}), and the study of this equation is the \textit{theory of duality}.

Finally, Theorem~\ref{tmonoidembedding} implies that equations in the monoid $B_3^{+}$ of positive braids can be verified by inspection. Let $\sigma,\tau\in B_3^{+}$ be equal positive words in $B_{3}$. Theorem~\ref{tmonoidembedding} states that $\sigma$ and $\tau$ are related by a finite sequence of braid relations, where each application of a braid relation has the property that one side of the braid relation is embedded in the relevant word. On the other hand, a version of this statement is not true for arbitrary braids in $B_3$ that are not necessarily positive braids. For example, we cannot directly apply a braid relation to either side of the equation $\sigma_1^{-1}\sigma_2^{-1}\sigma_1\sigma_2 = \sigma_2\sigma_1^{-1}$, but it is nonetheless correct since $\sigma_1^{-1}\sigma_2^{-1}\sigma_1\sigma_2 = \sigma_1^{-1}\sigma_2^{-1}\Delta\Delta^{-1}\sigma_1\sigma_2 = \sigma_1^{-1}\Delta\sigma_1^{-1}\sigma_2\Delta^{-1}\sigma_2 = \sigma_2\sigma_1^{-1}$. We have introduced a product $\Delta\Delta^{-1}$, and then applied braid relations in order to establish the equation, but for equations in the monoid $B_3^{+}$ of positive braids, the content of Theorem~\ref{tmonoidembedding} is that this is not necessary.

Let us commence the study of positive braids and the theory of duality in the braid group $B_3$. We may view the Garside normal form of a positive braid as a factorization of the positive braid into a nonnegative power of $\Delta$ and a positive braid indivisible by $\Delta$. The first step is to establish a simple characterization of elements in the monoid $B_3^{+}$ that are indivisible by the Garside element $\Delta$. An \textit{isolated generator} in a product expansion of a positive braid is a generator $\sigma_{i}$ that is not immediately preceded or immediately succeeded in the product expansion by another $\sigma_i$.

\begin{proposition}
\label{piDB3}
If $\sigma\in B_3^{+}$ is a positive braid indivisible by $\Delta$, then $\sigma$ has a unique product expansion, and the product expansion does not contain isolated generators except possibly at its beginning or end.
\end{proposition}
\begin{proof}
If $\sigma\in B_3^{+}$ is a positive braid indivisible by $\Delta$, then consider a product expansion of $\sigma$. The product expansion does not contain either $\sigma_1\sigma_2\sigma_1$ or $\sigma_2\sigma_1\sigma_2$ since $\sigma$ is indivisible by $\Delta$, and $\Delta\sigma_i = \sigma_{i'}\Delta$ if $\{i,i'\} = \{1,2\}$. In particular, we cannot apply a braid relation to modify the product expansion of $\sigma$. Theorem~\ref{tmonoidembedding} now implies that $\sigma$ has a unique product expansion. Therefore, the statement is established.
\end{proof}

The element $\Delta\in B_3^{+}$ is not ``prime" in the usual sense, since it is possible for a product of a pair of positive braids to be divisible by $\Delta$ even if each of the two factors is indivisible by $\Delta$. A simple example is furnished by the product of $\sigma_1\sigma_2$ and $\sigma_1$. The failure of $\Delta$ to be ``prime" accounts for certain technical difficulties, since otherwise it would be straightforward to compute the Garside normal form of a product of braids in terms of the Garside normal forms of the individual factors. The following technical result allows us to precisely study the cases where a product of a pair of positive braids is divisible by $\Delta$ even if neither factor is divisible by $\Delta$. 

\begin{proposition}
\label{pdualbraidproduct}
Let $\rho,\tau\in B_3^{+}$ be positive braids indivisible by $\Delta$. The product $\rho\tau$ is divisible by $\Delta$ if and only if we can write $\rho = \rho'\rho''$ and $\tau = \tau''\tau'$ such that $\rho''\tau'' = \Delta$. Indeed, $\rho\tau$ is divisible by $\Delta$ if and only if either $\rho$ ends with $\sigma_i\sigma_{i'}$ and $\tau$ begins with $\sigma_i$ for some $\{i,i'\} = \{1,2\}$, or $\rho$ ends with $\sigma_i$ and $\tau$ begins with $\sigma_{i'}\sigma_i$ for some $\{i,i'\} = \{1,2\}$.
\end{proposition}
\begin{proof}
Let us assume that $\rho\tau$ is divisible by $\Delta$. Proposition~\ref{piDB3} implies that the index of the last generator in $\rho$ is different from the index of the first generator in $\tau$, since neither $\rho$ nor $\tau$ is divisible by $\Delta$. Furthermore, Proposition~\ref{piDB3} implies that at least one of the last generator in $\rho$ and the first generator in $\tau$ is isolated, again since neither $\rho$ nor $\tau$ is divisible by $\Delta$. We can choose $\rho''$ and $\tau''$ as follows. If the last generator in $\rho$ is isolated, then $\rho''$ is the product of the last two generators in $\rho$ and $\tau''$ is the first generator in $\tau$. If the first generator in $\tau$ is isolated, then $\rho''$ is the last generator in $\rho$ and $\tau''$ is the product of the first two generators in $\tau$. We have thus established the forward implication. The reverse implication is immediate. Therefore, the statement is established.
\end{proof}

If $\rho,\tau\in B_3^{+}$ are positive braids indivisible by $\Delta$, then consider the largest positive integer $k$ such that we can write $\rho = \rho'\rho''$ and $\tau = \tau''\tau'$ with $\rho''\tau'' = \Delta^{k}$. The consideration of such decompositions of $\rho$ and $\tau$ with this value of $k$ is useful for isolating the part of $\rho$ and the part of $\tau$ ``responsible" for the power of $\Delta$ in the Garside normal form of $\rho\tau$. Firstly, we introduce terminology to study the equation $\rho''\tau'' = \Delta^{k}$, which is of fundamental importance in this paper.

\begin{definition}
\label{ddual}
Let $\rho,\tau\in B_3^{+}$ be positive braids indivisible by $\Delta$. We write that $\rho$ is \textit{dual} to $\tau$ if $\rho\tau$ is a (nonnegative) power of $\Delta$. Alternatively, $\rho$ is dual to $\tau$ if and only if $\tau^{-1} = \Delta^{-\omega\left(\tau\right)}\rho$ is the Garside normal form of $\tau^{-1}$, where $\omega\left(\tau\right)$ is a nonnegative integer depending on $\tau$. If $\rho$ is dual to $\tau$, then we write that $\tau$ is \textit{right-dual} to $\rho$ (or $\tau$ is the right-dual of $\rho$) and $\rho$ is \textit{left-dual} to $\tau$ (or $\rho$ is the left-dual of $\tau$). 
\end{definition}

In general, if we write that $\rho$ is dual to $\tau$ for positive braids $\rho,\tau\in B_3^{+}$, then it is implicitly understood that $\rho$ and $\tau$ are indivisible by $\Delta$, but we might still mention this for emphasis if it is relevant to the context. We have the following uniqueness statement for duals. 

\begin{proposition}
\label{puniquenessofduals}
If $\rho,\rho',\tau\in B_3^{+}$ are positive braids indivisible by $\Delta$ such that $\rho$ is dual to $\tau$ and $\rho'$ is dual to $\tau$, then $\rho = \rho'$. In particular, the integer $\omega\left(\tau\right)$ in Definition~\ref{ddual} is uniquely specified by $\tau$.
\end{proposition}
\begin{proof}
The statement is equivalent to the uniqueness of the Garside normal form of $\tau^{-1}$ (a special case of Theorem~\ref{Garside}). Indeed, $\rho$ is dual to $\tau$ if and only if the Garside normal form of $\tau^{-1}$ is $\tau^{-1} = \Delta^{-\omega\left(\tau\right)}\rho$, since $\rho\in B_3^{+}$ is a positive braid indivisible by $\Delta$. In particular, if $\rho'$ is also dual to $\tau$, then $\rho = \rho'$. The integer $\omega\left(\tau\right)$ is also uniquely specified by $\tau$ for the same reason. Therefore, the statement is established.
\end{proof}

The relation of duality is \textit{not} symmetric, that is, left-duality is \textit{not} the same as right-duality. For example, $\sigma_1$ is dual to $\sigma_2\sigma_1$ but $\sigma_2\sigma_1$ is not dual to $\sigma_1$. The left-dual of $\sigma_1$ is $\sigma_1\sigma_2$ and the right-dual of $\sigma_1$ is $\sigma_2\sigma_1$. However, the left-dual and right-dual of a positive braid are related by conjugation by a suitable power of $\Delta$. We recall the notation in Definition~\ref{dconjugationbyDelta} in order to state the following \textit{pseudo-symmetry of duals}.

\begin{proposition}
\label{ppseudosymmetryduality}
Let $\rho,\tau\in B_3^{+}$ be positive braids indivisible by $\Delta$. We have that $\rho$ is dual to $\tau$ (with $\rho\tau = \Delta^{\omega\left(\tau\right)}$) if and only if $\tau$ is dual to $\rho_{\left(\omega\left(\tau\right)\right)}$ (with $\tau\rho_{\left(\omega\left(\tau\right)\right)}=\Delta^{\omega\left(\tau\right)}$). In particular, if $\rho$ is dual to $\tau$, then $\omega\left(\rho\right) = \omega\left(\tau\right)$. 
\end{proposition}
\begin{proof}
We observe that the second equation $\tau\rho_{\left(\omega\left(\tau\right)\right)} = \Delta^{\omega\left(\tau\right)}$ is a rearrangement of the first equation $\rho\tau = \Delta^{\omega\left(\tau\right)}$. Furthermore, $\rho_{\left(\omega\left(\tau\right)\right)}\in B_3^{+}$ is a positive braid indivisible by $\Delta$. Therefore, the statement is established.
\end{proof}

We recall the following standard terminology since it will be convenient in the subsequent discussion.

\begin{definition}
Let $\sigma\in B_3^{+}$ be a positive braid. A positive braid $\sigma'\in B_3^{+}$ is a \textit{left-divisor} of $\sigma$ if we can write $\sigma = \sigma'\sigma''$ for some positive braid $\sigma''\in B_3^{+}$. A positive braid $\sigma''\in B_3^{+}$ is a \textit{right-divisor} of $\sigma$ if we can write $\sigma = \sigma'\sigma''$ for some positive braid $\sigma'\in B_3^{+}$. A \textit{divisor} of $\sigma$ is a positive braid that is simultaneously a left-divisor and a right-divisor of $\sigma$. A \textit{proper} left-divisor/right-divisor/divisor of $\sigma$ is a left-divisor/right-divisor/divisor of $\sigma$, respectively, that is not equal to $\sigma$. 
\end{definition}

We observe that the Garside element $\Delta$ is a divisor of a positive braid $\sigma\in B_3^{+}$ if and only if $\Delta$ is either a left-divisor \textit{or} a right-divisor of $\sigma$. 

A natural question is whether the function $\omega$ defined on the set of positive braids in $B_3$ indivisible by $\Delta$ is additive. More precisely, if $\rho',\rho''\in B_3^{+}$ are positive braids indivisible by $\Delta$ such that $\rho'\rho''$ is also indivisible by $\Delta$, then is it true that $\omega\left(\rho'\rho''\right) = \omega\left(\rho'\right) + \omega\left(\rho''\right)$? In other words, is the ``product of the duals" equal to ``the dual of the product"? The following statement combines the notions of divisors and duality of positive braids.

\begin{proposition}
\label{ppdb}
Let $\rho=\rho'\rho''$ for positive braids $\rho,\rho',\rho''\in B_3^{+}$ indivisible by $\Delta$, where $\rho''$ is dual to $\tau''$ and $\rho'_{\left(\omega\left(\rho''\right)\right)}$ is dual to $\tau'$. If $\tau = \tau''\tau'$, then $\rho\tau = \Delta^{\omega\left(\rho'\right)+\omega\left(\rho''\right)}$. In particular, $\rho$ is dual to $\tau$ if and only if $\tau$ is indivisible by $\Delta$. Furthermore, we have $\omega\left(\rho\right)\leq \omega\left(\rho'\right)+\omega\left(\rho''\right)$ with equality if and only if $\tau$ is indivisible by $\Delta$. 
\end{proposition}
\begin{proof}
We have the equation \[\rho\tau =\rho'\rho''\tau''\tau' = \rho'\Delta^{\omega\left(\rho''\right)}\tau' = \Delta^{\omega\left(\rho''\right)} \rho'_{\left(\omega\left(\rho''\right)\right)}\tau' = \Delta^{\omega\left(\rho'\right)+\omega\left(\rho''\right)}.\] The uniqueness of duals (Proposition~\ref{puniquenessofduals}) implies that $\rho$ is dual to $\tau$ if and only if $\tau$ is indivisible by $\Delta$. In this case, $\omega\left(\rho\right) = \omega\left(\rho'\right) + \omega\left(\rho''\right)$. However, if $\tau$ is divisible by $\Delta$, then the equation $\rho\tau = \Delta^{\omega\left(\rho'\right)+\omega\left(\rho''\right)}$ implies that $\omega\left(\rho\right)<\omega\left(\rho'\right)+\omega\left(\rho''\right)$. Indeed, if $\tau = \Delta^{k}\tau_0$ is the Garside normal form of $\tau$, then we have $\rho\left(\tau_0\right)_{\left(k\right)} = \Delta^{\omega\left(\rho'\right)+\omega\left(\rho''\right)-k}$, where $\left(\tau_0\right)_{\left(k\right)}\in B_3^{+}$ is a positive braid indivisible by $\Delta$. In particular, $\omega\left(\rho\right) + k = \omega\left(\rho'\right) + \omega\left(\rho''\right)$ in this case. Therefore, the statement is established.
\end{proof}

The assumption that $\tau = \tau''\tau'$ is indivisible by $\Delta$ in Proposition~\ref{ppdb} is crucial for it to be true that $\rho$ is dual to $\tau$. Indeed, the right-dual of a right-divisor of a positive braid $\sigma\in B_3^{+}$ is not necessarily a left-divisor of the right-dual of $\sigma$. For example, $\sigma_2$ is a right-divisor of $\sigma_1\sigma_2$ and the right-dual of $\sigma_1\sigma_2$ is $\sigma_1$. On the other hand, the right-dual of $\sigma_2$ is $\sigma_1\sigma_2$, which is not a left-divisor of $\sigma_1$. In this case, the product of the right-dual of $\sigma_2 = \left(\sigma_1\right)_{\left(1\right)}$ and the right-dual of $\sigma_2$ is not the right-dual of $\sigma_1\sigma_2$, since the product is $\sigma_1\sigma_2\sigma_1\sigma_2$, which is divisible by $\Delta$. However, if we factor $\Delta$ from this product (as in the proof of Proposition~\ref{ppdb}), then we obtain $\sigma_1$ which \textit{is} the right-dual of the product $\sigma_1\sigma_2$.

We now wish to characterize those right-divisors $\rho''$ of a positive braid $\rho$, for which the right-dual of $\rho''$ is a left-divisor of the right-dual of $\rho$. The study of such right-divisors of $\rho$ is of fundamental importance in this paper.

\begin{definition}
\label{dcd}
Let $\rho,\rho',\rho''\in B_3^{+}$ be positive braids indivisible by $\Delta$ such that $\rho =\rho'\rho''$. Let $\rho$ be dual to $\tau$, $\rho'_{\left(\omega\left(\rho''\right)\right)}$ be dual to $\tau'$, and $\rho''$ be dual to $\tau''$. If $\tau = \tau''\tau'$, then we write that $\rho'$ is a \textit{cut left-divisor} of $\rho$ and $\rho''$ is a \textit{cut right-divisor} of $\rho$.
\end{definition}

If $\rho = \rho'\rho''$, where $\rho,\rho',\rho''\in B_3^{+}$ are positive braids indivisible by $\Delta$, then Proposition~\ref{ppdb} implies that $\omega\left(\rho\right) = \omega\left(\rho'\right)+\omega\left(\rho''\right)$ if and only if $\rho'$ is a cut left-divisor of $\rho$ and $\rho''$ is a cut right-divisor of $\rho$.

We now establish a simple criterion to determine if a left-divisor or a right-divisor of a positive braid $\rho$ is a cut left-divisor or a cut right-divisor of $\rho$, respectively.

\begin{lemma}
\label{lcutdivisorcharacterization}
If $\rho = \rho'\rho''$ for positive braids $\rho,\rho',\rho''\in B_3^{+}$ indivisible by $\Delta$, then $\rho'$ is a cut left-divisor of $\rho$ and $\rho''$ is a cut right-divisor of $\rho$ if and only if the index of the last generator in $\rho'$ is equal to the index of the first generator in $\rho''$.
\end{lemma}
\begin{proof}
Let $\rho'_{\left(\omega\left(\rho''\right)\right)}$ be dual to $\tau'$ and let $\rho''$ be dual to $\tau''$. We claim that $\tau = \tau''\tau'$ is indivisible by $\Delta$ if and only if the index of the last generator in $\rho'$ is equal to the index of the first generator in $\rho''$. The statement is equivalent to this claim since Proposition~\ref{ppdb} implies that $\rho$ is dual to $\tau$ if and only if $\tau$ is indivisible by $\Delta$. We will establish the claim by induction on the length of $\rho$. The claim is trivial if $\rho'' = \rho$, so let us assume that the length of $\rho''$ is strictly less than the length of $\rho$ and appeal to the induction hypothesis. 

We can write $\rho'' = \sigma_i\xi''$ for $i\in \{1,2\}$ and a positive braid $\xi''\in B_3^{+}$. Firstly, we will show that $\tau = \tau''\tau'$ is divisible by $\Delta$ if the index of the last generator in $\rho'$ is different from the index of the first generator in $\rho''$. Indeed, in this case, the index of the first generator in $\xi''$ is the same as $i$ by Proposition~\ref{piDB3} since $\rho$ is indivisible by $\Delta$. Therefore, we can apply the induction hypothesis to deduce that $\xi''$ is a cut right-divisor of $\rho''$. In particular, $\tau''$ ends with a product of distinct generators $\left(\sigma_{i'}\sigma_i\right)_{\left(\omega\left(\xi''\right)\right)}$ for $i\neq i'\in \{1,2\}$, and this product is $\left(\sigma_i\sigma_{i'}\right)_{\left(\omega\left(\rho''\right)\right)}$ by Proposition~\ref{ppdb}. However, $\tau'$ begins with $\left(\sigma_i\right)_{\omega\left(\rho''\right)}$. Indeed, $\rho'_{\left(\omega\left(\rho''\right)\right)}$ is dual to $\tau'$ and the last generator in $\rho'$ is $\sigma_{i'}$ by assumption. In particular, Proposition~\ref{pdualbraidproduct} implies that $\tau = \tau''\tau'$ is divisible by $\Delta$.

Secondly, we will show that $\tau = \tau''\tau'$ is indivisible by $\Delta$ if the index of the last generator in $\rho'$ is the same as the index of the first generator in $\rho''$. We split the argument into cases according to the index of the first generator in $\xi''$. If the index of the first generator in $\xi''$ is the same as $i$, then we have already established that $\tau''$ ends with a product of distinct generators $\left(\sigma_i\sigma_{i'}\right)_{\left(\omega\left(\rho''\right)\right)}$ for $i\neq i'\in \{1,2\}$. However, in this case, $\tau'$ begins with $\left(\sigma_{i'}\right)_{\left(\omega\left(\rho''\right)\right)}$. Indeed, $\left(\rho'\right)_{\left(\omega\left(\rho''\right)\right)}$ is dual to $\tau'$ and the last generator in $\rho'$ is $\sigma_i$ by assumption. In particular, Proposition~\ref{pdualbraidproduct} implies that $\tau = \tau''\tau'$ is indivisible by $\Delta$ in this case. 

Finally, we assume that the index of the first generator in $\xi''$ is different from $i$, and we will show that $\tau''\tau'$ is indivisible by $\Delta$. In this case, we can write $\rho''=\sigma_i\sigma_{i'}\xi''_{0}$ for a positive braid $\xi''_{0}\in B_3^{+}$ (we have $\xi''=\sigma_{i'}\xi''_{0}$). The index of the first generator in $\xi''_{0}$ is the same as $i'$ by Proposition~\ref{piDB3} since $\rho$ is indivisible by $\Delta$. Therefore, we can apply the induction hypothesis to deduce that $\xi''_{0}$ is a cut right-divisor of $\rho''$. In particular, $\tau''$ ends with $\left(\sigma_i\right)_{\left(\omega\left(\xi''_{0}\right)\right)}$, and this is $\left(\sigma_{i'}\right)_{\left(\omega\left(\rho''\right)\right)}$ by Proposition~\ref{ppdb}. However, in this case, $\tau'$ begins with $\left(\sigma_{i'}\right)_{\left(\omega\left(\rho''\right)\right)}$. Indeed, $\left(\rho'\right)_{\left(\omega\left(\rho''\right)\right)}$ is dual to $\tau'$ and the last generator in $\rho'$ is $\sigma_i$ by assumption. In particular, Proposition~\ref{pdualbraidproduct} implies that $\tau=\tau''\tau'$ is indivisible by $\Delta$ in this case. Therefore, the statement is established.
\end{proof}

If $\rho,\tau\in B_3^{+}$ are positive braids and if $\rho$ is dual to $\tau$, then Proposition~\ref{puniquenessofduals} implies that $\tau$ is uniquely determined by $\rho$. We now establish an explicit characterization of $\tau$ in terms of $\rho$ using Lemma~\ref{lcutdivisorcharacterization}. We remark that if $\rho$ is either a generator $\sigma_i$ for some $i\in \{1,2\}$ or a product of distinct generators $\sigma_{i}\sigma_{i'}$ for some $\{i,i'\}= \{1,2\}$, then it is straightforward to determine $\tau$. In the first case, $\tau = \sigma_{i'}\sigma_i$ for $i\neq i'\in \{1,2\}$ and in the second case, $\tau = \sigma_i$. (A generator or a product of distinct generators in the braid group $B_3$ is referred to as a \textit{permutation braid} in $B_3$.)

\begin{lemma}
\label{lcharacterizationdualpositivebraids}
Let $\rho,\tau\in B_3^{+}$ be positive braids such that $\rho$ is dual to $\tau$. Let $\rho = \prod_{j=1}^{k} \rho_j$ be a product decomposition of $\rho$, where $\rho_j$ is either a single generator or a product of two distinct generators for each $1\leq j\leq k$, and the index of the last generator in $\rho_j$ is equal to the index of the first generator in $\rho_{j+1}$ for each $1\leq j\leq k-1$. 

If $\rho_j$ is dual to $\tau_j$ for each $1\leq j\leq k$, then $\tau = \prod_{j=1}^{k} \left(\tau_{k+1-j}\right)_{\left(j+1\right)}$. In particular, if $\rho$ ends with a product of two distinct generators, then $\tau$ begins with a product of two generators with the same index (unless $\tau$ itself is a generator), and vice versa. If $\rho$ begins with a product of two distinct generators, then $\tau$ ends with a product of two generators with the same index (unless $\tau$ itself is a generator), and vice versa. Finally, we have $k=\omega\left(\tau\right)$.
\end{lemma}
\begin{proof}
The statement that $\rho$ is dual to $\tau$ is a consequence of Lemma~\ref{lcutdivisorcharacterization} and induction on $k$. If $\rho$ ends with a product of two distinct generators, then $\rho_k$ is a product of two distinct generators and we can write $\rho_k = \sigma_{i}\sigma_{i'}$ for $\{i,i'\} = \{1,2\}$. In this case, $\tau_k = \sigma_i$ is a single generator. If $\rho\neq \rho_k$, then Proposition~\ref{piDB3} implies that the last generator in $\rho_{k-1}$ is $\sigma_{i}$, since $\rho$ is indivisible by $\Delta$. Furthermore, in this case, the first generator in $\tau_{k-1}$ is $\sigma_{i'}$. In particular, the first generator in $\left(\tau_{k-1}\right)_{\left(3\right)}$ is $\sigma_i$ and the product of the first two generators in $\tau$ is $\sigma_i\sigma_i$. The proofs of the other statements concerning the relationship between the beginning/end of $\rho$ and the end/beginning of $\tau$, respectively, are similar. Finally, Proposition~\ref{ppdb} implies that $k=\omega\left(\tau\right)$ by induction on $k$, since $\omega\left(\rho_j\right) = 1$ for each $1\leq j\leq k$. Therefore, the statement is established.
\end{proof}

We now refine Proposition~\ref{ppdb} to establish that the failure of $\omega$ to be additive is minimal.

\begin{proposition}
\label{pomegaalmostadditive}
If $\rho,\rho',\rho''\in B_3^{+}$ are positive braids indivisible by $\Delta$ such that $\rho = \rho'\rho''$, then $\omega\left(\rho\right) = \omega\left(\rho'\right) + \omega\left(\rho''\right) - \epsilon$ for some $\epsilon\in \{0,1\}$. Furthermore, $\epsilon = 0$ if and only if $\rho'$ is a cut left-divisor of $\rho$ or $\rho''$ is a cut right-divisor of $\rho$. 
\end{proposition}
\begin{proof}
If $\rho'$ is a cut left-divisor of $\rho$ or if $\rho''$ is a cut right-divisor of $\rho$, then Proposition~\ref{ppdb} implies that $\omega\left(\rho\right) = \omega\left(\rho'\right) + \omega\left(\rho''\right)$. If $\rho'$ is not a cut left-divisor of $\rho$, then Lemma~\ref{lcutdivisorcharacterization} implies that the index of the last generator in $\rho'$ is different from the index of the first generator in $\rho''$. Let $\sigma_i$ be the last generator in $\rho'$ and let $\sigma_{i'}$ be the first generator in $\rho''$ for $i\neq i'\in \{1,2\}$. 

Proposition~\ref{piDB3} implies that the second generator in $\rho''$ is not $\sigma_i$, since $\rho$ is indivisible by $\Delta$. Lemma~\ref{lcutdivisorcharacterization} now implies that $\rho'\sigma_{i'}$ is a cut left-divisor of $\rho$. Furthermore, Lemma~\ref{lcharacterizationdualpositivebraids} implies that $\omega\left(\rho'\sigma_{i'}\right) = \omega\left(\rho'\right)$. Indeed, Proposition~\ref{piDB3} implies that the second last generator in $\rho'$ is not $\sigma_{i'}$, since $\rho$ is indivisible by $\Delta$. Lemma~\ref{lcharacterizationdualpositivebraids} also implies that $\omega\left(\sigma_{i'}^{-1}\rho''\right) = \omega\left(\rho''\right) - 1$, since the second generator in $\rho''$ is not $\sigma_i$. Finally, Proposition~\ref{ppdb} implies that $\omega\left(\rho\right) = \omega\left(\rho'\sigma_{i'}\right) + \omega\left(\sigma_{i'}^{-1}\rho''\right)$. Therefore, the statement is established.
\end{proof}

We establish some useful statements concerning cut divisors that follow in a straightforward manner from the results thus far.

\begin{proposition}
\label{pcutdivisorsubdivisor}
Let $\rho\in B_3^{+}$ be a positive braid indivisible by $\Delta$. Let $\rho'$ and $\rho''$ be left-divisors of $\rho$. If $\rho'$ is a cut left-divisor of $\rho$ and if $\rho'$ is a left-divisor of $\rho''$, then $\rho'$ is a cut left-divisor of $\rho''$. Conversely, if $\rho'$ is a cut left-divisor of $\rho''$ and $\rho'\neq \rho''$, then $\rho'$ is a cut left-divisor of $\rho$.
\end{proposition}
\begin{proof}
The statement is an immediate consequence of Lemma~\ref{lcutdivisorcharacterization}.
\end{proof}

We now establish a monotonicity property of the function $\omega$ defined on positive braids indivisible by $\Delta$.

\begin{proposition}
\label{pomeganumbercutdivisor}
Let $\rho\in B_3^{+}$ be a positive braid indivisible by $\Delta$. Let $\rho'$ and $\rho''$ be left-divisors of $\rho$. If $\rho'$ is a left-divisor of $\rho''$, then $\omega\left(\rho'\right)\leq \omega\left(\rho''\right)$. If $\rho'$ and $\rho''$ are cut left-divisors of $\rho$, then $\rho'$ is a left-divisor of $\rho''$ if and only if $\omega\left(\rho'\right)\leq \omega\left(\rho''\right)$, and equality holds if and only if $\rho' = \rho''$.
\end{proposition}
\begin{proof}
We use the fact that $\omega\left(\sigma\right)>0$ if and only if $\sigma\in B_3^{+}$ is a nonidentity positive braid. If $\rho'$ is a left-divisor of $\rho''$, then Proposition~\ref{pomegaalmostadditive} implies that $\omega\left(\rho'\right)\leq \omega\left(\rho''\right)$. If $\rho'$ and $\rho''$ are left-divisors of $\rho$, then Proposition~\ref{piDB3} implies that either $\rho'$ is a left-divisor of $\rho''$ or $\rho''$ is a left-divisor of $\rho'$. If $\rho'$ and $\rho''$ are cut left-divisors of $\rho$, then Proposition~\ref{pcutdivisorsubdivisor} implies that either $\rho'$ is a cut left-divisor of $\rho''$ or $\rho''$ is a cut left-divisor of $\rho'$. Proposition~\ref{ppdb} now implies that $\rho'$ is a left-divisor of $\rho''$ if and only if $\omega\left(\rho'\right)\leq \omega\left(\rho''\right)$, and equality holds if and only if $\rho'=\rho''$. Therefore, the statement is established.
\end{proof}

Proposition~\ref{piDB3} implies that it is meaningful to refer to ``the maximal left-divisor" or ``the maximal right-divisor" of a positive braid indivisible by $\Delta$, with respect to a property. The following statement is useful for furnishing cut divisors in many situations. 

\begin{lemma}
\label{lmaximaldivisorcutdivisor}
Let $\rho,\sigma\in B_3^{+}$ be positive braids indivisible by $\Delta$ such that $\rho\sigma$ is divisible by $\Delta$. The maximal right-divisor $\rho''$ of $\rho$ that is dual to a left-divisor of $\sigma$ is a cut right-divisor of $\rho$. In particular, if $\rho$ is dual to $\tau$, and if $\rho''$ is dual to the left-divisor $\sigma''$ of $\sigma$, then $\sigma''$ is a left-divisor of $\tau$. 
\end{lemma}
\begin{proof}
If $\rho''$ is not a cut right-divisor of $\rho$, then Lemma~\ref{lcutdivisorcharacterization} implies that $\rho = \rho'\rho''$ for a positive braid $\rho'\in B_3^{+}$, such that the index of the last generator in $\rho'$ is different from the index of the first generator in $\rho''$. We will show that $\rho''$ is not the maximal right-divisor of $\rho$ that is dual to a left-divisor of $\sigma$ in this case.

Firstly, we write $\sigma = \sigma''\sigma'$ for positive braids $\sigma',\sigma''\in B_3^{+}$ such that $\rho''$ is dual to $\sigma''$. We can write $\rho''=\sigma_i\rho''_{0}$ for $i\in \{1,2\}$ and a positive braid $\rho''_{0}\in B_3^{+}$. In this case, the index of the first generator in $\rho''_{0}$ is the same as $i$ by Proposition~\ref{piDB3} since $\rho$ is indivisible by $\Delta$. Therefore, Lemma~\ref{lcharacterizationdualpositivebraids} implies that $\sigma''$ ends with a product of distinct generators $\left(\sigma_{i'}\sigma_i\right)_{\left(\omega\left(\rho_0''\right)\right)}$ for $i\neq i'\in \{1,2\}$. We can write $\sigma'' = \sigma_0''\left(\sigma_{i'}\sigma_i\right)_{\left(\omega\left(\rho_0''\right)\right)}$ for a positive braid $\sigma_0''\in B_3^{+}$ ($\sigma_0''$ is the right-dual of $\rho_0''$). However, this implies that $\left(\sigma_{i'}\right)\rho''$ is dual to $\sigma_{0}''\left(\sigma_{i'}\right)_{\left(\omega\left(\rho_0''\right)\right)}$, which shows that $\rho''$ is not the maximal-right divisor of $\rho$ that is dual to a left-divisor of $\sigma$.

The definition of a cut right-divisor (Definition~\ref{dcd}) implies that $\sigma''$ is a left-divisor of $\tau$. Therefore, the statement is established.
\end{proof}

We now establish a characterization of the maximal divisors in the statement of Lemma~\ref{lmaximaldivisorcutdivisor}.

\begin{proposition}
\label{pmaximaldivisordualindivisibleDelta}
Let $\sigma,\tau\in B_3^{+}$ be positive braids indivisible by $\Delta$. Let $\sigma''$ be a right-divisor of $\sigma$ that is dual to a left-divisor $\tau''$ of $\tau$. Let us write $\sigma = \sigma'\sigma''$ and $\tau = \tau''\tau'$ for positive braids $\sigma',\tau'\in B_3^{+}$. If $\sigma''$ is the maximal right-divisor of $\sigma$ that is left-dual to a left-divisor of $\tau$, then the product $\left(\sigma'\right)_{\left(\omega\left(\sigma''\right)\right)}\tau'$ is indivisible by $\Delta$. Conversely, if the product $\left(\sigma'\right)_{\left(\omega\left(\sigma''\right)\right)}\tau'$ is indivisible by $\Delta$ and if $\sigma''$ is a cut right-divisor of $\sigma$, then $\sigma''$ is the maximal right-divisor of $\sigma$ that is left-dual to a left-divisor of $\tau$.
\end{proposition}
\begin{proof}
Let $k$ be the maximal positive integer such that $\Delta^{k}$ divides $\sigma\tau$. Proposition~\ref{pdualbraidproduct} implies that $\sigma''\tau'' = \Delta^{k}$ if and only if $\left(\sigma'\right)_{\left(\omega\left(\sigma''\right)\right)}\tau'$ is indivisible by $\Delta$. Of course, if $\sigma''$ is the maximal right-divisor of $\sigma$ that is left-dual to a left-divisor of $\tau$, then $\sigma''\tau'' = \Delta^{k}$. Conversely, if $\sigma''\tau'' = \Delta^{k}$ and if $\sigma''$ is a cut right-divisor of $\sigma$, then Proposition~\ref{pomeganumbercutdivisor} and Lemma~\ref{lmaximaldivisorcutdivisor} imply that $\sigma''$ is the maximal right-divisor of $\sigma$ that is dual to a left-divisor of $\tau$. Therefore, the statement is established. 
\end{proof}

We observe another characterization of the maximal divisors in the statement of Lemma~\ref{lmaximaldivisorcutdivisor}.

\begin{proposition}
\label{pmaximaldivisordualcharacterization}
Let $\sigma,\tau\in B_3^{+}$ be positive braids indivisible by $\Delta$. Let $\sigma''$ be a right-divisor of $\sigma$ that is dual to a left-divisor $\tau''$ of $\tau$. Let us write $\sigma = \sigma'\sigma''$ and $\tau = \tau''\tau'$ for positive braids $\sigma',\tau'\in B_3^{+}$.

The right-divisor $\sigma''$ of $\sigma$ is the maximal right-divisor of $\sigma$ that is dual to a left-divisor of $\tau$ if and only if $\sigma''$ is a cut right-divisor of $\sigma$ and one of the following conditions is satisfied:
\begin{description}
\item[(i)] The right-divisor $\tau'$ of $\tau$ is not a cut right-divisor of $\tau$. In this case, the index of the last generator in $\left(\sigma'\right)_{\left(\omega\left(\sigma''\right)\right)}$ is equal to the index of the first generator in $\tau'$.
\item[(ii)] If the right-divisor $\tau'$ of $\tau$ is a cut right-divisor of $\tau$, then the left-divisor $\sigma'$ of $\sigma$ does not end with a product of two distinct generators and the right-divisor $\tau'$ of $\tau$ does not begin with a product of two distinct generators. In this case, the index of the last generator in $\left(\sigma'\right)_{\left(\omega\left(\sigma''\right)\right)}$ is different from the index of the first generator in $\tau'$. 
\end{description}
\end{proposition}
\begin{proof}
Lemma~\ref{lmaximaldivisorcutdivisor} and Proposition~\ref{pmaximaldivisordualindivisibleDelta} imply that $\sigma''$ is the maximal right-divisor of $\sigma$ that is dual to a left-divisor of $\tau$ if and only if $\sigma''$ is a cut right-divisor of $\sigma$ and the product $\left(\sigma'\right)_{\left(\omega\left(\sigma''\right)\right)}\tau'$ is indivisible by $\Delta$. However, Proposition~\ref{pdualbraidproduct} and Lemma~\ref{lcutdivisorcharacterization} imply that the product $\left(\sigma'\right)_{\left(\omega\left(\sigma''\right)\right)}\tau'$ is indivisible by $\Delta$ if and only if one of the conditions in the statement is satisfied.
\end{proof}

An important application of the theory of duality is to determine the Garside normal form of a product in terms of the Garside normal forms of the individual factors. 

\begin{proposition}
\label{pGarsidenormalformproduct}
Let $g = \Delta^{k}\sigma$ and $h = \Delta^{l}\tau$ be the Garside normal forms of $g,h\in B_3$, respectively. Let $\sigma''_{\left(l\right)}$ be the maximal right-divisor of $\sigma_{\left(l\right)}$ that is dual to a left-divisor of $\tau$. If $\sigma''_{\left(l\right)}$ is dual to $\tau''$, then write $\sigma = \sigma'\sigma''$ and $\tau = \tau''\tau'$ for positive braids $\sigma',\tau'\in B_3^{+}$. In this case, the Garside normal form of $gh$ is $gh = \Delta^{k+l+\omega\left(\sigma''\right)} \sigma'_{\left(l+\omega\left(\sigma''\right)\right)}\tau'$. 
\end{proposition}
\begin{proof}
The statement is equivalent to the statement that $\sigma'_{\left(l+\omega\left(\sigma''\right)\right)}\tau'$ is indivisible by $\Delta$. However, this is a consequence of Proposition~\ref{pmaximaldivisordualindivisibleDelta}. Therefore, the statement is established.
\end{proof}

We can view the Garside power of a braid as a measure of the complexity of the braid. Indeed, positive braids in $B_3$ indivisible by $\Delta$ are sufficiently simple to study since each such positive braid has a unique product expansion according to Proposition~\ref{piDB3}. The following consequence of Proposition~\ref{pGarsidenormalformproduct} characterizes how this measure of complexity changes with respect to products. 

\begin{proposition}
\label{preversetriangleinequality}
Let $g = \Delta^{k}\sigma$ and $h = \Delta^{l}\tau$ be the Garside normal forms of $g,h\in B_3$, respectively. If $\sigma''_{\left(l\right)}$ is the maximal right-divisor of $\sigma_{\left(l\right)}$ that is left-dual to a left-divisor of $\tau$, then the Garside power of $gh$ is equal to the sum of the Garside power of $g$, the Garside power of $h$, and $\omega\left(\sigma''_{\left(l\right)}\right)$. In particular, the Garside power of $gh$ is at least the sum of the Garside power of $g$ and the Garside power of $h$, with equality if and only if $\sigma_{\left(l\right)}\tau$ is indivisible by $\Delta$.
\end{proposition}
\begin{proof}
The statement is an immediate consequence of Proposition~\ref{pGarsidenormalformproduct}.
\end{proof}

We refer to the inequality in Proposition~\ref{preversetriangleinequality} as the \textit{reverse triangle inequality} for Garside powers. 

We have studied the failure of the Garside element $\Delta$ to be prime in the monoid $B_3^{+}$ of positive braids. In particular, we have seen that a product $\sigma\tau$ of positive braids $\sigma,\tau\in B_3^{+}$ can be divisible by $\Delta$ even if neither $\sigma$ nor $\tau$ is divisible by $\Delta$. We now wish to consider a variation of this statement for triple products. If $\sigma,\sigma',\sigma''\in B_3^{+}$ are positive braids, then it is possible for $\sigma\sigma'$ and $\sigma'\sigma''$ to be indivisible by $\Delta$ even if $\sigma\sigma'\sigma''$ is divisible by $\Delta$. For example, this occurs if $\sigma' = \sigma_i$ is a generator for $i\in \{1,2\}$, and the indices of the last generator in $\sigma$ and the first generator in $\sigma''$ are equal to each other and different from $i$. In this case, a copy of the Garside element $\Delta$ overlaps with all three factors in $\sigma\sigma'\sigma''$, in the sense that one generator of this copy of $\Delta$ is in each of the three factors. However, if we arrange that $\sigma$ does not end with a product of two distinct generators and $\sigma''$ does not begin with a product of two distinct generators, then Proposition~\ref{pdualbraidproduct} implies that $\sigma\sigma'$ and $\sigma'\sigma''$ are indivisible by $\Delta$. We establish that this is the only instance in which this phenomenon occurs.

\begin{lemma}
\label{ltripleproductindivisibleDelta}
Let $\sigma,\sigma',\sigma''\in B_3^{+}$ be positive braids indivisible by $\Delta$ such that $\sigma'$ is not a generator (i.e., $\sigma'$ has word length greater than one). If $\sigma\sigma'$ and $\sigma'\sigma''$ are indivisible by $\Delta$, then $\sigma\sigma'\sigma''$ is indivisible by $\Delta$.
\end{lemma}
\begin{proof}
If $\sigma\sigma'\sigma''$ is divisible by $\Delta$, then Proposition~\ref{piDB3} implies that we must have a copy of $\Delta$ in the product of the words representing $\sigma$, $\sigma'$, and $\sigma''$. Furthermore, the copy of $\Delta$ must overlap with each of these words, since otherwise that would contradict the hypothesis that $\sigma\sigma'$ and $\sigma'\sigma''$ are indivisible by $\Delta$. However, this forces $\sigma'$ to be a word of length one, which is also a contradiction. Therefore, the statement is established.
\end{proof}

We have the following generalization of Lemma~\ref{ltripleproductindivisibleDelta}, which characterizes precisely when the product of any number of positive braids is indivisible by $\Delta$. We will apply the statement several times in Section~\ref{smain}.

\begin{lemma}
\label{lproductindivisibleDelta}
Let $\xi_1,\xi_2,\dots,\xi_k\in B_3^{+}$ be positive braids indivisible by $\Delta$. The product $\xi_1\xi_2\cdots \xi_k$ is indivisible by $\Delta$ if and only if the following conditions are satisfied:
\begin{description}
\item[(i)] If $1\leq j\leq k-1$, then the product $\xi_j\xi_{j+1}$ is indivisible by $\Delta$ (for example, this is the case if the index of the last generator in $\xi_j$ is equal to the index of the first generator in $\xi_{j+1}$).
\item[(ii)] If $1\leq j\leq k-2$, and if $\xi_{j+1}$ is a generator, then the triple product $\xi_j\xi_{j+1}\xi_{j+2}$ is indivisible by $\Delta$.
\end{description}
\end{lemma}
\begin{proof}
The statement is a consequence of Proposition~\ref{piDB3} and the proof is similar to that of Lemma~\ref{ltripleproductindivisibleDelta}.
\end{proof}

The following notation and terminology will be convenient in Subsection~\ref{sspowerpositivehalftwist}, when we study the Garside normal form of the power of a positive half-twist in the braid group $B_3$.

\begin{definition}
\label{dcutclosure}
Let $\sigma\in B_3^{+}$ be a positive braid indivisible by $\Delta$ and let us write $\sigma = \sigma'\sigma''$, where $\sigma',\sigma''\in B_3^{+}$ are positive braids. The \textit{cut closure} of $\sigma'$ is the minimal cut left-divisor $\overline{\sigma'}$ of $\sigma$ of which $\sigma'$ is a left-divisor.

Similarly, the \textit{cut closure} of $\sigma''$ is the minimal cut right-divisor $\overline{\sigma''}$ of $\sigma$ of which $\sigma''$ is a right-divisor. 

Finally, we define the \textit{right cut closure} of $\sigma$ to be the maximal positive braid $\overline{\sigma}$ (indivisible by $\Delta$) such that $\sigma$ is a left-divisor of $\overline{\sigma}$ and the cut closure of $\sigma$ is $\overline{\sigma}$. We define the \textit{left cut closure} similarly with the words ``left" and ``right" interchanged in the previous sentence. 
\end{definition}

The notation for the left cut closure and the right cut closure is the same but which of these we are referring to will be clear from the context. Similarly, when we refer to the cut closure of a positive braid, it will be clear from the context whether we are referring to its cut closure as a divisor of another positive braid, or its cut closure in isolation. More precisely, if there is clearly an underlying positive braid $\sigma$ such that $\sigma'$ is a left-divisor of $\sigma$, then $\overline{\sigma'}$ will always denote the cut closure of $\sigma'$ as a left-divisor of $\sigma$. We will use the following observation frequently.

\begin{proposition}
\label{pomegacutclosure}
If $\sigma\in B_3^{+}$ is a positive braid indivisible by $\Delta$, and if $\sigma'$ is either a left-divisor of $\sigma$ or a right-divisor of $\sigma$, then $\omega\left(\overline{\sigma'}\right) = \omega\left(\sigma'\right)$.
\end{proposition}
\begin{proof}
The statement is a consequence of Lemma~\ref{lcharacterizationdualpositivebraids}.
\end{proof}

We conclude this subsection with an important technical result that we will use in Subsection~\ref{sspowerpositivehalftwist}. The statement of the result may seem unmotivated at present, but we include it here because its proof is based on the theory of duality developed in this subsection. We will use this result in Subsection~\ref{sspowerpositivehalftwist} in order to compute the Garside normal form of the conjugate $g_1g_2g_1^{-1}$ of a power of a positive half-twist $g_2$ by a power of a positive half-twist $g_1$, in terms of the Garside normal forms of $g_1$ and $g_2$. If we reduce this computation to a computation in the monoid $B_3^{+}$ of positive braids based on our theory of duality, then we will see in Subsection~\ref{sspowerpositivehalftwist} that we effectively reduce to the following statement.

\begin{proposition}
\label{ptechnicaldualityconjugate}
Let $\rho,\tau\in B_3^{+}$ and $\xi,\chi\in B_3^{+}$ be positive braids such that $\rho$ is dual to $\tau$ and $\xi$ is dual to $\chi$. Let $\xi''$ be the maximal right-divisor of $\xi$ that is left-dual to a left-divisor of the cut closure $\left(\overline{\rho}\right)_{\left(\omega\left(\rho\right)\right)}$. Let us denote the relevant left-divisor of $\left(\overline{\rho}\right)_{\left(\omega\left(\rho\right)\right)}$ by $\left(\rho''\right)_{\left(\omega\left(\rho\right)\right)}$. Let $\chi''$ be the cut left-divisor of $\chi$ that is right-dual to $\xi''$ and let $\tau''$ be the right-divisor of $\tau$ that is right-dual to $\left(\overline{\rho''}\right)_{\left(\omega\left(\rho\right)-\omega\left(\rho''\right)\right)}$. Let us write $\chi = \chi''\chi'$ and $\tau = \tau'\tau''$ for positive braids $\chi',\tau'\in B_3^{+}$.
\begin{description}
\item[(i)] Let us assume that $\rho''$ is a proper cut left-divisor of $\rho$. In this case, $\tau'$ ends with a product of two distinct generators $\alpha\in B_3^{+}$, and the positive braid $\alpha\tau''$ is the maximal right-divisor of $\tau$ that is left-dual to a left-divisor of $\chi$. The relevant left-divisor of $\chi$ is $\chi''\beta$, where $\beta\in B_3^{+}$ is a single generator at the beginning of $\chi'$.
\item[(ii)] Let us assume that $\overline{\rho''} = \overline{\rho}$. In this case, $\tau''$ is left-dual to $\chi''$, and $\tau$ itself is the maximal right-divisor of $\tau$ that is left-dual to a left-divisor of $\chi$. If $\rho'' = \rho$, then the relevant left-divisor of $\chi$ is $\chi''$. If $\rho'' = \overline{\rho}\neq \rho$, then $\tau'$ is a single generator and the index of the generator $\left(\tau'\right)_{\left(\omega\left(\tau''\right)\right)}$ is equal to the index of the first generator in $\chi'$.
\item[(iii)] Let us assume that $\rho''$ is a left-divisor of $\rho$ but that $\rho''$ is not a cut left-divisor of $\rho$. In this case, the first generator $\alpha$ in $\tau''$ is a cut left-divisor of $\tau''$, and the last generator $\beta$ in $\chi''$ is a cut right-divisor of $\chi''$. Furthermore, the positive braid $\alpha^{-1}\tau''$ is the maximal right-divisor of $\tau$ that is left-dual to a left-divisor of $\chi$. The relevant left-divisor of $\chi$ is $\chi''\beta^{-1}$.
\end{description}
\end{proposition}
\begin{proof}
The pseudo-symmetry of duals (Proposition~\ref{ppseudosymmetryduality}) implies that the $\omega$-numbers of a pair of dual positive braids are equal. For example, $\omega\left(\rho\right) = \omega\left(\tau\right)$, $\omega\left(\xi\right) = \omega\left(\chi\right)$, $\omega\left(\xi''\right) = \omega\left(\rho''\right)$, $\omega\left(\rho''\right) = \omega\left(\tau''\right)$, and $\omega\left(\xi''\right) = \omega\left(\chi''\right)$. We will use the equal values of the $\omega$-numbers of such pairs interchangeably in this proof. 
\begin{description}
\item[(i)] Let us write $\rho = \rho''\rho'$ for a positive braid $\rho'\in B_3^{+}$. The hypothesis is that $\left(\rho''\right)_{\left(\omega\left(\rho\right)-\omega\left(\rho''\right)\right)}$ is left-dual to $\tau''$, and the pseudo-symmetry of duals (Proposition~\ref{ppseudosymmetryduality}) implies that $\tau''$ is left-dual to $\left(\rho''\right)_{\left(\omega\left(\rho\right)\right)}$. The uniqueness of duals (Proposition~\ref{puniquenessofduals}) now implies that $\tau'' = \xi''$, and thus $\tau''$ is left-dual to $\chi''$. 

If $\rho''$ is a proper cut left-divisor of $\rho$, then Proposition~\ref{pmaximaldivisordualcharacterization} implies that $\xi'$ does not end with a product of two distinct generators and $\rho'$ does not begin with a product of two distinct generators. Furthermore, Proposition~\ref{pmaximaldivisordualcharacterization} implies that if the last generator in $\xi'$ is $\left(\sigma_i\right)_{\left(\omega\left(\rho\right)+\omega\left(\rho''\right)\right)}$ where $i\in \{1,2\}$, then the first generator in $\rho'$ is $\sigma_{i'}$, where $i\neq i'\in \{1,2\}$.

The hypothesis that $\rho''$ is a cut left-divisor of $\rho$ and the fact that $\xi''$ is a cut right-divisor of $\xi$ (Lemma~\ref{lmaximaldivisorcutdivisor}) imply that $\rho'$ is left-dual to $\tau'$, and $\left(\xi'\right)_{\left(\omega\left(\rho''\right)\right)}$ is left-dual to $\chi'$. Lemma~\ref{lcharacterizationdualpositivebraids} now implies that $\tau'$ ends with $\left(\sigma_{i}\sigma_{i'}\right)_{\left(\omega\left(\rho'\right)-1\right)}$ and $\chi'$ begins with $\left(\sigma_{i'}\sigma_i\right)_{\left(\omega\left(\rho\right)\right)}$. In particular, $\left(\tau'\right)_{\left(\omega\left(\tau''\right)\right)}$ ends with $\left(\sigma_{i'}\sigma_i\right)_{\left(\omega\left(\rho\right)\right)}$ (Proposition~\ref{ppdb} implies that $\omega\left(\rho\right) = \omega\left(\rho'\right) + \omega\left(\rho''\right)$ since $\rho''$ is a cut left-divisor of $\rho$). We conclude that $\left(\tau'\right)_{\left(\omega\left(\tau''\right)\right)}\chi'$ is divisible by $\Delta$, but that $\left(\tau'\right)_{\left(\omega\left(\tau''\right)\right)}\chi'$ is indivisible by $\Delta^2$. Indeed, Proposition~\ref{piDB3} implies that the $\left(\sigma_{i'}\sigma_i\right)_{\left(\omega\left(\rho\right)\right)}$ at the end of $\left(\tau'\right)_{\left(\omega\left(\tau''\right)\right)}$ is not preceded by $\left(\sigma_{i}\right)_{\left(\omega\left(\rho\right)\right)}$. We can apply Proposition~\ref{pdualbraidproduct} to show that $\left(\tau'\right)_{\left(\omega\left(\tau''\right)\right)}\chi'$ is indivisible by $\Delta^2$. Therefore, Proposition~\ref{pmaximaldivisordualindivisibleDelta} implies the statement with $\alpha = \left(\sigma_{i}\sigma_{i'}\right)_{\left(\omega\left(\rho'\right)-1\right)}$ and $\beta = \left(\sigma_{i'}\right)_{\left(\omega\left(\rho\right)\right)}$.
\item[(ii)] If $\overline{\rho''} = \overline{\rho}$, then either $\rho'' = \rho$ or $\rho'' = \overline{\rho}\neq \rho$. If $\rho'' = \rho$, then $\rho''$ is dual to $\tau = \tau''$. The pseudo-symmetry of duals (Proposition~\ref{ppseudosymmetryduality}) implies that $\tau$ is left-dual to $\left(\rho''\right)_{\left(\omega\left(\rho\right)\right)}$. The uniqueness of duals (Proposition~\ref{puniquenessofduals}) implies that $\tau= \xi''$, and thus $\tau$ is left-dual to $\chi''$. Therefore, the statement is established if $\rho'' = \rho$.

If $\rho'' = \overline{\rho}\neq \rho$, then the cut closure $\overline{\rho}$ ends with a product of distinct generators $\sigma_i\sigma_{i'}$, where $\{i,i'\} = \{1,2\}$, the last generator in $\rho$ is $\sigma_i$, and $\sigma_i$ is non-isolated in $\rho$. Lemma~\ref{lcharacterizationdualpositivebraids} implies that $\tau$ begins with a product of distinct generators $\sigma_{i'}\sigma_i$, since $\rho$ is dual to $\tau$. In this case, $\tau' = \sigma_{i'}$ and $\overline{\rho}$ is left-dual to $\tau'' = \left(\tau'\right)^{-1}\tau$. Furthermore, Lemma~\ref{lcharacterizationdualpositivebraids} implies that $\xi''$ begins with $\left(\sigma_{i'}\right)_{\left(\omega\left(\rho\right) + \omega\left(\rho\right) - 1\right)} = \sigma_{i}$, and $\chi''$ ends with a product of distinct generators $\left(\sigma_{i'}\sigma_{i}\right)_{\left(\omega\left(\rho\right)-1\right)}$. In particular, Proposition~\ref{piDB3} implies that $\chi'$ begins with $\left(\sigma_{i'}\right)_{\left(\omega\left(\rho\right)\right)}$, since $\chi$ is indivisible by $\Delta$. Of course, $\left(\tau'\right)_{\left(\omega\left(\tau''\right)\right)}=\left(\sigma_{i'}\right)_{\left(\omega\left(\rho\right)\right)}$. Proposition~\ref{pdualbraidproduct} now implies that $\left(\tau'\right)_{\left(\omega\left(\tau''\right)\right)}\chi'$ is indivisible by $\Delta$. Therefore, Proposition~\ref{pmaximaldivisordualindivisibleDelta} implies the statement. 
\item[(iii)] If $\rho''$ is a left-divisor of $\rho$ but not a cut left-divisor of $\rho$, then $\rho''$ is a proper left-divisor of its cut closure $\overline{\rho''}$. Moreover, the cut closure $\overline{\rho''}$ ends with a product $\sigma_i\sigma_{i'}$ of distinct generators, where $\{i,i'\} = \{1,2\}$ and $\sigma_i$ is the last generator of $\rho''$. Let us write $\rho = \overline{\rho''}\rho'$ for a positive braid $\rho'\in B_3^{+}$.

Lemma~\ref{lcharacterizationdualpositivebraids} implies that $\xi''$ begins with a product of distinct generators $\left(\sigma_i\sigma_{i'}\right)_{\left(\omega\left(\rho\right)+\omega\left(\rho''\right)-1\right)}$, since $\xi''$ is left-dual to $\left(\rho''\right)_{\left(\omega\left(\rho\right)\right)}$, and $\tau''$ begins with $\left(\sigma_i\right)_{\left(\omega\left(\rho'\right)\right)}$, since $\left(\overline{\rho''}\right)_{\left(\omega\left(\rho'\right)\right)}$ is left-dual to $\tau''$ (Proposition~\ref{ppdb} and Proposition~\ref{pomegacutclosure} imply that $\omega\left(\rho'\right)+\omega\left(\rho''\right) = \omega\left(\rho\right)$). Similarly, Lemma~\ref{lcharacterizationdualpositivebraids} implies that $\chi''$ ends with $\left(\sigma_i\right)_{\left(\omega\left(\rho\right)\right)}$. Furthermore, Lemma~\ref{lcharacterizationdualpositivebraids} implies that the generator at the beginning of $\tau''$ is a cut left-divisor of $\tau$, and the generator at the end of $\chi''$ is a cut right-divisor of $\chi''$.

The following statements (in this paragraph) are all consequences of Lemma~\ref{lcharacterizationdualpositivebraids}. If $\alpha' = \sigma_i$ and $\alpha = \left(\sigma_i\right)_{\left(\omega\left(\rho'\right)\right)}$, then the positive braid $\left(\rho''\left(\alpha'\right)^{-1}\right)_{\left(\omega\left(\rho'\right) + 1\right)}$ is left-dual to $\alpha^{-1}\tau''$.  If $\beta' = \left(\sigma_i\sigma_{i'}\right)_{\left(\omega\left(\rho\right)+\omega\left(\rho''\right)-1\right)}$ and $\beta = \left(\sigma_i\right)_{\left(\omega\left(\rho\right)\right)}$, then $\left(\beta'\right)^{-1}\xi''$ is left-dual to $\chi''\beta^{-1}$. Finally, $\left(\beta'\right)^{-1}\xi''$ is left-dual to $\left(\rho''\left(\alpha'\right)^{-1}\right)_{\left(\omega\left(\rho\right)\right)}$.

The pseudo-symmetry of duals (Proposition~\ref{ppseudosymmetryduality}) implies that $\alpha^{-1}\tau''$ is left-dual to $\left(\rho''\left(\alpha'\right)^{-1}\right)_{\left(\omega\left(\rho\right)\right)}$. The uniqueness of duals (Proposition~\ref{puniquenessofduals}) implies that $\alpha^{-1}\tau'' = \left(\beta'\right)^{-1}\xi''$, and thus $\alpha^{-1}\tau''$ is left-dual to $\chi''\beta^{-1}$. However, $\left(\alpha\right)_{\left(\omega\left(\tau''\right)\right)} = \left(\sigma_i\right)_{\left(\omega\left(\rho\right)\right)} = \beta$. Proposition~\ref{pdualbraidproduct} now implies that $\left(\tau'\alpha\right)_{\left(\omega\left(\tau''\right)\right)}\beta\chi'$ is indivisible by $\Delta$. Therefore, Proposition~\ref{pmaximaldivisordualindivisibleDelta} implies the statement.
\end{description}
Therefore, the cases in the statement have all been established and the proof is complete.
\end{proof}

\subsection{The Garside normal form of the power of a positive half-twist}
\label{sspowerpositivehalftwist}
In this subsection, we will explicitly characterize the Garside normal form of the power of a positive half-twist in the braid group $B_3$, or equivalently, the monodromy of a degree three simple Hurwitz curve in $\mathbb{CP}^2$ with respect to an $A_n$-singularity. We establish this new characterization using our theory of duality of positive braids in Subsection~\ref{ssdualitypositivebraid}. Firstly, we describe the role that our theory of duality plays in the study of the conjugacy classes in $B_3$.

We recall that conjugation by the Garside element in $B_3$ has a simple description (Definition~\ref{dconjugationbyDelta}). In particular, using the Garside normal form in $B_3$, we observe that in order to study conjugation in $B_3$, it suffices to study conjugation of positive braids by positive braids. Furthermore, in order to study the conjugation of a positive braid by a positive braid, it is necessary to study the inverse of a positive braid. If $\rho,\tau\in B_3^{+}$ are positive braids indivisible by $\Delta$, then we recall that $\rho$ is dual to $\tau$ if and only if the Garside normal form of $\tau^{-1}$ is $\Delta^{\omega\left(\tau\right)}\rho$ for some nonnegative integer $\omega\left(\tau\right)\geq 0$ (Definition~\ref{ddual}). 

We will use the properties of duality established in Subsection~\ref{ssdualitypositivebraid} in order to establish an explicit characterization of the conjugacy class of $\sigma_1^n$ in $B_3$ for each positive integer $n\geq 1$ (Corollary~\ref{cpositivehalftwistconjugacy} implies that this is the set of $n$th powers of positive half-twists in $B_3$).

We will also establish an important technical result on the effect of a Hurwitz move on a factorization into powers of positive half-twists in the braid group $B_3$. Indeed, we recall that the local effect of a Hurwitz move on a factorization is $\left(\dots,g_j,g_{j+1},\dots\right)\to \left(\dots,g_jg_{j+1}g_j^{-1},g_j,\dots\right)$. If $g_j$ and $g_{j+1}$ are powers of positive half-twists, then in order to compute the Hurwitz move, we need to compute the conjugate of a power of a positive half-twist by the power of a positive half-twist. We use our theory of duality (specifically, Proposition~\ref{ptechnicaldualityconjugate}) to compute the Garside normal form of this conjugate (Lemma~\ref{lHurwitzmoveGarsidenormalform}).

In Section~\ref{smain}, we will define the complexity of a factorization to be the sum of the absolute values of the Garside powers of the factors. We conclude this subsection with a technical result on the change in complexity of a factorization after the application of a Hurwitz move. We will use this technical result extensively in the classification of factorizations of $\Delta^2$ into powers of positive half-twists in Section~\ref{smain}. 

Indeed, a key step in the classification is to establish the existence of complexity-reducing Hurwitz moves for factorizations of positive complexity (in the case of factorizations into positive half-twists), factorizations of complexity greater than one (in the case of factorizations into powers of positive half-twists where there is not the fourth power of a positive half-twist), or factorizations of complexity greater than two (in the case of factorizations into powers of positive half-twists where there is the fourth power of a positive half-twist). Ultimately, we will use this step to reduce the classification of factorizations of $\Delta^2$ into powers of positive half-twists to the classification of such factorizations of complexity either $0$ or $1$ (or $2$ if there is the fourth power of a positive half-twist). The latter step is much more straightforward.

Let us now state and prove the explicit characterization of the Garside normal form of a power of a positive half-twist.

\begin{theorem}
\label{tGarsidenormalformpositivehalftwist}
Let $g\in B_3$ be the $e$th power of a positive half-twist for a positive integer $e\geq 1$. The Garside normal form of $g$ is \[g = \Delta^{-\omega\left(\tau\right)}\rho\sigma_i^{e}\tau,\] where $\rho,\tau\in B_3^{+}$ are positive braids such that $\rho$ is dual to $\tau$, and $i\in \{1,2\}$.
\end{theorem}
\begin{proof}
Corollary~\ref{cpositivehalftwistconjugacy} implies that $g$ is conjugate to $\sigma_1^{e}$. Theorem~\ref{Garside} (Garside's theorem) implies that $g=\left(\Delta^{l}\tau\right)^{-1}\sigma_1^{e}\left(\Delta^{l}\tau\right)$, for a positive braid $\tau\in B_3^{+}$ indivisible by $\Delta$, and an integer $l\in \mathbb{Z}$. However, we can write \[g = \left(\Delta^{l}\tau\right)^{-1}\sigma_1^{e}\left(\Delta^{l}\tau\right) = \tau^{-1}\left(\sigma_1\right)_{\left(l\right)}^{e}\tau = \Delta^{-\omega\left(\tau\right)}\rho\left(\sigma_1\right)_{\left(l\right)}^{e}\tau,\] where $\rho$ is dual to $\tau$. The expression is the Garside normal form of $g$ if and only if $\rho\left(\sigma_1\right)_{\left(l\right)}^{e}\tau$ is indivisible by $\Delta$. We claim that if $\omega\left(\tau\right)$ is minimal among all such expressions of $g$, then this expression is the Garside normal form of $g$. 

Indeed, let $i\in \{1,2\}$ be such that $\sigma_i = \left(\sigma_1\right)_{\left(l\right)}$. Lemma~\ref{lcharacterizationdualpositivebraids} implies that the index of the last generator in $\rho$ is different from the index of the first generator in $\tau$. Proposition~\ref{piDB3} implies that $\rho\sigma_i^{e}\tau$ is divisible by $\Delta$ if and only if either the product of the last two generators in $\rho$ is $\sigma_{i}\sigma_{i'}$ or the product of the first two generators in $\tau$ is $\sigma_{i'}\sigma_i$ for $i\neq i'\in \{1,2\}$. Lemma~\ref{lcharacterizationdualpositivebraids} implies that the first generator in $\tau$ is $\sigma_i$ in the former case, and the last generator in $\rho$ is $\sigma_i$ in the latter case. In the former case, we set $\rho' = \left(\rho\left(\sigma_i\sigma_{i'}\right)^{-1}\right)_{\left(1\right)}$ and $\tau' = \left(\sigma_i\right)^{-1}\tau$. In the latter case, we set $\rho' = \left(\rho\sigma_i^{-1}\right)_{\left(1\right)}$ and $\tau' = \left(\sigma_{i'}\sigma_i\right)^{-1}\tau$. In either case, $\rho',\tau'\in B_3^{+}$ are positive braids indivisible by $\Delta$ such that $\rho'$ is dual to $\tau'$, and Proposition~\ref{ppdb} implies that $\omega\left(\tau'\right) +1 = \omega\left(\tau\right)$. We can now write \[g = \Delta^{-\omega\left(\tau'\right)}\rho'\sigma_{i}^{e}\tau'\] in the former case and \[g= \Delta^{-\omega\left(\tau'\right)}\rho'\sigma_{i'}^{e}\tau'\] in the latter case. Of course, this contradicts the minimality of $\omega\left(\tau\right)$. Therefore, the statement is established. 
\end{proof}

We now characterize the powers of standard positive half-twists ($\sigma_i^{e}$ for some $i\in \{1,2\}$ and positive integer $e\geq 1$) among the powers of all positive half-twists.

\begin{corollary}
\label{cGarsidepowerpositivehalftwist}
If $g\in B_3$ is the power of a positive half-twist, then $g$ is the power of a standard positive half-twist if and only if the Garside power of $g$ is zero.
\end{corollary}
\begin{proof}
Theorem~\ref{tGarsidenormalformpositivehalftwist} states that the Garside normal form of the positive half-twist $g$ can be expressed as \[g = \Delta^{-\omega\left(\tau\right)}\rho\sigma_i^{e}\tau,\] where $\rho,\tau\in B_3^{+}$ are positive braids such that $\rho$ is dual to $\tau$, $i\in \{1,2\}$, and $e\geq 1$ is a positive integer. If $\tau\in B_3^{+}$ is a positive braid, then $\omega\left(\tau\right)\geq 0$ with equality if and only if $\tau$ is the identity. Furthermore, $\rho$ is the identity if and only if $\tau$ is the identity. Therefore, the statement is established.
\end{proof}

We observed special properties of the Garside normal form of the power of a positive half-twist in the proof of Theorem~\ref{tGarsidenormalformpositivehalftwist}. The properties will be important in the sequel and we isolate them as separate statements.

\begin{proposition}
\label{pcutclosureconjugate}
Let $g = \Delta^{-\omega\left(\tau\right)}\rho\sigma_{i}^e\tau$ be the Garside normal form of the $e$th power of a non-standard positive half-twist, where $\rho,\tau\in B_3^{+}$ are positive braids such that $\rho$ is dual to $\tau$, $i\in \{1,2\}$, and $e\geq 1$ is a positive integer (Theorem~\ref{tGarsidenormalformpositivehalftwist}). Let us write $\xi = \rho\sigma_i^{e}\tau$ and observe that $\xi\in B_3^{+}$ is a positive braid indivisible by $\Delta$, of which $\rho$ is a left-divisor and $\tau$ is a right-divisor. 

In this case, either $\rho$ ends with an isolated generator or $\tau$ begins with an isolated generator, but not both. In the former case, the cut closures $\overline{\rho} = \rho$ and $\overline{\tau} = \sigma_i\tau$. In the latter case, the cut closures $\overline{\rho} = \rho\sigma_i$ and $\overline{\tau} = \tau$. In general, $\xi = \overline{\rho}\sigma_i^{e-1}\overline{\tau}$.
\end{proposition}
\begin{proof}
The statement is a consequence of Definition~\ref{dcutclosure} (the definition of a cut closure) and the combination of Proposition~\ref{piDB3} and Lemma~\ref{lcharacterizationdualpositivebraids}. 
\end{proof}

We consider another property of the Garside normal form of the power of a positive half-twist. 

\begin{proposition}
\label{pdivisorconjugatedual}
Let $g = \Delta^{-\omega\left(\tau\right)}\rho\sigma_{i}^e\tau$ be the Garside normal form of the $e$th power of a positive half-twist, where $\rho,\tau\in B_3^{+}$ are positive braids such that $\rho$ is dual to $\tau$, $i\in \{1,2\}$, and $e\geq 1$ is a positive integer (Theorem~\ref{tGarsidenormalformpositivehalftwist}). Let us write $\xi = \rho\sigma_i^{e}\tau$ and observe that $\xi\in B_3^{+}$ is a positive braid indivisible by $\Delta$, of which $\rho$ is a left-divisor and $\tau$ is a right-divisor. 

If $\rho'$ is a cut left-divisor of the cut closure $\overline{\rho}$, then $\left(\rho'\right)_{\left(\omega\left(\tau\right)-\omega\left(\rho'\right)\right)}$ is left-dual to a right-divisor of $\tau$. If $\tau'$ is a cut right-divisor of the cut closure $\overline{\tau}$, then $\left(\tau'\right)_{\left(\omega\left(\tau\right)-\omega\left(\tau'\right)\right)}$ is right-dual to a left-divisor of $\rho$.
\end{proposition}
\begin{proof}
We prove the first statement since the proof of the second statement is analogous. If $\rho'$ is a left-divisor of $\rho$, then the statement follows from Proposition~\ref{ppseudosymmetryduality} and Proposition~\ref{ppdb}. Indeed, Proposition~\ref{pcutdivisorsubdivisor} implies that $\rho'$ is a cut left-divisor of $\rho$ in this case. 

If $\rho'$ is not a left-divisor of $\rho$, then $\rho' = \overline{\rho}$ and $\rho$ is a proper left-divisor of $\overline{\rho}$. If $i\neq i'\in \{1,2\}$, then Proposition~\ref{pcutclosureconjugate} implies that $\overline{\rho}$ ends with $\sigma_{i'}\sigma_i$ and $\rho$ ends with $\sigma_{i'}$. Proposition~\ref{piDB3} implies that the last $\sigma_{i'}$ in $\rho$ is non-isolated, since $\xi$ is indivisible by $\Delta$. Lemma~\ref{lcharacterizationdualpositivebraids} implies that $\tau$ begins with a product of distinct generators $\sigma_i\sigma_{i'}$, since $\rho$ is dual to $\tau$. A further application of Lemma~\ref{lcharacterizationdualpositivebraids} now implies that $\overline{\rho}$ is dual to the right divisor $\sigma_{i}^{-1}\tau\in B_3^{+}$ of $\tau$. Therefore, the statement is established.
\end{proof}

If $g\in B_3$ is the power of a positive half-twist with Garside normal form $g = \Delta^{k}\xi$, then there is a special relationship between $k$ and $\omega\left(\xi\right)$. 

\begin{proposition}
\label{pomeganumberpowerpositivehalftwist}
Let $g\in B_3$ be the $e$th power of a non-standard positive half-twist, where $e\geq 1$ is a positive integer. If $g = \Delta^{k}\xi$ is the Garside normal form of $g$, then $\omega\left(\xi\right) = -2k + e - 1$.
\end{proposition}
\begin{proof}
Theorem~\ref{tGarsidenormalformpositivehalftwist} implies that $k=-\omega\left(\tau\right)$ and $\xi = \rho\sigma_i^{e}\tau$, where $\rho,\tau\in B_3^{+}$ are positive braids such that $\rho$ is dual to $\tau$, and $i\in \{1,2\}$. Proposition~\ref{pcutclosureconjugate} implies that $\xi = \overline{\rho}\sigma_i^{e-1}\overline{\tau}$, where $\overline{\rho}$ is a cut left-divisor of $\overline{\rho}\sigma_i^{e-1}$ and $\overline{\tau}$ is a cut right-divisor of $\xi$. Furthermore, Lemma~\ref{lcharacterizationdualpositivebraids} implies that $\omega\left(\overline{\rho}\right) = \omega\left(\rho\right)$, $\omega\left(\overline{\tau}\right) = \omega\left(\tau\right)$, and $\omega\left(\sigma_i^{e-1}\right) = e-1$. Of course, $\omega\left(\rho\right) = \omega\left(\tau\right)$ according to the pseudo-symmetry of duals (Proposition~\ref{ppseudosymmetryduality}). Finally, Proposition~\ref{ppdb} implies that \[\omega\left(\xi\right) = \omega\left(\overline{\rho}\right) + \omega\left(\sigma_i^{e-1}\right) + \omega\left(\overline{\tau}\right) = 2\omega\left(\tau\right) + e - 1.\] Therefore, the statement is established.
\end{proof}

We remark that the statement in Proposition~\ref{pomeganumberpowerpositivehalftwist} is not true if $g$ is the $e$th power of a standard positive half-twist. Indeed, $\omega\left(\sigma_i^{e}\right) = e$ if $i\in \{1,2\}$ and $e\geq 1$ is a positive integer.

Finally, we determine the Garside normal form and, in particular, the Garside power, of the inverse of a non-standard positive half-twist. (If $i\in \{1,2\}$, then the Garside normal form of the inverse $\sigma_i^{-1}$ of the standard positive half-twist $\sigma_i$ is $\Delta^{-1}\sigma_{i}\sigma_{i'}$, where $i\neq i'\in \{1,2\}$.) We will not use this statement in the paper, and the reader may safely skip it. However, we include it for future reference since negative powers of positive half-twists appear naturally in the theory of braid monodromy of Hurwitz curves. For example, the braid monodromy with respect to a negative node of a Hurwitz curve is the square of the inverse of a positive half-twist.

The method of the proofs of the classification results of factorizations of $\Delta^2$ into products of positive powers of positive half-twists in Section~\ref{smain}, should generalize to the case where the factors are potentially negative powers of positive half-twists. The generalization would be based on the following statement.

\begin{theorem}
\label{tGarsidenormalforminversepositivehalftwist}
Let $g$ be a non-standard positive half-twist and write \[g = \Delta^{-\omega\left(\tau\right)}\rho\sigma_i\tau\] for the Garside normal form of $g$, where $\rho,\tau\in B_3^{+}$ are positive braids such that $\rho$ is dual to $\tau$, and $i\in \{1,2\}$ (Theorem~\ref{tGarsidenormalformpositivehalftwist}). Let $\rho'\in B_3^{+}$ be the cut left-divisor of $\rho$ and let $\tau'\in B_3^{+}$ be the cut right-divisor of $\tau$ such that $\omega\left(\rho'\right) = \omega\left(\rho\right) - 1 = \omega\left(\tau\right) - 1 = \omega\left(\tau'\right)$. The Garside normal form of the negative half-twist $g^{-1}$ is \[g^{-1} = \Delta^{-\omega\left(\tau\right)}\rho'\sigma_{i'}^2\tau',\] where $\rho',\tau'\in B_3^{+}$ are positive braids such that $\rho'$ is dual to $\left(\tau'\right)_{\left(1\right)}$, and $i\neq i'\in \{1,2\}$. 
\end{theorem}
\begin{proof}
Firstly, Lemma~\ref{lcharacterizationdualpositivebraids} implies that $\rho'$ is dual to $\left(\tau'\right)_{\left(1\right)}$. Proposition~\ref{pcutclosureconjugate} and Lemma~\ref{lcharacterizationdualpositivebraids} imply that $\rho'\sigma_{i'}$ is dual to $\overline{\tau}$ and $\overline{\rho}$ is dual to $\sigma_{i'}\tau'$. The pseudo-symmetry of duals (Proposition~\ref{ppseudosymmetryduality}) implies that $\sigma_{i'}\tau'$ is dual to $\left(\overline{\rho}\right)_{\left(\omega\left(\tau\right)\right)}$. Proposition~\ref{pdualbraidproduct} implies that $\rho'\sigma_{i'}\sigma_{i'}\tau'$ is indivisible by $\Delta$ and Proposition~\ref{pcutclosureconjugate} implies that $\rho\sigma_{i}\tau = \overline{\rho}\overline{\tau}$. Finally, Proposition~\ref{ppdb} implies that $\rho'\sigma_{i'}^2\tau'$ is dual to $\left(\overline{\rho}\overline{\tau}\right)_{\left(\omega\left(\tau\right)\right)}$, and $\omega\left(\overline{\rho}\overline{\tau}\right) = 2\omega\left(\tau\right)$. Therefore, the statement is established.
\end{proof}

We observe that there is a deceptive similarity between the Garside normal form of a non-standard negative half-twist (the inverse of a non-standard positive half-twist) in the statement of Theorem~\ref{tGarsidenormalforminversepositivehalftwist}, and the Garside normal form of the square of a positive half-twist in Theorem~\ref{tGarsidenormalformpositivehalftwist}. However, in the statement of Theorem~\ref{tGarsidenormalforminversepositivehalftwist}, $g^{-1}$ is not a square of a positive half-twist, since $\omega\left(\tau\right) = \omega\left(\tau'\right) + 1 > \omega\left(\tau'\right)$ and $\rho'$ is not dual to $\tau'$. On the other hand, if $\omega\left(\tau'\right)$ appeared in the exponent of $\Delta$ instead of $\omega\left(\tau\right)$, and if $\rho'$ were dual to $\tau'$, then the expression would be the Garside normal form of the square of a positive half-twist.

We have earlier mentioned that we can view the Garside power of a braid as a measure of the complexity of the braid. The following statement is very useful because inverses of braids naturally appear in conjugation.

\begin{proposition}
\label{pGarsidepowerinversepowerpositivehalftwist}
Let $g$ be the $e$th power of a non-standard positive half-twist for a positive integer $e\geq 1$. If the Garside power of $g$ is $k$, then the Garside power of $g^{-1}$ is $k-e+1$. In particular, if $g$ is a non-standard positive half-twist, then the Garside power of $g^{-1}$ is equal to the Garside power of $g$.
\end{proposition}
\begin{proof}
If $g = \Delta^{k}\xi$ is the Garside normal form of $g$, then Proposition~\ref{pomeganumberpowerpositivehalftwist} implies that $\omega\left(\xi\right) = -2k + e - 1$. If $\chi$ is dual to $\xi$, in which case the Garside normal form of $\xi^{-1}$ is $\xi^{-1} = \Delta^{-\omega\left(\xi\right)}\chi$, then the Garside normal form of $g^{-1}$ is $g^{-1} = \Delta^{-k-\omega\left(\xi\right)}\chi_{\left(k\right)}$. We conclude that the Garside power of $g^{-1}$ is $-k-\omega\left(\xi\right) = k - e + 1$. Therefore, the statement is established.
\end{proof}

The following setup will be extremely prominent in Section~\ref{smain}. Indeed, we will often begin statements in Section~\ref{smain} with ``We adopt Setup~\ref{setfactorization}..." in order to avoid repetition in the statements. The setup fixes notation whenever we consider a factorization into powers of positive half-twists in the braid group $B_3$.

\begin{setup}
\label{setfactorization}
Let $\{g_1,g_2,\dots,g_k\}$ be a set of powers of positive half-twists in $B_3$, where $g_j$ is the $e_j$th power of a positive half-twist and $e_j\geq 1$ is a positive integer for each $1\leq j\leq k$. We denote ${\cal F} = \left(g_1,g_2,\dots,g_k\right)$ for the corresponding factorization of $g_1g_2\cdots g_k$. Theorem~\ref{tGarsidenormalformpositivehalftwist} implies that we can write the Garside normal form of $g_j$ as \[g_j = \Delta^{-\omega\left(\tau_j\right)}\left(\rho_j\sigma_{i_j}^{e_j}\tau_j\right)_{\left(\omega_j\right)},\] where $\rho_j,\tau_j\in B_3^{+}$ are positive braids such that $\rho_j$ is dual to $\tau_j$, $i_j\in \{1,2\}$, and $\omega_j = \sum_{j'=1}^{j} \omega\left(\tau_{j'}\right)$. Let us write $\xi_j = \rho_j\sigma_{i_j}^{e_j}\tau_j$. We observe that $\xi_j\in B_3^{+}$ is a positive braid indivisible by $\Delta$ since $g_j = \Delta^{-\omega\left(\tau_j\right)}\left(\xi_j\right)_{\left(\omega_j\right)}$ is the Garside normal form of $g_j$. Furthermore, $g_1g_2\cdots g_k = \xi_1\xi_2\cdots \xi_k \Delta^{-\omega_k}$. 

Let $g_j^{-1} = \Delta^{-\omega\left(\tau_j\right)-e_j+1}\left(\chi_j\right)_{\left(\omega_j+\omega\left(\tau_j\right)+e_j-1\right)}$ be the Garside normal form of $g_j^{-1}$, where $\chi_j\in B_3^{+}$ is the positive braid right-dual to $\xi_j$. (The explicit value of the Garside power of $g_j^{-1}$ is established in Proposition~\ref{pGarsidepowerinversepowerpositivehalftwist}).
\end{setup}

The motivation underlying Setup~\ref{setfactorization} is that statements concerning the factorization $\left(g_1,g_2,\dots,g_k\right)$ of the product $g_1g_2\cdots g_k$ in the braid group $B_3$ can be converted to statements concerning the factorization $\left(\xi_1,\xi_2,\cdots,\xi_k\right)$ of the product $\xi_1\xi_2\cdots \xi_k$ in the monoid $B_3^{+}$ of positive braids in $B_3$. Indeed, Setup~\ref{setfactorization} achieves this using the Garside normal form. We will state and prove the results in Section~\ref{smain} concerning factorizations of $\Delta^2$ into powers of positive half-twists by converting them into statements concerning factorizations of powers of $\Delta$ in the monoid $B_3^{+}$. Finally, we will characterize such factorizations in the monoid $B_3^{+}$ by applying our theory of duality of positive braids in Subsection~\ref{ssdualitypositivebraid}. Indeed, the results in Subsection~\ref{ssdualitypositivebraid} are well-equipped to study equations of the form $\Delta^{l} = \xi_1\xi_2\cdots \xi_k$ for positive braids $\xi_1,\xi_2,\dots,\xi_k\in B_3^{+}$ indivisible by $\Delta$ and some positive integer $l\geq 1$. If $k=2$, then this is simply the equation of duality.

We conclude this subsection with a technical result that will be important in the analysis of Hurwitz moves applied to factorizations of $\Delta^2$ into powers of positive half-twists. The result is based on our theory of duality of positive braids, specifically Proposition~\ref{ptechnicaldualityconjugate} in Subsection~\ref{ssdualitypositivebraid}.

\begin{lemma}
\label{lHurwitzmoveGarsidenormalform}
We adopt Setup~\ref{setfactorization} with $k=2$. Let $\xi_1''\in B_3^{+}$ be the maximal right-divisor of $\xi_1$ left-dual to a left-divisor of $\xi_2$. Let us assume that the relevant left-divisor $\rho_2''$ of $\xi_2$ is a left-divisor of the the cut closure $\overline{\rho_2}$. Let $\chi_1''$ be the cut left-divisor of $\chi_1$ that is right-dual to $\xi_1''$ and let $\tau_2''$ be the right-divisor of $\tau_2$ that is right-dual to $\left(\overline{\rho_2''}\right)_{\left(\omega\left(\rho_2\right)-\omega\left(\rho_2''\right)\right)}$. Let us write $\xi_1 = \xi_1'\xi_1''$, $\rho_2 = \rho_2''\rho_2'$, $\chi_1=\chi_1''\chi_1'$, and $\tau_2 = \tau_2'\tau_2''$ for positive braids $\xi_1',\rho_2',\chi_1',\tau_2'\in B_3^{+}$.
\begin{description}
\item[(i)] Let us assume that $\rho_2''$ is a cut left-divisor of $\rho_2$. In this case, $\tau_2'$ ends with a product of two distinct generators $\alpha\in B_3^{+}$. Let $\beta\in B_3^{+}$ be the single generator at the beginning of $\chi_1'$. The Garside normal form of $g_1g_2g_1^{-1}$ is \[\xi_1'\left(\rho_2'\sigma_{i_2}^{e_2}\left(\tau_2'\alpha^{-1}\right)\right)_{\left(\omega\left(\xi_1''\right)\right)}\left(\beta^{-1}\chi_1'\right)_{\left(\omega\left(\tau_2\right)+1\right)}\Delta^{-\omega\left(\tau_2\right) - e_1+1 +2\left(\omega\left(\xi_1''\right)-\omega\left(\tau_1\right)\right)+1}.\]
\item[(ii)] If $\rho_2'' = \overline{\rho_2}\neq \rho_2$, then the Garside normal form of $g_1g_2g_1^{-1}$ is \[\xi_1'\left(\sigma_{i_2}^{e_2}\right)_{\left(\omega\left(\xi_1''\right)\right)}\left(\chi_1'\right)_{\left(\omega\left(\tau_2\right)\right)}\Delta^{-\omega\left(\tau_2\right) - e_1+1 +2\left(\omega\left(\xi_1''\right)-\omega\left(\tau_1\right)\right)}.\] If $\rho_2'' = \rho_2$, then the Garside normal form of $g_1g_2g_1^{-1}$ is either \[\xi_1'\left(\sigma_{i_2}^{e_2}\right)_{\left(\omega\left(\xi_1''\right)\right)}\left(\chi_1'\right)_{\left(\omega\left(\tau_2\right)\right)}\Delta^{-\omega\left(\tau_2\right)-e_1+1+2\left(\omega\left(\xi_1''\right)-\omega\left(\tau_1\right)\right)},\] or the Garside power of $g_1g_2g_1^{-1}$ is greater than $-\omega\left(\tau_2\right)-e_1+1+2\left(\omega\left(\xi_1''\right)-\omega\left(\tau_1\right)\right)$. 
\item[(iii)] Let us assume that $\rho_2''$ is a left-divisor of $\rho_2$ but that $\rho_2''$ is not a cut left-divisor of $\rho_2$. Let $\alpha$ be the first generator in $\tau_2''$ and let $\beta$ be the last generator in $\chi_2''$. The Garside normal form of $g_1g_2g_1^{-1}$ is \[\xi_1'\left(\rho_2'\sigma_{i_2}^{e_2}\left(\tau_2'\alpha\right)\right)_{\left(\omega\left(\xi_1''\right)\right)}\left(\beta\chi_1'\right)_{\left(\omega\left(\tau_2\right)+1\right)}\Delta^{-\omega\left(\tau_2\right) - e_1+1 +2\left(\omega\left(\xi_1''\right)-\omega\left(\tau_1\right)\right)-1}.\] 
\end{description}
\end{lemma}
\begin{proof}
In the notation of Setup~\ref{setfactorization}, the Garside normal form of $g_1g_2g_1^{-1}$ is 
\begin{align*}
g_1g_2g_1^{-1} &= \left(\Delta^{-\omega\left(\tau_1\right)}\left(\xi_1\right)_{\left(\omega_1\right)}\right)\left(\Delta^{-\omega\left(\tau_2\right)}\left(\xi_2\right)_{\left(\omega_2\right)}\right)\left(\left(\chi_1\right)_{\left(\omega_1\right)}\Delta^{-\omega\left(\tau_1\right)-e_1+1}\right) \\ &= \xi_1\xi_2\left(\chi_1\right)_{\left(\omega\left(\tau_2\right)\right)}\Delta^{-\omega\left(\tau_2\right) - e_1+1 -2\omega\left(\tau_1\right)}.
\end{align*} Of course, Proposition~\ref{pcutclosureconjugate} implies that $\xi_2 = \overline{\rho_2}\sigma_{i_2}^{e_2-1}\overline{\tau_2}$. We can use Proposition~\ref{pGarsidenormalformproduct} in order to compute the Garside normal form of $\xi_1\overline{\rho_2}$. We can use Proposition~\ref{ptechnicaldualityconjugate} and Proposition~\ref{pGarsidenormalformproduct} in order to compute the Garside normal form of $\overline{\tau_2}\left(\chi_1\right)_{\left(\omega\left(\tau_2\right)\right)}$. Finally, we can use Lemma~\ref{ltripleproductindivisibleDelta} in order to compute the Garside normal form of the product $\xi_1\overline{\rho_2}\sigma_{i_2}^{e_2-1}\overline{\tau_2}\left(\chi_1\right)_{\left(\omega\left(\tau_2\right)\right)}$, and the statement follows from this computation. 
\end{proof}

In Section~\ref{smain}, we will define the complexity of a factorization to be the sum of the absolute values of the Garside powers of the individual factors. We will use the following statement frequently in order to determine the change in complexity of a factorization into powers of positive half-twists, after the application of a Hurwitz move.

\begin{lemma}
\label{lHurwitzmoveGarsidepower}
We adopt the setting of Lemma~\ref{lHurwitzmoveGarsidenormalform}. The Garside power of $g_1g_2g_1^{-1}$ is $-\omega\left(\tau_2\right)-e_1+1+2\left(\omega\left(\xi_1''\right)-\omega\left(\tau_1\right)\right) -\epsilon$, where $\epsilon = -1$ in case \textbf{(i)}, $\epsilon\leq 0$ in case \textbf{(ii)}, and $\epsilon = 1$ in case \textbf{(iii)}, in Lemma~\ref{lHurwitzmoveGarsidenormalform}. In particular, if $k$ is the Garside power of $g_2$, then the Garside power of $g_1g_2g_1^{-1}$ is $k-e_1+1+2\left(\omega\left(\xi_1''\right)-\omega\left(\tau_1\right)\right)-\epsilon$.
\end{lemma}
\begin{proof}
The statement is an immediate consequence of Lemma~\ref{lHurwitzmoveGarsidenormalform}.
\end{proof}

\section{The classification of degree three simple Huwritz curves up to isotopy and braid monodromy factorizations up to Hurwitz equivalence}
\label{smain}

In this section, we establish the main results on the classification of isotopy classes of degree three simple Hurwitz curves in $\mathbb{CP}^2$. We use our theory developed in Section~\ref{sB3background} in order to establish the classification. Indeed, the results are equivalent to the classification of factorizations of $\Delta^2$ into powers of positive half-twists up to Hurwitz equivalence in the braid group $B_3$. The outline of this section is as follows.

In Subsection~\ref{ssstandardfactorizations}, we will define the \textit{standard factorizations} of $\Delta^2$ into powers of positive half-twists (Definition~\ref{dstandardfactorization}). In Subsections~\ref{ssclassificationfactorizationspositivehalftwists} -~\ref{sssingularityconstraints}, we will establish that every factorization of $\Delta^2$ into powers of positive half-twists is Hurwitz equivalent to one of the standard factorizations in Definition~\ref{dstandardfactorization}.

Firstly, in Subsection~\ref{ssclassificationfactorizationspositivehalftwists}, we will establish that every factorization of $\Delta^2$ into positive half-twists is Hurwitz equivalent to the standard factorization $\Delta^2\equiv \left(\sigma_1,\sigma_2,\sigma_1,\sigma_1,\sigma_2,\sigma_1\right)$ (Theorem~\ref{tmainfactorizationpositivehalftwists}). Equivalently, every pair of smooth degree three Hurwitz curves in $\mathbb{CP}^2$ are isotopic (a special case of Theorem~\ref{tmainHurwitzcurve}). We will also establish technical statements on factorizations that we will use in Subsection~\ref{ssclassificationfactorizationspowerspositivehalftwists}.

The \textit{type of a factor} in a factorization of $\Delta^2$ into powers of positive half-twists is the exponent $n$ such that the factor is the $n$th power of a positive half-twist. In Section~\ref{ssclassificationfactorizationspowerspositivehalftwists}, we will show that the number of factors of each type in a factorization of $\Delta^2$ into powers of positive half-twists is a complete invariant of the Hurwitz equivalence class of the factorization. The statement is equivalent to Theorem~\ref{tmainHurwitzcurve} in the Introduction.

Finally, in Subsection~\ref{sssingularityconstraints}, we will establish Theorem~\ref{tmainsingularities} in the Introduction, which constrains the possible singularities of a degree three simple Hurwitz curve in $\mathbb{CP}^2$. Indeed, we will establish a complete set of constraints on the types of factors in a factorization of $\Delta^2$ into powers of positive half-twists.
 
The strategy of the proofs in this section is concisely as follows. We will assign a notion of complexity to factorizations of $\Delta^2$ into powers of positive half-twists. We will use Lemma~\ref{lHurwitzmovecomplexitychange} (a reformulation of Lemma~\ref{lHurwitzmoveGarsidepower}) to apply complexity reducing Hurwitz moves to such factorizations, and reduce to the problem of classifying low complexity (zero or one) factorizations of $\Delta^2$.

\subsection{The standard factorizations of $\Delta^2$ into powers of positive half-twists}
\label{ssstandardfactorizations}
The goal of this subsection is to establish simple statements on factorizations of $\Delta^2$ into powers of positive half-twists in the braid group $B_3$. We will also highlight the connections between these statements and the theory of simple Hurwitz curves of degree three in $\mathbb{CP}^2$.

The following terminology concerning factorizations in the braid group $B_3$ will be convenient in the sequel.

\begin{definition}
Let ${\cal F} = \left(g_1,g_2,\dots,g_k\right)$ be a factorization into a product of powers of positive half-twists. The \textit{type of a factor} $g_j$ in ${\cal F}$ is the positive integer $e_j$ such that $g_j$ is the $e_j$th power of a positive half-twist. 
\end{definition}

Proposition~\ref{pHurwitzequivalencefunctorial} implies that the numbers of factors of each type in a factorization ${\cal F}$ into powers of positive half-twists is an invariant of the Hurwitz equivalence class of ${\cal F}$. The first step is to use a class invariant on the braid group $B_3$ in order to constrain the numbers of factors of each type in a factorization into powers of positive half-twists. If $G$ is a group, then a \textit{class invariant} on $G$ is a function on $G$ that is constant on conjugacy classes in $G$. A simple and interesting class invariant on $G$ is the abelianization homomorphism $G\to H_1\left(G\right) = G/\left[G,G\right]$, where $\left[G,G\right]$ denotes the commutator subgroup of $G$. We note that $H_1\left(B_3\right)\cong \mathbb{Z}$ and we introduce notation to refer to this class invariant on $B_3$.

\begin{definition}
Let $\epsilon:B_3\to \mathbb{Z}$ denote the abelianization homomorphism, which is uniquely defined by the rule $\epsilon\left(\sigma_1\right) = 1 = \epsilon\left(\sigma_2\right)$. The map $\epsilon$ is a class invariant on $B_3$.
\end{definition}

We use the abelianization homomorphism $\epsilon$ in order to establish the following preliminary constraint on factorizations of a braid into powers of positive half-twists.

\begin{proposition}
\label{pnumberoffactorsinfactorization}
Let ${\cal F} = \left(g_1,g_2,\dots,g_k\right)$ be a factorization of $g\in B_3$ into powers of positive half-twists. If $\nu_n$ is the number of factors in ${\cal F}$ that are the $n$th power of a positive half-twist (i.e., the number of factors in ${\cal F}$ of type $n$), then $\epsilon\left(g\right) = \sum_{n=1}^{\epsilon\left(g\right)} n\nu_n$. 

In particular, if $g = \Delta^2$, then $6 = \sum_{n=1}^{6} n\nu_n$. Furthermore, if $g = \Delta^2$ and $g_j$ is a positive half-twist for each $1\leq j\leq k$, then $k = 6$. 
\end{proposition}
\begin{proof}
If $g_j$ is the $n$th power of a positive half-twist, then we have $\epsilon\left(g_j\right) = n$, since $\epsilon:B_3\to \mathbb{Z}$ is a class invariant on $B_3$ and $\epsilon\left(\sigma_1^{n}\right) = n$. If ${\cal F} = \left(g_1,g_2,\dots,g_k\right)$ is a factorization of $g$, then $\epsilon\left(g\right) = \epsilon\left(g_1g_2\cdots g_k\right)$. The first statement follows since $\epsilon:B_3\to \mathbb{Z}$ is a homomorphism. The second and third statements follow from the first statement since $\epsilon\left(\Delta^2\right) = 6$.
\end{proof}

We recall that if $C\subseteq \mathbb{CP}^2$ is a simple Hurwitz curve, then the Hurwitz equivalence class of a braid monodromy factorization of $C$ is an invariant of the isotopy class of $C$. Furthermore, a factor of type $n$ in a braid monodromy factorization of $C$ corresponds to the braid monodromy with respect to an $A_n$-singularity in $C$. In particular, the number of factors of type $n$ in a braid monodromy factorization of $C$ is equal to the number of $A_n$-singularities in $C$. 

The geometric interpretation of Proposition~\ref{pnumberoffactorsinfactorization} is the following formula concerning the numbers of $A_n$-singularities of a degree three simple Hurwitz curve $C\subseteq \mathbb{CP}^2$ for $n\geq 1$ a positive integer. 

\begin{lemma}
\label{lsingularityformula}
Let $C\subseteq \mathbb{CP}^2$ be a degree three simple Hurwitz curve. If $n\geq 1$ is a positive integer, then let $\nu_n$ be the number of $A_n$-singularities in $C$. We have the \textit{singularity formula}: \[\sum_{n=1}^{6} n\nu_n = 6.\]
\end{lemma}
\begin{proof}
A braid monodromy factorization of $C$ is a factorization ${\cal F}$ of $\Delta^2$ into powers of positive half-twists. The number of factors in ${\cal F}$ that are the $n$th power of a positive half-twist is $\nu_n$ by definition of the braid monodromy of $C$. Therefore, the statement is a consequence of Proposition~\ref{pnumberoffactorsinfactorization}. 
\end{proof}

We refer to the formula in the statement of Lemma~\ref{lsingularityformula} as the \textit{singularity formula}. 

The goal of this section is to establish that the numbers of factors of each type in a factorization of $\Delta^2$ into powers of positive half-twists constitute a \textit{complete} set of invariants of the Hurwitz equivalence class of the factorization. Equivalently, if $C\subseteq \mathbb{CP}^2$ is a degree three simple Hurwitz curve, then the numbers of $A_n$-singularities in $C$ for each positive integer $n\geq 1$ is a complete set of invariants of the isotopy class of $C$ (Theorem~\ref{tmainHurwitzcurve} in the Introduction).

Let $\nu_n$ denote the number of $A_n$-singularities in a simple Hurwitz curve $C$. In Subsection~\ref{sssingularityconstraints}, we will show that the singularity formula (Lemma~\ref{lsingularityformula}) is not the only constraint on the singularities of a simple Hurwitz curve $C\subseteq \mathbb{CP}^2$. Indeed, it is not the case that every preassigned tuple $\left(\nu_n\right)_{n=1}^{\infty}$ of nonnegative integers such that $\sum_{n=1}^{6} n\nu_n = 6$ corresponds to a simple Hurwitz curve $C\subseteq \mathbb{CP}^2$. We now define the standard factorizations of $\Delta^2$ into powers of positive half-twists corresponding to those preassigned tuples which are possible on the algebraic level. (The statement that these are all the possible tuples is the content of Theorem~\ref{tmainsingularities} in the Introduction. We will establish this explicitly in Theorem~\ref{tfinal}.)

\begin{definition}
\label{dstandardfactorization}
Let $\left(\nu_n\right)_{n=1}^{\infty}$ be one of the following tuples of nonnegative integers. We define the \textit{standard factorizations} of $\Delta^2$ as follows:
\begin{description}
\item[(i)] If $\nu_1 = 2$, $\nu_4 = 1$ and $\nu_j = 0$ for $j\neq 1,4$, then the standard factorization of $\Delta^2$ into two positive half-twists and one fourth power of a positive half-twist is \[\Delta^2\equiv \left(\sigma_1^4,\sigma_1^{-1}\sigma_2\sigma_1,\sigma_2^{-1}\sigma_1\sigma_2\right).\]
\item[(ii)] If $\nu_2 = 3$ and $\nu_j = 0$ for $j\neq 2$, then the standard factorization of $\Delta^2$ into three squares of positive half-twists is \[\Delta^2\equiv \left(\sigma_1^2,\sigma_1^{-1}\sigma_2^2\sigma_1,\sigma_2^2\right).\]
\item[(iii)] If $\nu_1 = 3$, $\nu_3 = 1$ and $\nu_j = 0$ for $j\neq 1,3$, then the standard factorization of $\Delta^2$ into three positive half-twists and one cube of a positive half-twist is \[\Delta^2\equiv \left(\sigma_1^3,\sigma_1^{-1}\sigma_2\sigma_1,\sigma_1,\sigma_2\right).\]
\item[(iv)] If $\nu_1 = 2$, $\nu_2 = 2$ and $\nu_j = 0$ for $j\neq 1,2$, then the standard factorization of $\Delta^2$ into two positive half-twists and two squares of positive half-twists is \[\Delta^2\equiv \left(\sigma_1^2,\sigma_2,\sigma_1^2,\sigma_2\right).\]
\item[(v)] If $\nu_1 = 4$, $\nu_2 = 1$ and $\nu_j=0$ for $j\neq 1,2$, then the standard factorization of $\Delta^2$ into four positive half-twists and one square of a positive half-twist is \[\Delta^2\equiv \left(\sigma_1^2,\sigma_2,\sigma_1,\sigma_1,\sigma_2\right).\]
\item[(vi)] If $\nu_1 = 6$ and $\nu_j=0$ for $j\neq 1$, then the \textit{standard factorization of $\Delta^2$} into six positive half-twists is \[\Delta^2 \equiv \left(\sigma_1,\sigma_2,\sigma_1,\sigma_1,\sigma_2,\sigma_1\right).\]
\end{description}
\end{definition}

In Subsection~\ref{sssingularityconstraints}, we will establish that there are no other tuples $\left(\nu_n\right)_{n=1}^{\infty}$ corresponding to a factorization of $\Delta^2$ into powers of positive half-twists. The proof will be based on our theory of duality of positive braids and the algebraic characterization of powers of positive half-twists in $B_3$ in Section~\ref{sB3background}. However, in the following Subsection~\ref{ssclassificationfactorizationspositivehalftwists} and Subsection~\ref{ssclassificationfactorizationspowerspositivehalftwists}, we will show that every factorization ${\cal F}$ of $\Delta^2$ into powers of positive half-twists is Hurwitz equivalent to the unique standard factorization in Definition~\ref{dstandardfactorization} with the same numbers of factors of each type as the factorization ${\cal F}$.

The strategy is to assign a notion of complexity to a factorization and establish the statement by induction on the complexity. If $g\in B_n$, then we recall that the \textit{Garside power} of $g$ is the exponent of the Garside element $\Delta$ appearing in the Garside normal form of $g$.

\begin{definition}
\label{dcomplexityfactorization}
If ${\cal F}=\left(g_1,g_2,\dots,g_k\right)$ is a factorization of $g_1g_2\cdots g_k$, then the \textit{complexity of the factorization} ${\cal F}$ is the sum of the absolute values of the Garside powers of the factors $g_1,g_2,\dots,g_k$. We denote the complexity of ${\cal F}$ by $c\left({\cal F}\right)$.
\end{definition}

The following statement is a useful observation in our approach to classifying factorizations into powers of positive half-twists by induction on the complexity. 

\begin{proposition}
\label{pcomplexityzeroiffstandard}
A factorization ${\cal F}$ into powers of positive half-twists has complexity $c\left({\cal F}\right) = 0$ if and only if ${\cal F}$ is a factorization into powers of standard positive half-twists. 
\end{proposition}
\begin{proof}
The statement follows from Corollary~\ref{cGarsidepowerpositivehalftwist}, which states that the absolute value of the Garside power of the power of a positive half-twist is zero if and only if the relevant positive half-twist is a standard positive half-twist.
\end{proof}

In Definition~\ref{dstandardfactorization}, the standard factorization in \textbf{(i)} has complexity two, the standard factorizations in \textbf{(ii)} - \textbf{(iii)} have complexity one, and the standard factorizations in \textbf{(iv)} - \textbf{(vi)} have complexity zero. In Subsection~\ref{ssclassificationfactorizationspositivehalftwists}, we will show that if ${\cal F}$ is a factorization of $\Delta^2$ into positive half-twists with $c\left({\cal F}\right)>0$, then we can always apply a complexity reducing Hurwitz move to ${\cal F}$. In Subsection~\ref{ssclassificationfactorizationspowerspositivehalftwists}, we will show that if ${\cal F}$ is a factorization of $\Delta^2$ into at least one positive half-twist and squares of positive half-twists with $c\left({\cal F}\right)>0$, then we can always apply a complexity reducing Hurwitz move to ${\cal F}$. In Subsection~\ref{ssclassificationfactorizationspowerspositivehalftwists}, we will also show that if ${\cal F}$ is a factorization of $\Delta^2$ into powers of positive half-twists with $c\left({\cal F}\right)>1$ such that there is no fourth power of a positive half-twist, then we can always apply a complexity reducing Hurwitz move to ${\cal F}$. Finally, in Subsection~\ref{ssclassificationfactorizationspowerspositivehalftwists}, we will show that if ${\cal F}$ is a factorization of $\Delta^2$ into powers of positive half-twists with $c\left({\cal F}\right)>2$ such that there is the fourth power of a positive half-twist, then we can always apply a complexity reducing Hurwitz move to ${\cal F}$.

The main ingredient in our proofs is the following computation of the effect of a Hurwitz move on the complexity of a factorization. We recall Setup~\ref{setfactorization}, which we will use prominently in this section in order to avoid repetition in the statements.

\begin{lemma}
\label{lHurwitzmovecomplexitychange}
We adopt Setup~\ref{setfactorization} with $k=2$. Let $\xi_1''\in B_3^{+}$ be the maximal right-divisor of $\xi_1$ that is left-dual to a left-divisor of $\xi_2$. Let us assume that the relevant left-divisor $\rho_2''$ of $\xi_2$ is a left-divisor of the the cut closure $\overline{\rho_2}$.

If ${\cal F}' = \left(g_1g_2g_1^{-1},g_1\right)$ is the result of a Hurwitz move applied to the factorization ${\cal F} = \left(g_1,g_2\right)$, then we have the following formula relating the complexity of ${\cal F}'$ to the complexity of ${\cal F}$: \[c\left({\cal F}'\right) = c\left({\cal F}\right) + e_1 - 1 - 2\left[\omega\left(\xi_1''\right)-\omega\left(\tau_1\right)\right] + \epsilon,\] where $\epsilon\leq 1$. Furthermore, $\epsilon = 1$ if and only if $\rho_2''$ is a left-divisor of $\rho_2$ but $\rho_2''$ is not a cut left-divisor of $\rho_2$.
\end{lemma}
\begin{proof}
The statement is an immediate consequence of Lemma~\ref{lHurwitzmoveGarsidepower}, which is a formula that relates the Garside power of $g_1g_2g_1^{-1}$ to the Garside power of $g_2$, based on Lemma~\ref{lHurwitzmoveGarsidenormalform}. (Note that we take the absolute value of the Garside power to determine the complexity.)
\end{proof}

An important observation concerning the statement of Lemma~\ref{lHurwitzmovecomplexitychange} is that the larger the right-divisor $\xi_1''$ of $\xi_1$ compared to the values of $\omega\left(\tau_1\right)$ and $e_1$, the smaller the complexity $c\left({\cal F}'\right)$. We will use this observation to strongly constrain factorizations of $\Delta^2$ into powers of positive half-twists of minimal complexity in their Hurwitz equivalence class. Ultimately, we will show that all such factorizations have complexity zero, one, or two (depending on the numbers of factors of each type in the factorization). 

We also remark that there is a symmetric version of Lemma~\ref{lHurwitzmovecomplexitychange} for the Hurwitz move $\left(g_1,g_2\right)\to \left(g_2,g_2^{-1}g_1g_2\right)$ instead of the Hurwitz move $\left(g_1,g_2\right)\to \left(g_1g_2g_1^{-1},g_1\right)$. In this paper, we will not repeat the symmetric statement and a reference to Lemma~\ref{lHurwitzmovecomplexitychange} will be a reference to either Lemma~\ref{lHurwitzmovecomplexitychange} as stated or its symmetric version.

\subsection{The classification of factorizations of $\Delta^2$ into positive half-twists up to Hurwitz equivalence}
\label{ssclassificationfactorizationspositivehalftwists}
In this subsection, we establish that every factorization ${\cal F}$ of $\Delta^2$ into positive half-twists is Hurwitz equivalent to the standard factorization $\Delta^2\equiv \left(\sigma_1,\sigma_2,\sigma_1,\sigma_1,\sigma_2,\sigma_1\right)$ (Definition~\ref{dstandardfactorization}). Equivalently, if $C,C'\subseteq \mathbb{CP}^2$ are degree three smooth Hurwitz curves, then $C$ is isotopic to $C'$, which is a special case of Theorem~\ref{tmainHurwitzcurve}. In particular, every degree three smooth Hurwitz curve in $\mathbb{CP}^2$ is isotopic to a projective algebraic curve, which is a special case of Theorem~\ref{tmainHurwitzalgebraic}.

The strategy of the proof is by induction on the complexity of a factorization of $\Delta^2$ into positive half-twists. Indeed, Proposition~\ref{pcomplexityzeroiffstandard} implies that a factorization ${\cal F}$ of $\Delta^2$ into positive half-twists is a factorization of $\Delta^2$ into standard positive half-twists if and only if the complexity $c\left({\cal F}\right) = 0$. If ${\cal F}$ is a factorization of $\Delta^2$ into positive half-twists with complexity $c\left({\cal F}\right)>0$, then we will prove that we can apply a finite sequence of Hurwitz moves to ${\cal F}$ to obtain a factorization ${\cal F}'$ of $\Delta^2$ into positive half-twists with complexity $c\left({\cal F}'\right)<c\left({\cal F}\right)$. 

Finally, we will show that every factorization of $\Delta^2$ into standard positive half-twists is Hurwitz equivalent to the standard factorization $\Delta^2\equiv \left(\sigma_1,\sigma_2,\sigma_1,\sigma_1,\sigma_2,\sigma_1\right)$. In fact, this final step is straightforward (based on Garside's theorem (Theorem~\ref{tmonoidembedding})) and we will also prove that every pair of factorizations of the same braid in $B_3$ into standard positive half-twists is Hurwitz equivalent.

Let us briefly describe how we will find a finite sequence of Hurwitz moves that reduces the complexity of a factorization ${\cal F}$ of $\Delta^2$ into positive half-twists, if at least one of the factors in ${\cal F}$ is a non-standard positive half-twist. In fact, if every factor in ${\cal F}$ is a non-standard positive half-twist, then we will show that there is a single complexity reducing Hurwitz move. The proof of this statement will be based on the combinatorics in our theory of duality of positive braids in Subsection~\ref{ssdualitypositivebraid}. Indeed, in a factorization $\left(g_1,g_2,\dots,g_k\right)$ of $\Delta^2$ into positive half-twists, there are strong constraints on how adjacent factors $g_j$ and $g_{j+1}$ combine for at least one $1\leq j\leq k-1$. We will state these constraints by firstly using our theory of duality of positive braids in order to study the factorization in the monoid $B_3^{+}$ of positive braids. Indeed, if we adopt Setup~\ref{setfactorization}, then the statement that $\left(g_1,g_2,\dots,g_k\right)$ is a factorization of $\Delta^2$ is equivalent to the equation $\Delta^{2+\omega_k} = \xi_1\xi_2\cdots \xi_k$. 

We will use our theory of duality in order to establish a lower bound on either (1) the size of the maximal right-divisor of $\xi_j$ that is left-dual to a left-divisor of $\xi_{j+1}$ or (2) the size of the maximal left-divisor of $\xi_{j}$ that is right-dual to a right-divisor of $\xi_{j-1}$, for at least one $1\leq j\leq k$ (if none of the factors are standard positive half-twists). Indeed, we will show that the equation $\Delta^{2+\omega_k} = \xi_1\xi_2\cdots\xi_k$ is invalid if the sizes of these divisors of $\xi_j$ are less than half the size of $\xi_j$ for every $1\leq j\leq k$. In particular, we will show that the size of one of these divisors of $\xi_j$ is strictly greater than half the size of $\xi_j$ for at least one $1\leq j\leq k$. 

Finally, we will use this constraint to establish the existence of a complexity reducing Hurwitz move locally applied to either the pair $\left(g_j,g_{j+1}\right)$ or the pair $\left(g_{j-1},g_j\right)$, in this case. The final step will be based on Lemma~\ref{lHurwitzmovecomplexitychange}, which computes the change in complexity of a factorization after the application of a Hurwitz move (see also the discussion following Lemma~\ref{lHurwitzmovecomplexitychange}).

If there are standard positive half-twists in the factorization, then it is still true that there is a finite sequence of Hurwitz moves that reduces the complexity of the factorization (as long as there are non-standard positive half-twists in the factorization).  The argument is quite similar to the case described in the previous paragraph, where none of the factors are standard positive half-twists. However, it requires multiple Hurwitz moves since standard positive half-twists have length one. In particular, a triple product $\xi_j\xi_{j+1}\xi_{j+2}$ might be divisible by $\Delta$, where $\xi_{j+1}$ is a positive half-twist, even if neither $\xi_j\xi_{j+1}$ nor $\xi_{j+1}\xi_{j+2}$ is divisible by $\Delta$ (see also the discussion preceding Lemma~\ref{ltripleproductindivisibleDelta}). If $\xi_j$ is also a standard positive half-twist, then we will show that we can apply a single Hurwitz move that does not change the complexity to the pair $\left(g_{j+1},g_{j+2}\right)$. We then reduce to a situation where a complexity reducing Hurwitz move can be applied (subsequently, to the pair $\left(g_j,g_{j+1}g_{j+2}g_{j+1}^{-1}\right)$).
 
Let us begin the proof. Firstly, we establish that local applications of the braid relations in a factorization can be achieved through Hurwitz moves. We will use this statement to show that if ${\cal F}$ and ${\cal F}'$ are factorizations of $\Delta^2$ into standard positive half-twists, then ${\cal F}$ is Hurwitz equivalent to ${\cal F}'$. We prove the statement in the generality of the braid group $B_n$ since it is no extra effort to do so. 

\begin{proposition}
\label{pbraidrelationHurwitzmove}
Let $g\in B_n$. Let ${\cal F}$ be a factorization of $g$ that is locally a triple of the form $\left(\dots,\sigma_i,\sigma_{i+1},\sigma_i,\dots\right)$, for some $1\leq i\leq n-1$. Let ${\cal F}'$ be the factorization of $g$ that is the same as the factorization ${\cal F}$ of $g$, except that the aforementioned local triple in ${\cal F}$ is modified to $\left(\dots,\sigma_{i+1},\sigma_i,\sigma_{i+1},\dots\right)$. In this case, the factorization ${\cal F}$ of $g$ is Hurwitz equivalent to the factorization ${\cal F}'$ of $g$ by a finite sequence of Hurwitz moves that does not include global conjugation. 
\end{proposition}
\begin{proof}
Firstly, it suffices to show that the factorization $\left(\sigma_1,\sigma_2,\sigma_1\right)$ of $\Delta$ is Hurwitz equivalent to the factorization $\left(\sigma_2,\sigma_1,\sigma_2\right)$ of $\Delta$ by a finite sequence of Hurwitz moves that does not include global conjugation, in the braid group $B_3$. Indeed, Hurwitz equivalence is functorial (Proposition~\ref{pHurwitzequivalencefunctorial}) and we have a map $B_3\to B_n$ defined by the rule $\sigma_1\to \sigma_i$ and $\sigma_2\to \sigma_{i+1}$. 

We have the following two Hurwitz moves \[\left(\sigma_1,\sigma_2,\sigma_1\right)\sim \left(\sigma_2,\sigma_2^{-1}\sigma_1\sigma_2,\sigma_1\right)\sim \left(\sigma_2,\sigma_1,\sigma_2\right).\] Indeed, $\sigma_1^{-1}\left(\sigma_2^{-1}\sigma_1\sigma_2\right)\sigma_1 = \sigma_1^{-1}\sigma_2^{-1}\sigma_2\sigma_1\sigma_2 = \sigma_2$, which justifies the equality in the second Hurwitz move. Therefore, the statement is established.
\end{proof}

Of course, it is clear that the local application of a commutativity relation in a factorization can be achieved through a Hurwitz move. Indeed, if $g_j$ and $g_{j+1}$ commute, then a Hurwitz move applied to a factorization $\left(\cdots,g_j,g_{j+1},\cdots\right)$ at the $j$th position results in the factorization $\left(\cdots,g_{j+1},g_j,\cdots\right)$. A positive word in the Artin generators representing a positive braid $\sigma\in B_n^{+}$ defines a factorization of $\sigma$ into standard positive half-twists. The application of Artin relations to such a positive word gives rise to other such positive words, and consequently defines other factorizations of the positive braid $\sigma$ into standard positive half-twists. In particular, Proposition~\ref{pbraidrelationHurwitzmove} implies that the uniqueness of factorizations of a positive braid into standard positive half-twists up to Hurwitz equivalence does not already fail at the level of the Artin relations in $B_n^{+}$. We summarize the discussion in the following consequence of Proposition~\ref{pbraidrelationHurwitzmove}.

\begin{lemma}
\label{lstandardfactorizationsHurwitzequivalent}
If ${\cal F}$ and ${\cal F}'$ are factorizations of a positive braid $\sigma\in B_n^{+}$ into standard positive half-twists, then ${\cal F}$ is Hurwitz equivalent to ${\cal F}'$ by a finite sequence of Hurwitz moves that does not include global conjugation.
\end{lemma}
\begin{proof}
A factorization of a positive braid $\sigma\in B_n^{+}$ into standard positive half-twists is equivalent to a product expansion of $\sigma$ in terms of the Artin generators $\sigma_1,\sigma_2,\dots,\sigma_{n-1}$. Garside's theorem (Theorem~\ref{tmonoidembedding}) implies that every pair of product expansions of a positive braid $\sigma\in B_n^{+}$ (in terms of the Artin generators) is related by a finite sequence of Artin relations within the monoid $B_n^{+}$ of positive braids. Proposition~\ref{pbraidrelationHurwitzmove} and the straightforward fact that the local application of a commutativity relation in a factorization can be achieved through a Hurwitz move, imply that each Artin relation within the monoid $B_n^{+}$ can be achieved by a finite sequence of Hurwitz moves that does not include global conjugation. Therefore, the statement is established.
\end{proof}

In fact, if ${\cal F}'$ is a factorization of $\Delta^2$ into positive half-twists that is Hurwitz equivalent to the standard factorization ${\cal F} = \left(\sigma_1,\sigma_2,\sigma_1,\sigma_1,\sigma_2,\sigma_1\right)$ of $\Delta^2$ (Definition~\ref{dstandardfactorization}), then the global conjugation move is redundant in any Hurwitz equivalence between ${\cal F}$ and ${\cal F}'$. Although we will not need the following statement in the sequel and the reader may skip the proof, we include it as motivation. Indeed, in this subsection, we will not use the global conjugation move. On the other hand, the global conjugation move is not redundant for Hurwitz equivalences between factorizations of $\Delta^2$ into powers of positive half-twists. We will use the global conjugation move in Subsection~\ref{ssclassificationfactorizationspowerspositivehalftwists} and Subsection~\ref{sssingularityconstraints} when we classify such factorizations.

\begin{lemma}
\label{lnoglobalconjugation}
If a factorization ${\cal F}'$ of $\Delta^2$ into positive half-twists is Hurwitz equivalent to the standard factorization ${\cal F}$ of $\Delta^2$, then ${\cal F}'$ can be obtained from ${\cal F}$ by a finite sequence of Hurwitz moves that does not include global conjugation. 
\end{lemma}
\begin{proof}
Proposition~\ref{pglobalconjugationcommutesHurwitz} states that global conjugation commutes with Hurwitz moves. Thus, it suffices to establish the statement if ${\cal F}'$ is obtained from the standard factorization ${\cal F} = \left(\sigma_1,\sigma_2,\sigma_1,\sigma_1,\sigma_2,\sigma_1\right)$ of $\Delta^2$ by a global conjugation move. 

A global conjugation move by an element $g\in B_3$ is the composition of a finite sequence of \textit{elementary} global conjugation moves, where each elementary global conjugation move is global conjugation by an Artin generator $\sigma_i$ for some $i\in \{1,2\}$. In particular, it suffices to establish the statement if ${\cal F}'$ is obtained from the standard factorization ${\cal F}$ of $\Delta^2$ by an elementary global conjugation move. The global conjugation of the standard factorization by $\sigma_1$ is \[\left(\sigma_1,\sigma_1^{-1}\sigma_2\sigma_1,\sigma_1,\sigma_1,\sigma_1^{-1}\sigma_2\sigma_1,\sigma_1\right)\sim \left(\sigma_2,\sigma_1,\sigma_1,\sigma_2,\sigma_1,\sigma_1\right),\] where the Hurwitz equivalence $\sim$ is realized by two Hurwitz moves. However, Lemma~\ref{lstandardfactorizationsHurwitzequivalent} implies that this latter factorization of $\Delta^2$ into standard positive half-twists is Hurwitz equivalent to the standard factorization ${\cal F}$ of $\Delta^2$ by a finite sequence of Hurwitz moves that does not include global conjugation. 

Similarly, by a symmetric argument, the global conjugation of ${\cal F}$ by $\sigma_2$ is Hurwitz equivalent to ${\cal F}$ by a finite sequence of Hurwitz moves that does not include global conjugation. (Alternatively, we can observe that conjugation of ${\cal F}$ by $\Delta$ is a factorization of $\Delta^2$ into standard positive half-twists, and $B_3 = \left\langle \sigma_1, \Delta \right\rangle$.) Therefore, the statement is established. 
\end{proof}

If ${\cal F}$ is a factorization of $\Delta^2$ into positive half-twists of minimal complexity in its Hurwitz equivalence class, then the strategy is to show that $c\left({\cal F}\right) = 0$. Ultimately, the following reformulation of Lemma~\ref{lstandardfactorizationsHurwitzequivalent} will complete the proof.

\begin{lemma}
\label{lstandardfactorizationspositivehalftwistsHurwitzequivalent}
If ${\cal F}$ is a factorization of $\Delta^2$ into positive half-twists and if $c\left({\cal F}\right) = 0$, then ${\cal F}$ is Hurwitz equivalent to the standard factorization $\Delta^2\equiv \left(\sigma_1,\sigma_2,\sigma_1,\sigma_1,\sigma_2,\sigma_1\right)$.
\end{lemma}
\begin{proof}
The statement is equivalent to Lemma~\ref{lstandardfactorizationsHurwitzequivalent}. Indeed, Proposition~\ref{pcomplexityzeroiffstandard} implies that $c\left({\cal F}\right) = 0$ if and only if ${\cal F}$ is a factorization of $\Delta^2$ into standard positive half-twists.
\end{proof}

In the rest of the proof, we will establish constraints on factorizations of $\Delta^2$ into positive half-twists that have minimal complexity in their Hurwitz equivalence class. Ultimately, we will use these constraints to establish that every such factorization has complexity zero. The following terminology will be convenient in the sequel.

\begin{definition}
We refer to a factorization as a \textit{minimal complexity factorization} if its complexity is minimal in its Hurwitz equivalence class. We also write that a factorization has \textit{minimal complexity} if it is a minimal complexity factorization.
\end{definition}

In general, we will establish strong constraints on minimal complexity factorizations into positive half-twists (even factorizations of elements other than $\Delta^2$). The following terminology is necessary for the subsequent statement. 

\begin{definition}
If ${\cal F} = \left(g_1,g_2,\dots,g_k\right)$ is a factorization of $g_1g_2\cdots g_k$, then a \textit{subfactorization} of ${\cal F}$ is a factorization ${\cal F}'=\left(g_{j'},g_{j'+1},\dots,g_j\right)$ of $g_{j'}g_{j'+1}\cdots g_j$ for some $1\leq j'\leq j\leq k$. 
\end{definition}

The following observation is the main motivation for the previous definition. 

\begin{proposition}
If ${\cal F}$ is a minimal complexity factorization, then every subfactorization of ${\cal F}$ is also a minimal complexity factorization. 
\end{proposition}
\begin{proof}
If ${\cal F}'$ is a subfactorization of ${\cal F}$, then a finite sequence of Hurwitz moves applied to ${\cal F}'$ corresponds to a finite sequence of Hurwitz moves applied to ${\cal F}$. Therefore, the statement is established. 
\end{proof} 

In particular, the condition that the complexity of ${\cal F}$ is minimized in its Hurwitz equivalence class also implies the set of local conditions that the complexity of each subfactorization ${\cal F}'$ of ${\cal F}$ is minimized in its Hurwitz equivalence class. We use this simple observation to restrict our attention to establishing constraints on minimal complexity factorizations with small numbers of factors.

We recall Lemma~\ref{lHurwitzmovecomplexitychange} which computes the change in complexity of a factorization into powers of positive half-twists after the application of a Hurwitz move. The following consequence of Lemma~\ref{lHurwitzmovecomplexitychange} is an important technical step for characterizing minimal complexity factorizations into positive half-twists.

\begin{lemma}
\label{lminimalcomplexitypositivehalftwists}
We adopt Setup~\ref{setfactorization} with $k=2$ and where $g_1$ and $g_2$ are powers of non-standard positive half-twists. If ${\cal F} = \left(g_1,g_2\right)$ is a minimal complexity factorization of $g_1g_2$, then the following conditions are satisfied:
\begin{description}
\item[(i)] If $g_1$ is a non-standard positive half-twist, then the maximal right-divisor $\xi_1''$ of $\xi_1$ that is left-dual to a left-divisor of $\overline{\rho_2}$ is a right-divisor of $\overline{\tau_1}$.
\item[(ii)] If $g_2$ is a non-standard positive half-twist, then the maximal left-divisor $\xi_2''$ of $\xi_2$ that is right-dual to a right-divisor of $\overline{\tau_1}$ is a left-divisor of $\overline{\rho_2}$.
\end{description}
\end{lemma}
\begin{proof}
We prove \textbf{(i)} since the proof of \textbf{(ii)} is similar. Firstly, Lemma~\ref{lmaximaldivisorcutdivisor} implies that $\xi_1''$ is a cut right-divisor of $\xi_1$. Furthermore, Proposition~\ref{pomeganumbercutdivisor} implies that $\xi_1''$ is a right-divisor of $\overline{\tau_1}$ if and only if $\omega\left(\xi_1''\right)\leq \omega\left(\overline{\tau_1}\right)$. Let us consider the factorization $\left(g_1g_2g_1^{-1},g_1\right)$ of $g_1g_2$ obtained from a Hurwitz move applied to the factorization ${\cal F}=\left(g_1,g_2\right)$ of $g_1g_2$. If $\omega\left(\xi_1''\right)>\omega\left(\tau_1\right)$, then Lemma~\ref{lHurwitzmovecomplexitychange} with $e_1 = 1$ implies that the complexity of the factorization $\left(g_1g_2g_1^{-1},g_1\right)$ is strictly less than the complexity of ${\cal F}$, since $\epsilon\leq 1$. Of course, this would contradict the assumption that ${\cal F}$ is a minimal complexity factorization. Therefore, the statement is established. 
\end{proof}

The following technical statement concerns how products of adjacent factors combine from the perspective of our theory of duality in a minimal complexity factorization ${\cal F}$ into positive half-twists. The statement is an improvement on Lemma~\ref{lminimalcomplexitypositivehalftwists}.

\begin{lemma}
\label{lboundedcontraction}
We adopt Setup~\ref{setfactorization} with $k = 2$ and where $g_1$ and $g_2$ are non-standard positive half-twists. Let us assume that ${\cal F} = \left(g_1,g_2\right)$ is a minimal complexity factorization of $g_1g_2$. 

In this case, the maximal right-divisor $\xi_1''$ of $\xi_1$ that is left-dual to a left-divisor of $\xi_2$ is a right-divisor of the cut closure $\overline{\tau_1}$. Furthermore, the maximal left-divisor $\xi_2''$ of $\xi_2$ that is right-dual to a right-divisor of $\xi_1$ is a left-divisor of the cut closure $\overline{\rho_2}$. 

Finally, $\xi_1''$ is not left-dual to $\overline{\rho_2}$ and $\xi_2''$ is not right-dual to $\overline{\tau_1}$. 
\end{lemma}
\begin{proof}
We will establish the first statement that $\xi_1''$ is a right-divisor of the cut closure $\overline{\tau_1}$. The proof that $\xi_2''$ is a left-divisor of the cut closure $\overline{\rho_2}$ is analogous.

Firstly, Lemma~\ref{lmaximaldivisorcutdivisor} implies that $\xi_1''$ is a cut right-divisor of $\xi_1$. Let us assume, for a contradiction, that $\xi_1''$ is not a right-divisor of the cut closure $\overline{\tau_1}$, or equivalently, $\overline{\tau_1}$ is a proper right-divisor of $\xi_1''$. In this case, Proposition~\ref{pcutdivisorsubdivisor} implies that $\overline{\tau_1}$ is a proper cut right-divisor of $\xi_1''$. 

We observe that the left-divisor $\chi_1''$ of $\xi_2$ that is right-dual to $\xi_1''$ cannot be a left-divisor of $\overline{\rho_2}$. Indeed, if $\chi_1''$ were a left-divisor of $\overline{\rho_2}$, then this would contradict the assumption that ${\cal F}$ is a minimal complexity factorization, according to Lemma~\ref{lminimalcomplexitypositivehalftwists}. We deduce that $\overline{\rho_2}$ is a proper left-divisor of $\chi_1''$.

If $\rho_2''$ is the left-divisor of $\chi_1''$ that is right-dual to $\overline{\tau_1}$, then Lemma~\ref{lminimalcomplexitypositivehalftwists} implies that $\rho_2''$ is a left-divisor of $\overline{\rho_2}$. However, we have assumed that $\overline{\tau_1}$ is a proper right-divisor of $\xi_1''$. In particular, we deduce that $\omega\left(\rho_2''\right) = \omega\left(\overline{\rho_2}\right)$. Indeed, if $\omega\left(\rho_2''\right)<\omega\left(\overline{\rho_2}\right)$, then $\overline{\rho_2}$ is right-dual to a right-divisor of $\xi_1''$ for which $\overline{\tau_1}$ is a proper right-divisor. However, this contradicts the assumption that ${\cal F}$ is a minimal complexity factorization, according to Lemma~\ref{lminimalcomplexitypositivehalftwists}. 

We use the constraint $\omega\left(\rho_2''\right) = \omega\left(\overline{\rho_2}\right)$ to obtain a contradiction and complete the proof of the first statement. Proposition~\ref{pcutclosureconjugate} implies that $\overline{\tau_1}$ begins with a product of two distinct generators. Lemma~\ref{lcharacterizationdualpositivebraids} implies that $\rho_2''$ does not end with a product of two distinct generators, since $\overline{\tau_1}$ is left-dual to $\rho_2''$. In particular, $\rho_2''\neq \overline{\rho_2}$, since Proposition~\ref{pcutclosureconjugate} implies that $\overline{\rho_2}$ ends with a product of two distinct generators.

However, if $\rho_2''\neq \overline{\rho_2}$, then $\rho_2''$ is not a cut left-divisor of $\xi_2$. Proposition~\ref{pmaximaldivisordualcharacterization}~\textbf{(i)} now implies that $\overline{\tau_1}$ is the maximal right-divisor of $\xi_1$ that is left-dual to a left-divisor of $\xi_2$. Of course, this contradicts our initial assumption that $\overline{\tau_1}$ is a proper right-divisor of $\xi_1''$.

The final statement is a consequence of Proposition~\ref{pmaximaldivisordualcharacterization}~\textbf{(ii)} since $\overline{\rho_2}$ is a cut left-divisor of $\xi_2$ that is succeeded by a product of two distinct generators, according to Proposition~\ref{pcutclosureconjugate}. Therefore, the statements are established.
\end{proof}

We now turn our attention to standard positive half-twists in a minimal complexity factorization ${\cal F}$ into powers of positive half-twists. We will find it convenient to rearrange the factors in ${\cal F}$ by a finite sequence of Hurwitz moves in order to strategically place the factors that are standard positive half-twists. However, the finite sequence of Hurwitz moves must not change the complexity of ${\cal F}$. 

Firstly, we consider a particular case where a standard positive half-twist is adjacent to the power of a non-standard positive half-twist. (We will establish statements in the greater generality of factorizations into powers of positive half-twists rather than simply factorizations into positive half-twists if it is no extra effort to do so. We will use the more general statements in Subsection~\ref{ssclassificationfactorizationspowerspositivehalftwists} and Subsection~\ref{sssingularityconstraints}.)

\begin{proposition}
\label{ptripledivisible}
We adopt Setup~\ref{setfactorization} with $k=3$. Let us assume that $\xi_1\xi_2\xi_3$ is divisible by $\Delta$, but neither $\xi_1\xi_2$ nor $\xi_2\xi_3$ is divisible by $\Delta$. In this case, $\xi_2$ is an Artin generator. 

If $g_1$ is the power of a non-standard positive half-twist, then the Hurwitz move $\left(g_1,g_2,g_3\right)\to \left(g_2,g_2^{-1}g_1g_2,g_3\right)$ does not change the complexity of the factorization. If $g_3$ is the power of a non-standard positive half-twist, then the Hurwitz move $\left(g_1,g_2,g_3\right)\to \left(g_1,g_2g_3g_2^{-1},g_2\right)$ does not change the complexity of the factorization.
\end{proposition}
\begin{proof}
Firstly, Lemma~\ref{ltripleproductindivisibleDelta} implies that $\xi_2$ is an Artin generator. Furthermore, if we write $\xi_2 = \sigma_k$ for some $k\in \{1,2\}$, then Proposition~\ref{pdualbraidproduct} implies that $\xi_1$ ends with $\sigma_{k'}$ and $\xi_3$ begins with $\sigma_{k'}$, where $k\neq k'\in \{1,2\}$.

Let us consider the case where $g_1$ is the power of a non-standard positive half-twist, since the case where $g_3$ is the power of a non-standard positive half-twist is similar. In this case, we can write $\xi_1 = \rho_1\sigma_{i_1}^{e_1}\tau_1$, where $\rho_1,\tau_1\in B_3^{+}$ are non-identity positive braids such that $\rho_1$ is dual to $\tau_1$. Lemma~\ref{lcharacterizationdualpositivebraids} implies that $\xi_1$ begins with $\left(\sigma_k\right)_{\left(\omega\left(\tau_1\right)\right)}$. In particular, the Garside power of $g_2^{-1}g_1g_2$ is equal to the Garside power of $g_1$, since $g_1 = \Delta^{-\omega\left(\tau_1\right)}\left(\xi_1\right)_{\left(\omega\left(\tau_1\right)\right)}$ and $g_2 = \left(\xi_2\right)_{\left(\omega\left(\tau_1\right)\right)}$, as in Setup~\ref{setfactorization}. Therefore, the complexity of the factorization $\left(g_2,g_2^{-1}g_1g_2,g_3\right)$ is equal to the complexity of the factorization $\left(g_1,g_2,g_3\right)$.
\end{proof}

Let ${\cal F}$ be a factorization, where we have a triple of adjacent factors in ${\cal F}$ that are standard positive half-twists with product equal to $\Delta$. In this case, we can move the triple to the beginning of ${\cal F}$ by a finite sequence of Hurwitz moves that does not change the complexity. Indeed, we have the following more general statement.

\begin{proposition}
\label{pfactorizationmoveGarsideelement}
We adopt Setup~\ref{setfactorization}. If $g_2g_3\cdots g_k = \Delta$, then the composition of $k-1$ Hurwitz moves \[\left(g_1,g_2,\dots,g_k\right)\to \left(g_2,g_3,\dots,g_k,\left(g_k^{-1}g_{k-1}^{-1}\cdots g_2^{-1}\right)g_1\left(g_2g_3\cdots g_k\right)\right)\] does not change the complexity of the factorization.
\end{proposition}
\begin{proof}
If $g_2g_3\cdots g_k= \Delta$, then $\left(g_k^{-1}g_{k-1}^{-1}\cdots g_2^{-1}\right)g_1\left(g_2g_3\cdots g_k\right) = \Delta^{-1}g_1\Delta$, which has the same Garside power as $g_1$. Therefore, the statement is established.
\end{proof}

We summarize Proposition~\ref{ptripledivisible} and Proposition~\ref{pfactorizationmoveGarsideelement} to find a canonical rearrangement of every factorization into powers of positive half-twists, where the standard positive half-twists are placed strategically.

\begin{lemma}
\label{lreorderingfactorssamecomplexity}
Let ${\cal F}' = \left(g_1',g_2',\dots,g_k'\right)$ be a factorization into powers of positive half-twists. The factorization ${\cal F}'$ is Hurwitz equivalent to a factorization ${\cal F}=\left(g_1,g_2,\dots,g_k\right)$ with the following properties, where we adopt the notation in Setup~\ref{setfactorization} for ${\cal F}$.
\begin{description}
\item[(i)] The complexity $c\left({\cal F}\right) = c\left({\cal F}'\right)$.
\item[(ii)] If $\xi_{j-1}\xi_j\xi_{j+1}$ is divisible by $\Delta$ and neither $g_{j-1}$ nor $g_{j+1}$ is the power of a standard positive half-twist, then either $\xi_{j-1}\xi_j$ or $\xi_j\xi_{j+1}$ is divisible by $\Delta$, for each $2\leq j\leq k-1$.
\item[(iii)] We have a (possibly empty) totally ordered subset $S = \{1,2,\dots,3m\}$ such that $g_j$ is a standard positive half-twist for each $j\in S$, the product $g_{3l+1}g_{3l+2}g_{3l+3} = \Delta$ for each $0\leq l\leq m-1$, and no triple $\{j,j+1,j+2\}\subseteq \{1,\dots,k\}\setminus S$ such that $g_j,g_{j+1},g_{j+2}$ are standard positive half-twists satisfies $g_jg_{j+1}g_{j+2} = \Delta$.
\end{description}
\end{lemma}
\begin{proof}
The statement is a consequence of repeated applications of Proposition~\ref{ptripledivisible} and Proposition~\ref{pfactorizationmoveGarsideelement}.
\end{proof}

We now wish to understand the role played by standard positive half-twists in a minimal complexity factorization. We will split the process into understanding the role played by a single standard positive half-twist and the role played by two adjacent standard positive half-twists, separately. Firstly, we consider factorizations where one of the factors is a standard positive half-twist. 

\begin{lemma}
\label{ladjacentstandardpositivehalftwist}
We adopt Setup~\ref{setfactorization} with $k=2$, where $g_1$ is the power of a non-standard positive half-twist and $g_2$ is a standard positive half-twist. If ${\cal F}=\left(g_1,g_2\right)$ is a minimal complexity factorization of $g_1g_2$, then $\xi_1\xi_2$ is indivisible by $\Delta$.  
\end{lemma}
\begin{proof}
Let us assume, for a contradiction, that $\xi_1\xi_2$ is divisible by $\Delta$. If $\xi_2 = \sigma_k$ for some $k\in \{1,2\}$, then Proposition~\ref{pdualbraidproduct} implies that $\xi_1$ ends with a product $\sigma_{k}\sigma_{k'}$ of distinct generators, where $k\neq k'\in \{1,2\}$. We can write $\xi_1 = \rho_1\sigma_{i_1}^{e_1}\tau_1$, where $\rho_1,\tau_1\in B_3^{+}$ are non-identity positive braids such that $\rho_1$ is dual to $\tau_1$. Lemma~\ref{lcharacterizationdualpositivebraids} implies that $\xi_1$ begins with $\left(\sigma_{k}\right)_{\left(\omega\left(\tau_1\right)\right)}$. We deduce that the Garside power of $g_2^{-1}g_1g_2$ is one less than the Garside power of $g_1$, since $g_1 = \Delta^{-\omega\left(\tau_1\right)}\left(\xi_1\right)_{\left(\omega\left(\tau_1\right)\right)}$ and $g_2 = \left(\xi_2\right)_{\left(\omega\left(\tau_1\right)\right)}$, as in Setup~\ref{setfactorization}. In particular, the complexity of the factorization $\left(g_2,g_2^{-1}g_1g_2\right)$ is one less than the complexity of ${\cal F} = \left(g_1,g_2\right)$, which contradicts the assumption that ${\cal F}$ is a minimal complexity factorization. Therefore, the statement is established.
\end{proof}

We now consider factorizations where two adjacent factors are standard positive half-twists.

\begin{lemma}
\label{ltwoadjacentstandardpositivehalftwists}
We adopt Setup~\ref{setfactorization} with $k=3$, where $g_1$ is the power of a non-standard positive half-twist, and $g_2$ and $g_3$ are standard positive half-twists. If ${\cal F}=\left(g_1,g_2,g_3\right)$ is a minimal complexity factorization of $g_1g_2g_3$, then $\xi_1\xi_2\xi_3$ is indivisible by $\Delta$. 
\end{lemma}
\begin{proof}
Let us assume, for a contradiction, that $\xi_1\xi_2\xi_3$ is divisible by $\Delta$. Lemma~\ref{ladjacentstandardpositivehalftwist} implies that $\xi_1\xi_2$ is indivisible by $\Delta$. In particular, if the last generator in $\xi_1$ is $\sigma_k$ for some $k\in \{1,2\}$, then Proposition~\ref{pdualbraidproduct} implies that $\xi_2 = \sigma_{k'}$ and $\xi_3 = \sigma_k$, where $k\neq k'\in \{1,2\}$. Furthermore, the last two generators in $\xi_1$ are both $\sigma_k$. We can write $\xi_1 = \rho_1\sigma_{i_1}^{e_1}\tau_1$, where $\rho_1,\tau_1\in B_3^{+}$ are non-identity positive braids such that $\rho_1$ is dual to $\tau_1$. Lemma~\ref{lcharacterizationdualpositivebraids} implies that $\xi_1$ begins with a product of distinct generators $\left(\sigma_{k'}\sigma_{k}\right)_{\left(\omega\left(\tau_1\right)\right)}$. 

Let us consider the composition of Hurwitz moves ${\cal F} = \left(g_1,g_2,g_3\right)\to  \left(g_2,g_3,g_3^{-1}g_2^{-1}g_1g_2g_3\right)$. We observe that the Garside power of $g_3^{-1}g_2^{-1}g_1g_2g_3$ is one less than the Garside power of $g_1$, since $g_1 = \Delta^{-\omega\left(\tau_1\right)}\left(\xi_1\right)_{\left(\omega\left(\tau_1\right)\right)}$, $g_2 = \left(\xi_2\right)_{\left(\omega\left(\tau_1\right)\right)}$ and $g_3 = \left(\xi_3\right)_{\left(\omega\left(\tau_1\right)\right)}$, as in Setup~\ref{setfactorization}. We deduce that the complexity of the factorization $\left(g_2,g_3,g_3^{-1}g_2^{-1}g_1g_2g_3\right)$ is one less than the complexity of ${\cal F}=\left(g_1,g_2,g_3\right)$, which contradicts the assumption that ${\cal F}$ is a minimal complexity factorization. Therefore, the statement is established.
\end{proof}

We are now prepared to establish the main result of this subsection.

\begin{theorem}
\label{tfactorizationpositivehalftwistsstandard}
If ${\cal F}$ is a factorization of $\Delta^k$ into positive half-twists for some positive integer $k$, then ${\cal F}$ is Hurwitz equivalent to a factorization of $\Delta^k$ into standard positive half-twists.
\end{theorem}
\begin{proof}
Let ${\cal F}$ be a minimal complexity factorization of $\Delta^k$ into positive half-twists. We will show that the complexity of ${\cal F}$ is zero. Let us assume, for a contradiction, that the complexity $c\left({\cal F}\right)>0$. We will apply the technical results in this subsection on minimal complexity factorizations in order to constrain the factorization ${\cal F}$, and obtain a contradiction. We can assume that the factorization ${\cal F}$ satisfies the conditions in Lemma~\ref{lreorderingfactorssamecomplexity}, after possibly replacing ${\cal F}$ with another minimal complexity factorization.

We adopt Setup~\ref{setfactorization}. Proposition~\ref{pnumberoffactorsinfactorization} implies that there are $3k$ factors in the factorization since $\epsilon\left(\Delta^{k}\right) = 3k$. We have $\Delta^{\omega_{3k} + k} = \xi_1\xi_2\cdots \xi_{3k}$. Let $J\subsetneq \{1,2,\dots,3k\}$ be the proper subset of all $j\in \{1,2,\dots,3k\}$ such that $g_j$ is a standard positive half-twist, and let $J' = \{1,2,\dots,3k\}\setminus J$ be the complement of $J$. We can assume that the subset $S$ in condition~\textbf{(iii)} of Lemma~\ref{lreorderingfactorssamecomplexity} is empty after possibly replacing the factorization $\left(g_1,g_2,\dots,g_{3k}\right)$ of $\Delta^{k}$ with the factorization $\left(g_{3k'+1},g_{3k'+2},\dots,g_{3k}\right)$ of $\Delta^{k-k'}$, if $S = \{1,2,\dots,3k'\}$. In this case, there is no adjacent triple of standard positive half-twists in the factorization with product equal to $\Delta$. 

If $j\in \{1,2,\dots,3k\}$, then define $\xi_j''$ to be the maximal right-divisor of $\xi_j$ that is left-dual to a left-divisor of $\xi_{j+1}$. Let us denote the relevant left-divisor of $\xi_{j+1}$ by $\chi_j''$. Lemma~\ref{lboundedcontraction} implies that $\xi_j''$ is a right-divisor of $\overline{\tau_j}$ for each $j\in J'$ such that $j+1\in J'$. Lemma~\ref{lboundedcontraction} also implies that $\chi_{j}''$ is a proper left-divisor of $\overline{\rho_{j+1}}$ for each $j+1\in J'$ such that $j\in J'$. If either $j\in J$ or $j+1\in J$, then Lemma~\ref{ladjacentstandardpositivehalftwist} implies that $\xi_j''$ is the identity braid. Similarly, if either $j\in J$ or $j+1\in J$, then Lemma~\ref{ladjacentstandardpositivehalftwist} implies that $\chi_j''$ is the identity braid. 

We define $\zeta_j\in B_3^{+}$ to be the positive braid such that $\xi_j=\chi_{j-1}''\zeta_j\xi_j''$. If $j\in J'$, then Proposition~\ref{pcutclosureconjugate} and Lemma~\ref{lboundedcontraction} imply that $\chi_{j-1}''\xi_j''\neq \xi_j$. If $j\in J$, then $\chi_{j-1}''$ and $\xi_j''$ are the identity braid. We conclude that  $\zeta_j$ is not the identity braid for each $j\in \{1,2,\dots,3k\}$. Let $\omega_j'' = \sum_{j'=1}^{j-1}\omega\left(\xi_{j'}''\right)$. We claim that the Garside normal form of $\xi_1\xi_2\cdots \xi_{3k}$ is \[\xi_1\xi_2\cdots\xi_{3k} = \left[\prod_{j=1}^{3k}\left(\zeta_j\right)_{\left(\omega_j''\right)}\right]\Delta^{\omega_{3k}''},\] where the positive braid in brackets is indivisible by $\Delta$. Of course, this would contradict the assumption that $\Delta^{\omega_{3k}+k} = \xi_1\xi_2\cdots \xi_{3k}$, since the positive braid indivisible by $\Delta$ in brackets has length at least $3k$. 

Let us prove the claim. Firstly, the equation is a consequence of the definition of $\zeta_j$ and the fact that $\xi_j''$ is left-dual to $\chi_{j}''$. Thus, we only need to show that $\prod_{j=1}^{3k} \left(\zeta_j\right)_{\left(\omega_j''\right)}$ is indivisible by $\Delta$. Of course, each factor $\left(\zeta_j\right)_{\left(\omega_j''\right)}$ is indivisible by $\Delta$, and we will use Lemma~\ref{lproductindivisibleDelta} in order to establish that the product $\prod_{j=1}^{3k} \left(\zeta_j\right)_{\left(\omega_j''\right)}$ is indivisible by $\Delta$.

We can write $J$ as a union $J_1\cup J_2\cup\cdots \cup J_s$ and the complement $J' = \{1,2,\dots,3k\}\setminus J$ as a union $J_1'\cup J_2'\cup\cdots \cup J_s'$. We can arrange that $J_r$ is a maximal totally ordered subset of $J$ and $J_r'$ is a maximal totally ordered subset of $J'$ for each $1\leq r\leq s$, such that $\sup{J_r}<\inf{J_r'}$ and $\sup{J_r'}<\inf{J_{r+1}}$, for each $1\leq r\leq s$ (where possibly either (or both) $J_1=\emptyset$ or $J_s'=\emptyset$). 

If $j\in J_r'$ is not the maximal element of $J_r'$, then Proposition~\ref{pmaximaldivisordualindivisibleDelta} and the maximality of $\xi_j''$ imply that $\left(\zeta_j\right)_{\left(\omega_j''\right)}\left(\zeta_{j+1}\right)_{\left(\omega_{j+1}''\right)}$ is indivisible by $\Delta$. If $\zeta_j$ is an Artin generator (a positive braid with word length one), then Proposition~\ref{pcutclosureconjugate} and Lemma~\ref{lboundedcontraction} imply that $\left(\chi_{j-1}''\right)^{-1}\overline{\rho_j} = \zeta_j$ is an Artin generator and $\xi_j'' = \overline{\tau_j}$. In particular, if $j\in J_r'$, then $j$ is neither the maximal element nor the minimal element of $J_r'$. Furthermore, Proposition~\ref{pcutclosureconjugate} implies that $\overline{\rho_j}$ ends with a product of two distinct generators, and thus Lemma~\ref{lcutdivisorcharacterization} implies that $\chi_{j-1}''$ is not a cut left-divisor of $\overline{\rho_j}$. Proposition~\ref{pmaximaldivisordualcharacterization}~\textbf{(i)} implies that the index of the last generator in $\left(\zeta_{j-1}\right)_{\left(\omega_{j-1}''\right)}$ is the same as the index of the first generator in $\left(\zeta_j\right)_{\left(\omega_j''\right)}$. We conclude by using Lemma~\ref{lproductindivisibleDelta} that the product $\prod_{j\in J_r'} \left(\zeta_j\right)_{\left(\omega_j''\right)}$ is indivisible by $\Delta$ for each $1\leq r\leq s$.

If $J_r = \{j_0\}$ for some $1\leq r\leq s$, then the product \[\left(\prod_{j\in J_{r-1}'} \left(\zeta_j\right)_{\left(\omega_j''\right)}\right)\zeta_{j_0}\left(\prod_{j\in J_r'} \left(\zeta_j\right)_{\left(\omega_j''\right)}\right)\] is indivisible by $\Delta$. Indeed, if not, then either $\left(\prod_{j\in J_{r-1}'} \left(\zeta_j\right)_{\left(\omega_j''\right)}\right)\zeta_{j_0}$ or $\zeta_{j_0}\left(\prod_{j\in J_r'} \left(\zeta_j\right)_{\left(\omega_j''\right)}\right)$ is divisible by $\Delta$, since the factorization ${\cal F}$ satisfies condition~\textbf{(ii)} in Lemma~\ref{lreorderingfactorssamecomplexity}. However, this would contradict the assumption that ${\cal F}$ is a minimal complexity factorization, according to Lemma~\ref{ladjacentstandardpositivehalftwist}. 

The product $\prod_{j\in J_r} \left(\zeta_j\right)_{\left(\omega_j''\right)}$ is indivisible by $\Delta$ for each $1\leq r\leq s$, since the factorization ${\cal F}$ satisfies condition~\textbf{(iii)} in Lemma~\ref{lreorderingfactorssamecomplexity} and we have reduced to the case where $S=\emptyset$. Furthermore, if $1\leq r\leq s$ and if the cardinality of $J_r$ is at least two, then the products $\left(\prod_{j\in J_r} \left(\zeta_j\right)_{\left(\omega_j''\right)}\right)\left(\prod_{j\in J_r'} \left(\zeta_{j}\right)_{\left(\omega_j''\right)}\right)$ and $\left(\prod_{j\in J_{r-1}'} \left(\zeta_j\right)_{\left(\omega_j''\right)}\right)\left(\prod_{j\in J_r} \left(\zeta_{j}\right)_{\left(\omega_j''\right)}\right)$ are indivisible by $\Delta$. Indeed, this is a consequence of Lemma~\ref{ladjacentstandardpositivehalftwist} and Lemma~\ref{ltwoadjacentstandardpositivehalftwists}, since ${\cal F}$ is a minimal complexity factorization. Finally, Lemma~\ref{lproductindivisibleDelta} implies that $\prod_{j=1}^{3k}\left(\zeta_j\right)_{\left(\omega_j''\right)}$ is indivisible by $\Delta$. We conclude that the Garside normal form of $\xi_1\xi_2\cdots\xi_{3k}$ is \[\xi_1\xi_2\cdots \xi_{3k} =  \left[\prod_{j=1}^{3k}\left(\zeta_j\right)_{\left(\omega_j''\right)}\right]\Delta^{\omega_{3k}''}.\] Therefore, the statement is established. 
\end{proof}

Let us summarize the results in this subsection.

\begin{theorem}
\label{tmainfactorizationpositivehalftwists}
If ${\cal F}$ is a factorization of $\Delta^2$ into positive half-twists, then ${\cal F}$ is Hurwitz equivalent to the standard factorization \[\Delta^2\equiv \left(\sigma_1,\sigma_2,\sigma_1,\sigma_1,\sigma_2,\sigma_1\right).\]
\end{theorem}
\begin{proof}
Theorem~\ref{tfactorizationpositivehalftwistsstandard} implies that ${\cal F}$ is Hurwitz equivalent to a factorization of $\Delta^2$ into standard positive half-twists. Lemma~\ref{lstandardfactorizationspositivehalftwistsHurwitzequivalent} implies that a factorization of $\Delta^2$ into standard positive half-twists is Hurwitz equivalent to the standard factorization $\Delta^2\equiv \left(\sigma_1,\sigma_2,\sigma_1,\sigma_1,\sigma_2,\sigma_1\right)$. Therefore, the statement is established.
\end{proof}

We state the geometric reformulation of Theorem~\ref{tmainfactorizationpositivehalftwists}.

\begin{theorem}
\label{tuniqueisotopyclasssmooth}
If $C,C'\subseteq \mathbb{CP}^2$ are degree three smooth Hurwitz curves, then $C$ is isotopic to $C'$.
\end{theorem}
\begin{proof}
The braid monodromy factorization of a degree three smooth Hurwitz curve is a factorization of $\Delta^2$ into positive half-twists. Furthermore, the Hurwitz equivalence class of the braid monodromy factorization of the curve completely determines the isotopy class of the curve. 

Theorem~\ref{tmainfactorizationpositivehalftwists} implies that the braid monodromy factorization of $C$ is Hurwitz equivalent to the braid monodromy factorization of $C'$. Therefore, the statement is established.
\end{proof}

In fact, Theorem~\ref{tmainfactorizationpositivehalftwists} is equivalent to Theorem~\ref{tmainHurwitzcurve}, Theorem~\ref{tmainHurwitzalgebraic} and Theorem~\ref{tmainsymplectic} (in the Introduction) in the special case of smooth curves. In the following subsection, we establish these results in full generality (for curves with singularities).

\subsection{The classification of factorizations of $\Delta^2$ into powers of positive half-twists}
\label{ssclassificationfactorizationspowerspositivehalftwists}
In this subsection, we will extend the classification of factorizations of $\Delta^2$ into positive half-twists in Subsection~\ref{ssclassificationfactorizationspositivehalftwists} to a classification of factorizations of $\Delta^2$ into powers of positive half-twists. Let ${\cal F}$ be a factorization of $\Delta^2$ into powers of positive half-twists. Let us assume that ${\cal F}$ has the same number of factors of each type as a standard factorization ${\cal F}'$ (in Definition~\ref{dstandardfactorization}). We will show that ${\cal F}$ is Hurwitz equivalent to ${\cal F}'$. In particular, we will show that if ${\cal F}$ is a factorization of $\Delta^2$ into positive half-twists and squares of positive half-twists, then the number of squares of positive half-twists in ${\cal F}$ is a complete invariant of the Hurwitz equivalence class of ${\cal F}$. The corresponding geometric statement is that the number of nodes in a degree three nodal Hurwitz curve $C\subseteq \mathbb{CP}^2$ is a complete invariant of the isotopy class of $C\subseteq \mathbb{CP}^2$. (Proposition~\ref{pnumberoffactorsinfactorization} implies that the number of nodes in a degree three nodal Hurwitz curve $C\subseteq \mathbb{CP}^2$ is an element of $\{1,2,3\}$.). 

We will also show that there is a unique factorization of $\Delta^2$ into the cube of a positive half-twist and three positive half-twists up to Hurwitz equivalence. The corresponding geometric statement is that there is a unique isotopy class of degree three simple Hurwitz curves with a single cuspidal singularity. Finally, we will show that there is a unique factorization of $\Delta^2$ into the fourth power of a positive half-twist and two positive half-twists up to Hurwitz equivalence. The corresponding geometric statement is that there is a unique isotopy class of degree three simple Hurwitz curves with a single tacnodal singularity. In Subsection~\ref{sssingularityconstraints}, we will show that no other configurations of singularities are possible for degree three simple Hurwitz curves in $\mathbb{CP}^2$, partially using technical statements established in this subsection. 

The proofs in this subsection are very similar in style to those in Subsection~\ref{ssclassificationfactorizationspositivehalftwists}. For example, we will need a statement analogous to Lemma~\ref{ladjacentstandardpositivehalftwist} (on standard positive half-twists in a factorization) concerning squares of standard positive half-twists in a factorization. We will also need statements analogous to Lemma~\ref{lboundedcontraction} (on adjacent factors that are positive half-twists in a minimal complexity factorization) on adjacent factors that are powers of positive half-twists in a minimal complexity factorization. 

However, we will need to consider global conjugation moves applied to factorizations of $\Delta^2$ into powers of positive half-twists, although global conjugation moves did not play a role in Subsection~\ref{ssclassificationfactorizationspositivehalftwists}. Indeed, an analogue of Lemma~\ref{lnoglobalconjugation} is not necessarily true for factorizations of $\Delta^2$ into powers of positive half-twists. In other words, a pair of factorizations of $\Delta^2$ into powers of positive half-twists that are equivalent by a finite sequence of Hurwitz moves and global conjugation moves, may not necessarily be equivalent by a finite sequence of Hurwitz moves alone.

Let us commence the classification of factorizations of $\Delta^2$ into powers of positive half-twists. The first goal is to establish a criterion for the existence of complexity reducing global conjugation moves applied to certain factorizations of $\Delta^2$ into powers of positive half-twists. We recall that a \textit{permutation braid} in $B_3$ is either an Artin generator or the product of two distinct Artin generators (this convention is slightly different from the literature, where $\Delta$ is also often considered to be a permutation braid). Let $g\in B_3$ be the power of a non-standard positive half-twist. We begin by establishing a technical result concerning the reduction of the absolute value of the Garside power of $g$ by conjugation in $B_3$. We consider conjugation by permutation braids in $B_3$ since every conjugation operation in $B_3$ can be decomposed as a composition of conjugation operations by permutation braids.

\begin{proposition}
\label{pGarsidepowerreduction}
Let $g\in B_3$ be the power of a non-standard positive half-twist and let us write the Garside normal form of $g$ as \[g = \Delta^{-\omega\left(\tau\right)}\rho\sigma_i^{e}\tau,\] where $\rho,\tau\in B_3^{+}$ are positive braids such that $\rho$ is dual to $\tau$, $e\geq 1$ is a positive integer, and $i\in \{1,2\}$ (Theorem~\ref{tGarsidenormalformpositivehalftwist}). Let $\tau'\in B_3^{+}$ be a permutation braid that is a right-divisor of the cut closure $\overline{\tau}$. If $\tau'$ is a cut right-divisor of $\overline{\tau}$, then the absolute value of the Garside power of $\tau'g\left(\tau'\right)^{-1}$ is strictly less than the absolute value of the Garside power of $g$. 
\end{proposition}
\begin{proof}
If $\tau'\in B_3^{+}$ is a permutation braid that is a cut right-divisor of $\overline{\tau}$, then we claim that $\left(\tau'\right)_{\left(\omega\left(\tau\right)\right)}\rho$ is divisible by $\Delta$. Of course, this claim implies that the Garside power of $\tau'\Delta^{-\omega\left(\tau\right)}\rho$ is strictly greater than $-\omega\left(\tau\right)$. Furthermore, the reverse triangle inequality for Garside powers (Proposition~\ref{preversetriangleinequality}) implies that the absolute value of the Garside power of $\tau'g\left(\tau'\right)^{-1} = \left(\tau'\Delta^{-\omega\left(\tau\right)}\rho\right)\left(\sigma_i^{e}\tau\left(\tau'\right)^{-1}\right)$ is strictly less than the absolute value of the Garside power of $g$, since the Garside power of $\left(\sigma_i^{e}\tau\left(\tau'\right)^{-1}\right)$ is $0$. In other words, the claim implies the statement.

Let us establish the claim. If $\tau'$ is a cut right-divisor of $\overline{\tau}$, then Proposition~\ref{pdivisorconjugatedual} implies that there is a left-divisor $\rho'$ of $\rho$ such that $\left(\rho'\right)_{\left(\omega\left(\tau\right)-1\right)}$ is left-dual to $\tau'$. The pseudo-symmetry of duals (Proposition~\ref{ppseudosymmetryduality}) implies that $\left(\tau'\right)_{\left(1\right)}$ is left-dual to $\left(\rho'\right)_{\left(\omega\left(\tau\right)-1\right)}$. We deduce that $\left(\tau'\right)_{\left(\omega\left(\tau\right)\right)}$ is left-dual to $\rho'$. In particular, $\left(\tau'\right)_{\left(\omega\left(\tau\right)\right)}\rho$ is divisible by $\Delta$ and the claim is true. Therefore, the statement is established.
\end{proof}

Let $g\in B_3$ be the power of a non-standard positive half-twist. We now characterize the cases where conjugation of $g$ by a permutation braid does not increase the absolute value of the Garside power.

\begin{proposition}
\label{pGarsidepowerconjugateconditions}
Let $\tau'\in B_3^{+}$ be a permutation braid. Let $g\in B_3$ be the power of a non-standard positive half-twist and let us write the Garside normal form of $g$ as \[g = \Delta^{-\omega\left(\tau\right)}\rho\sigma_i^{e}\tau,\] where $\rho,\tau\in B_3^{+}$ are positive braids such that $\rho$ is dual to $\tau$, $e\geq 1$ is a positive integer, and $i\in \{1,2\}$ (Theorem~\ref{tGarsidenormalformpositivehalftwist}). The absolute value of the Garside power of $\tau'g\left(\tau'\right)^{-1}$ is at most the absolute value of the Garside power of $g$ if and only if either $\left(\tau'\right)_{\left(\omega\left(\tau\right)\right)}\overline{\rho}$ is divisible by $\Delta$ or $\tau'$ is a right-divisor of $\overline{\tau}$.
\end{proposition}
\begin{proof}
We have that \[\tau'g\left(\tau'\right)^{-1} = \Delta^{-\omega\left(\tau\right)}\left[\left(\tau'\right)_{\left(\omega\left(\tau\right)\right)}\rho\sigma_i^{e}\tau\left(\tau'\right)^{-1}\right].\] In particular, the absolute value of the Garside power of $\tau'g\left(\tau'\right)^{-1}$ is equal to the sum of the absolute value of the Garside power of $g$ (which is $\omega\left(\tau\right)$) and $-\epsilon$, where $\epsilon$ is the Garside power of the braid in brackets. Furthermore, we have that $\epsilon$ is equal to the sum of the Garside power of $\left(\tau'\right)_{\left(\omega\left(\tau\right)\right)}\overline{\rho}\sigma_i^{e-1}$ and the Garside power of $\overline{\tau}\left(\tau'\right)^{-1}$. Indeed, this is a consequence of Proposition~\ref{pcutclosureconjugate}, which implies that $\rho\sigma_i^{e}\tau = \overline{\rho}\sigma_i^{e-1}\overline{\tau}$, and Proposition~\ref{preversetriangleinequality}, since $\tau'$ is a permutation braid.

The Garside power of $\left(\tau'\right)_{\left(\omega\left(\tau\right)\right)}\overline{\rho}\sigma_i^{e-1}$ is an element of $\{0,1\}$, and the Garside power of $\overline{\tau}\left(\tau'\right)^{-1}$ is an element of $\{-1,0\}$, also since $\tau'$ is a permutation braid. Furthermore, the Garside power of $\left(\tau'\right)_{\left(\omega\left(\tau\right)\right)}\overline{\rho}\sigma_i^{e-1}$ is $1$ if and only if $\left(\tau'\right)_{\left(\omega\left(\tau\right)\right)}\overline{\rho}$ is divisible by $\Delta$, and the Garside power of $\overline{\tau}\left(\tau'\right)^{-1}$ is $0$ if and only if $\tau'$ is a right-divisor of $\overline{\tau}$. Therefore, the statement is established.
\end{proof}

We will use Proposition~\ref{pGarsidepowerreduction} to find complexity reducing global conjugation moves for certain factorizations of $\Delta^2$ into powers of positive half-twists. However, we will need to show that global conjugation reduces the (overall) complexity of the factorization and not simply the absolute value of the Garside power of an individual factor. We can achieve this by combining Proposition~\ref{pGarsidepowerreduction} with the following technical result. 

\begin{proposition}
\label{pGarsidepowerreductionconsistency}
We adopt Setup~\ref{setfactorization} with $k=2$, where $g_1$ and $g_2$ are powers of non-standard positive half-twists. Let us assume that $\xi_1\xi_2$ is divisible by $\Delta$. Let $\tau_1'\in B_3^{+}$ be a permutation braid. If the absolute value of the Garside power of $\tau_1'g_1\left(\tau_1'\right)^{-1}$ is at most the absolute value of the Garside power of $g_1$, then the absolute value of the Garside power of $\tau_1'g_2\left(\tau_1'\right)^{-1}$ is at most the absolute value of the Garside power of $g_2$. 
\end{proposition}
\begin{proof}
If the absolute value of the Garside power of $\tau_1'g_1\left(\tau_1'\right)^{-1}$ is at most the absolute value of the Garside power of $g_1$, then Proposition~\ref{pGarsidepowerconjugateconditions} implies that either $\tau_1'\overline{\rho_1}$ is divisible by $\Delta$ or $\left(\tau_1'\right)_{\left(\omega\left(\tau_1\right)\right)}$ is a right-divisor of $\overline{\tau_1}$ (in the notation of Setup~\ref{setfactorization}, which differs slightly from the notation in the statement of Proposition~\ref{pGarsidepowerconjugateconditions}). We will consider each of these two cases separately. We will show that either $\left(\tau_1'\right)_{\left(\omega\left(\tau_1\right)\right)}\xi_2$ is divisible by $\Delta$ or $\left(\tau_1'\right)_{\omega\left(\tau_1\right)+\omega\left(\tau_2\right)}$ is a right-divisor of $\overline{\tau_2}$. Indeed, in the notation of Setup~\ref{setfactorization}, we have $g_2 = \Delta^{-\omega\left(\tau_2\right)}\left(\xi_2\right)_{\left(\omega\left(\tau_1\right)+\omega\left(\tau_2\right)\right)}$, and thus either of these conclusions would imply that the absolute value of the Garside power of $\tau_1'g_2\left(\tau_1'\right)^{-1}$ is at most the absolute value of the Garside power of $g_2$.

Firstly, if $\left(\tau_1'\right)_{\left(\omega\left(\tau_1\right)\right)}$ is a right-divisor of $\overline{\tau_1}$, then we consider the cases according to whether or not $\left(\tau_1'\right)_{\left(\omega\left(\tau_1\right)\right)}$ is a cut right-divisor of $\overline{\tau_1}$ separately. If $\left(\tau_1'\right)_{\left(\omega\left(\tau_1\right)\right)}$ is a cut right-divisor of $\overline{\tau_1}$, then the hypothesis that $\xi_1\xi_2$ is divisible by $\Delta$ implies that $\left(\tau_1'\right)_{\left(\omega\left(\tau_1\right)\right)}\xi_2$ is divisible by $\Delta$. Indeed, $\tau_1'$ is a permutation braid and we can apply Proposition~\ref{pdualbraidproduct} and Lemma~\ref{lcutdivisorcharacterization}. In particular, the absolute value of the Garside power of $\tau_1'g_2\left(\tau_1'\right)^{-1}$ is at most the absolute value of the Garside power of $g_2$ in this case.

If $\left(\tau_1'\right)_{\left(\omega\left(\tau_1\right)\right)}$ is not a cut right-divisor of $\overline{\tau_1}$, then Lemma~\ref{lcutdivisorcharacterization} implies that $\overline{\tau_1}$ ends with a product $\sigma_k\sigma_{k'}$ of distinct generators, where $\{k,k'\} = \{1,2\}$ and $\tau_1' = \left(\sigma_{k'}\right)_{\left(\omega\left(\tau_1\right)\right)}$. The hypothesis that $\xi_1\xi_2$ is divisible by $\Delta$ and Proposition~\ref{pdualbraidproduct} imply that $\rho_2$ begins with $\sigma_k$. Lemma~\ref{lcharacterizationdualpositivebraids} implies that $\tau_2$ ends with $\left(\sigma_{k'}\right)_{\left(\omega\left(\tau_2\right)\right)}$. We deduce that $\left(\tau_1'\right)_{\left(\omega\left(\tau_1\right)+\omega\left(\tau_2\right)\right)}$ is a right-divisor of $\overline{\tau_2}$. In particular, the absolute value of the Garside power of $\tau_1'g_2\left(\tau_1'\right)^{-1}$ is at most the absolute value of the Garside power of $g_2$ in this case.

Secondly, if $\tau_1'\overline{\rho_1}$ is divisible by $\Delta$, then $\tau_1'$ is left-dual to a left-divisor $\rho_1'$ of $\overline{\rho_1}$, since $\tau_1'$ is a permutation braid. We consider the cases according to whether or not $\rho_1'$ is a cut left-divisor of $\overline{\rho_1}$ separately. If $\rho_1'$ is a cut left-divisor of $\overline{\rho_1}$, then Proposition~\ref{pdivisorconjugatedual} implies that $\left(\rho_1'\right)_{\left(\omega\left(\tau_1\right)-1\right)}$ is left-dual to a right-divisor of $\tau_1$. However, the pseudo-symmetry of duals (Proposition~\ref{ppseudosymmetryduality}) implies that $\left(\rho_1'\right)_{\left(1\right)}$ is left-dual to $\tau_1'$. We deduce that $\left(\tau_1'\right)_{\left(\omega\left(\tau_1\right)\right)}$ is a right-divisor of $\tau_1$. However, we have already shown that the absolute value of the Garside power of $\tau_1'g_2\left(\tau_1'\right)^{-1}$ is at most the absolute value of the Garside power of $g_2$ in this case.

If $\rho_1'$ is not a cut left-divisor of $\overline{\rho_1}$, then Lemma~\ref{lcutdivisorcharacterization} implies that $\overline{\rho_1}$ begins with a product $\sigma_k\sigma_{k'}$ of distinct generators, where $\{k,k'\}=\{1,2\}$ and $\rho_1' = \sigma_k$. In this case, $\tau_1' = \sigma_k\sigma_{k'}$ since $\tau_1'$ is left-dual to $\rho_1'$. Lemma~\ref{lcharacterizationdualpositivebraids} implies that $\tau_1$ ends with $\left(\sigma_k\right)_{\left(\omega\left(\tau_1\right)-1\right)}$. The hypothesis that $\xi_1\xi_2$ is divisible by $\Delta$ and Proposition~\ref{pdualbraidproduct} imply that $\rho_2$ begins with $\left(\sigma_{k'}\right)_{\left(\omega\left(\tau_1\right)-1\right)}$. However, the product $\left(\tau_1'\right)_{\left(\omega\left(\tau_1\right)\right)}\left(\sigma_{k'}\right)_{\left(\omega\left(\tau_1\right)-1\right)}=\Delta$. We conclude that $\left(\tau_1'\right)_{\left(\omega\left(\tau_1\right)\right)}\xi_2$ is divisible by $\Delta$. In particular, the absolute value of the Garside power of $\tau_1'g_2\left(\tau_1'\right)^{-1}$ is at most the absolute value of the Garside power of $g_2$ in this case.

Therefore, the statement is established in all cases.
\end{proof}

We conclude that a complexity reducing global conjugation move can be applied to certain factorizations of $\Delta^2$ into powers of non-standard positive half-twists.

\begin{lemma}
\label{lcomplexityreducingglobalconjugation}
We adopt Setup~\ref{setfactorization}, where $g_j$ is the power of a non-standard positive half-twist for each $1\leq j\leq k$. If $\xi_j\xi_{j+1}$ is divisible by $\Delta$ for each $1\leq j\leq k-1$, then we can apply a complexity reducing global conjugation move to the factorization ${\cal F}$. 
\end{lemma}
\begin{proof}
The statement follows from Proposition~\ref{pGarsidepowerreduction} and Proposition~\ref{pGarsidepowerreductionconsistency}. Indeed, if $\tau_1'$ is the maximal permutation braid right-divisor of $\xi_1$ that is left-dual to a left-divisor of $\xi_2$, then Proposition~\ref{pdualbraidproduct} and Lemma~\ref{lcutdivisorcharacterization} imply that $\tau_1'$ is a cut right-divisor of $\xi_1$. The hypothesis that $g_1$ is the power of a non-standard positive half-twist implies that $\tau_1'$ is a cut right-divisor of $\overline{\tau_1}$. Proposition~\ref{pGarsidepowerreduction} implies that the absolute value of the Garside power of $\tau_1'g_1\left(\tau_1'\right)^{-1}$ is strictly less than the absolute value of the Garside power of $g_1$. Furthermore, repeated applications of Proposition~\ref{pGarsidepowerreductionconsistency} imply that the absolute value of the Garside power of $\tau_1'g_j\left(\tau_1'\right)^{-1}$ is at most the absolute value of the Garside power of $g_j$ for each $1\leq j\leq k$. Therefore, we have the desired inequality on the complexity $c\left(\tau_1'{\cal F}\left(\tau_1'\right)^{-1}\right) < c\left({\cal F}\right)$ and the statement is established.
\end{proof}

The goal is to establish that every factorization of $\Delta^2$ into powers of positive half-twists is Hurwitz equivalent to a (unique) standard factorization (in Definition~\ref{dstandardfactorization}). Firstly, we establish that every complexity zero factorization of $\Delta^2$ into powers of positive half-twists is Hurwitz equivalent to a (unique) standard factorization. The statement is an analogue of Lemma~\ref{lstandardfactorizationsHurwitzequivalent} on complexity zero factorizations of $\Delta^2$ into positive half-twists. We establish the statement in several steps.

\begin{lemma}
\label{loneortwosquarespositivehalftwistsstandard}
Let ${\cal F}$ be a factorization of $\Delta^2$ into positive half-twists and squares of positive half-twists. Let us also assume that there is at least one and at most two squares of positive half-twists in ${\cal F}$. If the complexity $c\left({\cal F}\right) = 0$, then ${\cal F}$ is Hurwitz equivalent to the unique standard factorization in Definition~\ref{dstandardfactorization} with the same number of factors of each type as ${\cal F}$. 
\end{lemma}
\begin{proof}
Firstly, we consider the case where there is one square of a positive half-twist in the factorization ${\cal F}$. We can assume after global conjugation of ${\cal F}$ by a power of $\Delta$ that the square of a positive half-twist in ${\cal F}$ is $\sigma_1^2$. If $1\leq j\leq 5$, then define ${\cal F}_j$ to be the factorization of $\Delta^2$ with one square $\sigma_1^2$ of a positive half-twist in the $j$th position such that if we arrange the factors of ${\cal F}_j$ in a circle, then $\sigma_1^2$ is followed by the factors $\sigma_{2},\sigma_1,\sigma_1,\sigma_{2}$ (in either the clockwise or the counterclockwise direction). 

For example, ${\cal F}_1 = \left(\sigma_1^2,\sigma_2,\sigma_1,\sigma_1,\sigma_2\right)$ (the standard factorization in Definition~\ref{dstandardfactorization}) and ${\cal F}_2 = \left(\sigma_2,\sigma_1^2,\sigma_2,\sigma_1,\sigma_1\right)$. The hypotheses that $c\left({\cal F}\right) = 0$ and that ${\cal F}$ is a factorization of $\Delta^2$ imply that ${\cal F} = {\cal F}_j$ for some $1\leq j\leq 5$. We will show that ${\cal F}_{j}$ is Hurwitz equivalent to ${\cal F}_{j'}$ for every pair $j,j'\in \{1,2,\dots,5\}$. 

Indeed, we will show that ${\cal F}_1$ is Hurwitz equivalent to ${\cal F}_3$ and ${\cal F}_4$ and symmetrical versions of the argument will handle the other cases. Firstly, we have the following composition of two Hurwitz moves \[{\cal F}_1\sim \left(\sigma_2,\sigma_1,\left(\sigma_2\sigma_1\right)^{-1}\sigma_1^2\left(\sigma_2\sigma_1\right),\sigma_1,\sigma_2\right) = \left(\sigma_2,\sigma_1,\sigma_2^2,\sigma_1,\sigma_2\right),\] and ${\cal F}_3 = \left(\sigma_1,\sigma_2,\sigma_1^2,\sigma_2,\sigma_1\right)$ is the result of global conjugation of the latter factorization by $\Delta$. We deduce that ${\cal F}_1$ is Hurwitz equivalent to ${\cal F}_3$. Secondly, we have the following composition of Hurwitz moves \begin{align*}
{\cal F}_1 &\sim \left(\sigma_2,\sigma_2^{-1}\sigma_1^2\sigma_2,\sigma_1,\sigma_2,\sigma_2^{-1}\sigma_1\sigma_2\right) \\ &\sim \left(\sigma_2,\sigma_1,\sigma_2,\left(\sigma_1\sigma_2\right)^{-1}\sigma_2^{-1}\sigma_1^2\sigma_2\left(\sigma_1\sigma_2\right),\sigma_2^{-1}\sigma_1\sigma_2\right) \\ &\sim \left(\sigma_2,\sigma_2,\sigma_2^{-1}\sigma_1\sigma_2,\sigma_2^2,\sigma_2^{-1}\sigma_1\sigma_2\right) \end{align*} and ${\cal F}_4 = \left(\sigma_1,\sigma_1,\sigma_2,\sigma_1^2,\sigma_2\right)$ is the result of global conjugation of the latter factorization by $\Delta\sigma_2$. We deduce that ${\cal F}_1$ is Hurwitz equivalent to ${\cal F}_4$.

We now consider the case where there are two squares of positive half-twists in the factorization ${\cal F}$. The hypotheses that $c\left({\cal F}\right) = 0$ and that ${\cal F}$ is a factorization of $\Delta^2$ imply that there is no adjacent pair of factors in ${\cal F}$ that are squares of positive half-twists. We deduce that either ${\cal F} = \left(\sigma_{i}^2,\sigma_{i'},\sigma_{i}^2,\sigma_{i'}\right)$ or ${\cal F} = \left(\sigma_{i'},\sigma_{i}^2,\sigma_{i'},\sigma_{i}^2\right)$ for some $\{i,i'\} = \{1,2\}$. We can assume after global conjugation of ${\cal F}$ by a power of $\Delta$ that either ${\cal F} = \left(\sigma_1^2,\sigma_2,\sigma_1^2,\sigma_2\right)$ or ${\cal F} = \left(\sigma_2,\sigma_1^2,\sigma_2,\sigma_1^2\right)$. We will show that $\left(\sigma_1^2,\sigma_2,\sigma_1^2,\sigma_2\right)$ is Hurwitz equivalent to $\left(\sigma_2,\sigma_1^2,\sigma_2,\sigma_1^2\right)$. Indeed, $\left(\sigma_1^2,\sigma_2,\sigma_1^2,\sigma_2\right)$ is Hurwitz equivalent to $\left(\sigma_2,\sigma_2^{-1}\sigma_1^2\sigma_2,\sigma_2,\sigma_2^{-1}\sigma_1^2\sigma_2\right)$ after the application of two Hurwitz moves. Finally, global conjugation of the latter factorization by $\sigma_2$ is $\left(\sigma_2,\sigma_1^2,\sigma_2,\sigma_1^2\right)$.

Therefore, the statement is established.
\end{proof} 

We now establish basic constraints on complexity zero factorizations of $\Delta^2$ into powers of positive half-twists. 

\begin{lemma}
\label{lcomplexityzerofactorization}
Let ${\cal F}$ be a factorization of $\Delta^2$ into a product of powers of standard positive half-twists. The $n$th power of a standard positive half-twist cannot be a factor of ${\cal F}$ if $n\geq 3$. Furthermore, the factorization ${\cal F}$ cannot consist of three factors which are all squares of standard positive half-twists.
\end{lemma}
\begin{proof}
If ${\cal F}$ is a factorization of $\Delta^2$ into a product of powers of standard positive half-twists, then each factor of ${\cal F}$ is of the form $\sigma_i^{n}$ for some $i\in \{1,2\}$ and some positive integer $n\geq 1$. However, if $n\geq 3$, then the positive braid $\sigma_i^{n}$ is not a divisor of $\Delta^2$ in the monoid $B_3^{+}$ of positive braids. Indeed, in this case we have $\omega\left(\sigma_i^{n}\right) = n\geq 3$, where $\Delta^{\omega\left(\sigma\right)}$ is the minimal power of $\Delta$ of which $\sigma$ is a divisor in $B_3^{+}$, by definition of duality (Definition~\ref{ddual}). We conclude that the $n$th power of a standard positive half-twist cannot be a factor in ${\cal F}$ if $n\geq 3$.

Finally, ${\cal F}$ cannot consist of three factors which are squares of standard positive half-twists since a product of such squares cannot equal $\Delta^2$. In fact, Proposition~\ref{piDB3} implies that a product of such squares is indivisible by $\Delta$. Therefore, the statement is established.
\end{proof}

We have a complete classification of factorizations of $\Delta^2$ into powers of standard positive half-twists up to Hurwitz equivalence. 

\begin{lemma}
\label{lstandardfactorizationspowersHurwitzequivalent}
If ${\cal F}$ and ${\cal F}'$ are factorizations of $\Delta^2$ into powers of standard positive half-twists, then ${\cal F}$ is Hurwitz equivalent to ${\cal F}'$ if and only if ${\cal F}$ and ${\cal F}'$ have the same number of factors of each type. Furthermore, a factorization of $\Delta^2$ into powers of standard positive half-twists has at most two factors that are squares of standard positive half-twists and no factor that is the $n$th power of a standard positive half-twist for any $n\geq 3$. 
\end{lemma}
\begin{proof}
The statement is an immediate consequence of Lemma~\ref{loneortwosquarespositivehalftwistsstandard} and Lemma~\ref{lcomplexityzerofactorization}.
\end{proof}

We now focus on classifying Hurwitz equivalence classes of factorizations of $\Delta^2$ into positive half-twists and squares of positive half-twists. Afterwards, we will classify Hurwitz equivalence classes of factorizations of $\Delta^2$ into powers of positive half-twists where factors of type $n$ with $n\geq 3$ occur. The strategy of the proof is very similar to that of the classification of Hurwitz equivalence classes of factorizations of $\Delta^2$ into positive half-twists in Subsection~\ref{ssclassificationfactorizationspositivehalftwists}. Indeed, we will establish analogous technical statements in the course of the proof. Firstly, we consider the case of the square of a standard positive half-twist in a minimal complexity factorization. (In Lemma~\ref{ladjacentstandardpositivehalftwist} and Lemma~\ref{ltwoadjacentstandardpositivehalftwists}, we considered standard positive half-twists in a minimal complexity factorization.)

\begin{proposition}
\label{padjacentsquarestandardpositivehalftwist}
We adopt Setup~\ref{setfactorization} with $k=2$, where $g_2$ is the square of a standard positive half-twist. If ${\cal F}$ is a minimal complexity factorization of $g_1g_2$, then either $\xi_1\xi_2$ is indivisible by $\Delta$ or ${\cal F} = \left(\sigma_{i_2'}\sigma_{i_1}^{e_1}\sigma_{i_2'}^{-1},\sigma_{i_2'}^2\right)$, where $i_2\neq i_2'\in \{1,2\}$.
\end{proposition}
\begin{proof}
Let us assume that $\xi_1\xi_2$ is divisible by $\Delta$. In this case, $g_1$ must be the power of a non-standard positive half-twist, since $g_2$ is the square of a standard positive half-twist. We will show that ${\cal F} = \left(\sigma_{i_2'}\sigma_{i_1}^{e_1}\sigma_{i_2'}^{-1},\sigma_{i_2'}^2\right)$.

In Setup~\ref{setfactorization}, we have $\xi_2 = \sigma_{i_2}^2$, and Proposition~\ref{pdualbraidproduct} implies that $\xi_1$ ends with a product $\sigma_{i_2}\sigma_{i_2'}$ of distinct generators, where $i_2\neq i_2'\in \{1,2\}$. In fact, Proposition~\ref{piDB3} implies that $\xi_1$ ends with $\sigma_{i_2}\sigma_{i_2}\sigma_{i_2'}$, since $\xi_1$ is indivisible by $\Delta$. In Setup~\ref{setfactorization}, we also have $\xi_1 = \rho_1\sigma_{i_1}^{e_1}\tau_1$, where $\rho_1,\tau_1\in B_3^{+}$ are non-identity positive braids such that $\rho_1$ is dual to $\tau_1$. 

If $\omega\left(\tau_1\right)>1$, then $\tau_1$ also ends with $\sigma_{i_2}\sigma_{i_2}\sigma_{i_2'}$. Furthermore, Lemma~\ref{lcharacterizationdualpositivebraids} implies that $\rho_1$ begins with $\left(\sigma_{i_2}\sigma_{i_2}\right)_{\left(\omega\left(\tau_1\right)\right)}$. If $\omega\left(\tau_1\right)>1$, then we deduce that the Garside power of $g_2^{-1}g_1g_2$ is at least one less than the Garside power of $g_1$, since $g_1 = \Delta^{-\omega\left(\tau_1\right)}\left(\xi_1\right)_{\left(\omega\left(\tau_1\right)\right)}$ and $g_2 = \left(\sigma_{i_2}^{2}\right)_{\left(\omega\left(\tau_1\right)\right)}$, as in Setup~\ref{setfactorization}. In particular, if $\omega\left(\tau_1\right)>1$, then the complexity of the factorization $\left(g_2,g_2^{-1}g_1g_2\right)$ is at least one less than the complexity of ${\cal F} = \left(g_1,g_2\right)$. We conclude that $\omega\left(\tau_1\right) = 1$ and $\tau_1 = \sigma_{i_2}\sigma_{i_2'}$. We now have $\rho_1 = \sigma_{i_2'}$ and $\xi_1 = \sigma_{i_2'}\sigma_{i_1}^{e_1}\sigma_{i_2}\sigma_{i_2'}$. Therefore, $g_1 = \sigma_{i_2'}\sigma_{i_1}^{e_1}\sigma_{i_2'}^{-1}$ and $g_2 = \sigma_{i_2'}^2$.
\end{proof}

The following strengthening of Lemma~\ref{lminimalcomplexitypositivehalftwists} for squares of non-standard positive half-twists rather than simply non-standard positive half-twists is an important technical step for characterizing minimal complexity factorizations into positive half-twists and squares of positive half-twists.

\begin{lemma}
\label{lminimalcomplexitysquarespositivehalftwists}
We adopt Setup~\ref{setfactorization} with $k=2$ and where $g_1$ and $g_2$ are powers of non-standard positive half-twists. If ${\cal F} = \left(g_1,g_2\right)$ is a minimal complexity factorization of $g_1g_2$, then the following conditions are satisfied:
\begin{description}
\item[(i)] If $g_1$ is either a non-standard positive half-twist or the square of a non-standard positive half-twist, then the maximal right-divisor $\xi_1''$ of $\xi_1$ that is left-dual to a left-divisor of $\overline{\rho_2}$ is a right-divisor of $\overline{\tau_1}$.
\item[(ii)] If $g_2$ is either a non-standard positive half-twist or the square of a non-standard positive half-twist, then the maximal left-divisor $\xi_2''$ of $\xi_2$ that is right-dual to a right-divisor of $\overline{\tau_1}$ is a left-divisor of $\overline{\rho_2}$.
\end{description}
\end{lemma}
\begin{proof}
We prove \textbf{(i)} since the proof of \textbf{(ii)} is similar. If $g_1$ is a non-standard positive half-twist, then the statement is equivalent to Lemma~\ref{lminimalcomplexitypositivehalftwists}~\textbf{(i)}, and thus we assume that $g_1$ is the square of a non-standard positive half-twist. Firstly, Lemma~\ref{lmaximaldivisorcutdivisor} implies that $\xi_1''$ is a cut right-divisor of $\xi_1$. Furthermore, Proposition~\ref{pomeganumbercutdivisor} implies that $\xi_1''$ is a right-divisor of $\overline{\tau_1}$ if and only if $\omega\left(\xi_1''\right)\leq \omega\left(\overline{\tau_1}\right)$. 

Let us assume, for a contradiction, that $\omega\left(\xi_1''\right)>\omega\left(\tau_1\right)$. We will apply Lemma~\ref{lHurwitzmovecomplexitychange} with $e_1 = 2$, using the assumption that ${\cal F}$ is a minimal complexity factorization. Let $\rho_2''$ be the left-divisor of $\overline{\rho_2}$ that is right-dual to $\xi_1''$. If $\omega\left(\xi_1''\right) - \omega\left(\tau_1\right)\geq 2$, then we obtain a contradiction as an immediate consequence of Lemma~\ref{lHurwitzmovecomplexitychange} with $e_1 = 2$. If $\omega\left(\xi_1''\right)-\omega\left(\tau_1\right) = 1$, then we claim that $\rho_2''$ is a cut left-divisor of $\overline{\rho_2}$, which implies that $\epsilon\leq 0$ in the statement of Lemma~\ref{lHurwitzmovecomplexitychange}. Of course, in this case, we also obtain a contradiction as a consequence of Lemma~\ref{lHurwitzmovecomplexitychange} with $e_1 = 2$.

Let us prove the claim that $\rho_2''$ is a cut left-divisor of $\overline{\rho_2}$. Firstly, Proposition~\ref{pcutclosureconjugate} implies that $\xi_1 = \overline{\rho_1}\sigma_{i_1}\overline{\tau_1}$ and $\overline{\rho_1}$ ends with $\sigma_{i_1'}\sigma_{i_1}$, where $i_1\neq i_1'\in \{1,2\}$. In particular, $\xi_1'' = \sigma_{i_1}\overline{\tau_1}$ and $\xi_1''$ begins with a product of two generators with the same index. Lemma~\ref{lcharacterizationdualpositivebraids} now implies that $\rho_2''$ ends with a product of two distinct generators, since $\xi_2''$ is left-dual to $\rho_2''$. Finally, Lemma~\ref{lcutdivisorcharacterization} implies that $\rho_2''$ is a cut left-divisor of $\overline{\rho_2}$.

We conclude that $\omega\left(\xi_1''\right)\leq \omega\left(\overline{\tau_1}\right)$. Therefore, the statement is established. 
\end{proof}

We now establish a strengthening of Lemma~\ref{lboundedcontraction}, where one or both factors are squares of non-standard positive half-twists rather than simply non-standard positive half-twists. The proof is exactly the same as that of Lemma~\ref{lboundedcontraction} but it is based on Lemma~\ref{lminimalcomplexitysquarespositivehalftwists} rather than Lemma~\ref{lminimalcomplexitypositivehalftwists}.

\begin{lemma}
\label{lsquareboundedcontraction}
We adopt Setup~\ref{setfactorization} with $k = 2$ and where $g_j$ is either a non-standard positive half-twist or the square of a non-standard positive half-twist for each $j\in \{1,2\}$. Let us assume that ${\cal F}$ is a minimal complexity factorization of $g_1g_2$.

In this case, the maximal right-divisor $\xi_1''$ of $\xi_1$ that is left-dual to a left-divisor of $\xi_2$ is a right-divisor of the cut closure $\overline{\tau_1}$. Furthermore, the maximal left-divisor $\xi_2''$ of $\xi_2$ that is right-dual to a right-divisor of $\xi_1$ is a left-divisor of the cut closure $\overline{\rho_2}$. 

Finally, if $g_2$ is a positive half-twist, then $\xi_1''$ is not left-dual to $\overline{\rho_2}$, and if $g_1$ is a positive half-twist, then $\xi_2''$ is not right-dual to $\overline{\tau_1}$. 
\end{lemma}
\begin{proof}
The proof of Lemma~\ref{lboundedcontraction} follows through in the case where either (or both) $g_1$ or $g_2$ is the square of a non-standard positive half-twist. The only modification is that every application of Lemma~\ref{lminimalcomplexitypositivehalftwists} is replaced with the application of Lemma~\ref{lminimalcomplexitysquarespositivehalftwists}. 
\end{proof}

We remark that in the statement of Lemma~\ref{lsquareboundedcontraction}, it is possible for $\xi_1''$ to be left-dual to $\overline{\rho_2}$ if $g_2$ is the square of a positive half-twist, and similarly, it is possible for $\xi_2''$ to be right-dual to $\overline{\tau_1}$ if $g_1$ is the square of a positive half-twist. Indeed, if $g_2$ is the square of a positive half-twist, then Proposition~\ref{pcutclosureconjugate} implies that $\overline{\rho_2}$ is not followed by a product of two distinct generators in $\xi_2$, and Proposition~\ref{pmaximaldivisordualcharacterization} does not rule out the possibility that $\xi_1''$ is left-dual to $\overline{\rho_2}$. On the other hand, if $g_2$ is a positive half-twist, then $\overline{\rho_2}$ is followed by a product of two distinct generators, and Proposition~\ref{pmaximaldivisordualcharacterization} implies that $\xi_1''$ is not left-dual to $\overline{\rho_2}$ (as in the proof of Lemma~\ref{lboundedcontraction}).

In fact, this is notable. Otherwise, the proof of the classification of factorizations of $\Delta^2$ into positive half-twists and squares of positive half-twists would be very similar to the proof of the classification of factorizations of $\Delta^2$ into positive half-twists (Theorem~\ref{tfactorizationpositivehalftwistsstandard}). The main differences would be that every application of Lemma~\ref{lboundedcontraction} in the proof of Theorem~\ref{tfactorizationpositivehalftwistsstandard} would be replaced by the application of Lemma~\ref{lsquareboundedcontraction}, and we would have to appeal to Proposition~\ref{padjacentsquarestandardpositivehalftwist} to handle squares of standard positive half-twists in the factorization.

We will use the following technical statement to handle the failure of the more general conclusion in the statement of Lemma~\ref{lsquareboundedcontraction}, and accordingly adapt the proof of Theorem~\ref{tfactorizationpositivehalftwistsstandard} to the present context.

\begin{lemma}
\label{ltwoHurwitzmovessquarepositivehalftwist}
We adopt Setup~\ref{setfactorization} with $k = 3$. Let us assume that $g_j$ is either a non-standard positive half-twist or the square of a non-standard positive half-twist for each $j\in \{1,3\}$, and $g_2$ is the square of a non-standard positive half-twist. If $j\in \{1,2\}$, then let $\xi_j''$ be the maximal right-divisor of $\xi_j$ that is left-dual to a left-divisor of $\xi_{j+1}$. Let us denote the relevant left-divisor of $\xi_{j+1}$ by $\chi_j''$. Let us write $\xi_1 = \xi_1'\xi_1''$ and $\xi_3 = \chi_2''\xi_3'$ for positive braids $\xi_1',\xi_3'\in B_3^{+}$.

Let us assume that $\xi_1''$ is a right-divisor of $\overline{\tau_1}$, $\chi_1''=\overline{\rho_2}$, $\xi_2'' = \overline{\tau_2}$, and $\chi_2''$ is a left-divisor of $\overline{\rho_3}$. Let $\xi_1'''$ be the maximal right-divisor of $\xi_1'$ such that $\xi_1'''$ is left-dual to a left-divisor of $\left(\sigma_{i_2}\right)_{\left(\omega\left(\tau_2\right)\right)}\xi_3'$. Let us denote the relevant left-divisor of $\left(\sigma_{i_2}\right)_{\left(\omega\left(\tau_2\right)\right)}\xi_3'$ by $\left(\sigma_{i_2}\right)_{\left(\omega\left(\tau_2\right)\right)}\xi_3'''$. 

If either the right-divisor $\xi_1'''\xi_1''$ of $\xi_1$ is not a right-divisor of $\overline{\tau_1}$ or the left-divisor $\chi_2''\xi_3'''$ is not a left-divisor of $\overline{\rho_3}$, then the following statements are true:
\begin{description}
\item[(i)] The composition of the Hurwitz move $\left(g_1,g_2,g_3\right)\to \left(g_2,g_2^{-1}g_1g_2,g_3\right)$ followed by a Hurwitz move applied to the second and third factors, does not increase the complexity of the factorization. 
\item[(ii)] The composition of the Hurwitz move $\left(g_1,g_2,g_3\right)\to \left(g_1,g_2g_3g_2^{-1},g_2\right)$ followed by a Hurwitz move applied to the first and second factors, does not increase the complexity of the factorization. 
\end{description}
\end{lemma}
\begin{proof}
We will assume that the right-divisor $\xi_1'''\xi_1''$ of $\xi_1$ is not a right-divisor of $\overline{\tau_1}$ and prove \textbf{(i)}. The proofs of the other cases are similar. (A symmetric version of) Lemma~\ref{lHurwitzmovecomplexitychange} with $e_2 = 2$ implies that the Hurwitz move $\left(g_1,g_2,g_3\right)\to \left(g_2,g_2^{-1}g_1g_2,g_3\right)$ increases the complexity of the factorization by one, since $\xi_1''$ is a cut right-divisor of $\overline{\tau_1}$ and $\chi_1''=\overline{\rho_2}$. However, the assumption that $\xi_1'''\xi_1''$ is not a right-divisor of $\overline{\tau_1}$ implies that the conclusion in the first statement of Lemma~\ref{lsquareboundedcontraction} is not satisfied for the second and third factors in $\left(g_2,g_2^{-1}g_1g_2,g_3\right)$. Indeed, this is a consequence of (a symmetric version of) Lemma~\ref{lHurwitzmoveGarsidenormalform}, which determines the Garside normal form of $g_2^{-1}g_1g_2$. The contrapositive of Lemma~\ref{lsquareboundedcontraction} implies that we can apply a complexity reducing Hurwitz move to the factorization $\left(g_2^{-1}g_1g_2,g_3\right)$. Therefore, the statement is established.
\end{proof}

We will adapt the proof of Theorem~\ref{tfactorizationpositivehalftwistsstandard} to the present context. We will use the following preparatory statement on minimal complexity factorizations of $\Delta^2$ into positive half-twists and squares of positive half-twists, in order to constrain the factors that are squares of positive half-twists.

\begin{theorem}
\label{texistencesquarestandardpositivehalftwist}
We adopt Setup~\ref{setfactorization}. Let us assume that ${\cal F}$ is a minimal complexity factorization of $\Delta^2$ into positive half-twists and squares of positive half-twists, and at least one factor in ${\cal F}$ is the square of a positive half-twist. In this case, for some $1\leq j\leq k$, the factor $g_j$ is the square of a standard positive half-twist and either $\xi_{j-1}\xi_j$ or $\xi_j\xi_{j+1}$ is divisible by $\Delta$.
\end{theorem}
\begin{proof}
Let us assume that the conclusion is false. We will derive a contradiction and the proof is very similar to that of Theorem~\ref{tfactorizationpositivehalftwistsstandard}, with adaptations that we will explain. We can assume that the factorization ${\cal F}$ satisfies the conditions in Lemma~\ref{lreorderingfactorssamecomplexity}, after possibly replacing ${\cal F}$ with another minimal complexity factorization (as in the proof of Theorem~\ref{tfactorizationpositivehalftwistsstandard}). We adopt Setup~\ref{setfactorization} and we have the equation $\Delta^{\omega_k+2} = \xi_1\xi_2\cdots \xi_{k}$. Furthermore, if $1\leq j\leq k$, then we adopt the notation in the proof of Theorem~\ref{tfactorizationpositivehalftwistsstandard} for $\xi_j''$, $\chi_j''$, $\omega_j''$, and $\zeta_j$. However, we define $J\subseteq \{1,2,\dots,k\}$ to be the set of all $1\leq j\leq k$ such that $g_j$ is either a standard positive half-twist or the square of a standard positive half-twist, rather than simply the set of all $1\leq j\leq k$ such that $g_j$ is a standard positive half-twist (the latter is the definition of $J$ in the proof of Theorem~\ref{tfactorizationpositivehalftwistsstandard}). 

We also apply Lemma~\ref{lsquareboundedcontraction} rather than Lemma~\ref{lboundedcontraction} (as in the proof of Theorem~\ref{tfactorizationpositivehalftwistsstandard}) in order to conclude that $\zeta_j$ is not the identity braid for each $1\leq j\leq k$. Let us briefly elaborate on this point. If $1\leq j\leq k$, then Lemma~\ref{lsquareboundedcontraction} implies that $\chi_{j-1}''$ is a left-divisor of $\overline{\rho_j}$ and $\xi_j''$ is a right-divisor of $\overline{\tau_j}$, but we cannot necessarily conclude that $\chi_{j-1}''$ is a proper left-divisor of $\overline{\rho_j}$ in the case $g_j$ is the square of a non-standard positive half-twist. However, if $g_j$ is the square of a non-standard positive half-twist, then Proposition~\ref{pcutclosureconjugate} implies that $\xi_j = \overline{\rho_j}\sigma_{i_j}\overline{\tau_j}$. On the other hand, if $g_j$ is a non-standard positive half-twist, in which case $\xi_j = \overline{\rho_j}\overline{\tau_j}$, then Lemma~\ref{lboundedcontraction} (as in the proof of Theorem~\ref{tfactorizationpositivehalftwistsstandard}) implies that $\chi_{j-1}''$ is a proper left-divisor of $\overline{\rho_j}$. We now conclude that $\zeta_j$ is not the identity braid for every $1\leq j\leq k$, based on the equation $\xi_j = \chi_{j-1}''\zeta_j\xi_j''$ defining $\zeta_j$.

We also have the same equation \[\xi_1\xi_2\cdots\xi_{k} = \left[\prod_{j=1}^{k}\left(\zeta_j\right)_{\left(\omega_j''\right)}\right]\Delta^{\omega_{k}''},\] as in the proof of Theorem~\ref{tfactorizationpositivehalftwistsstandard}. If $\zeta_j$ is not an Artin generator when $g_j$ is the square of a non-standard positive half-twist, then the same proof as that of Theorem~\ref{tfactorizationpositivehalftwistsstandard} shows that this expression is the Garside normal form of $\xi_1\xi_2\cdots \xi_{k}$. 

Indeed, in this case, the assumption that there is no $1\leq j\leq k$ such that $g_j$ is the square of a standard positive half-twist with either $\xi_{j-1}\xi_j$ or $\xi_j\xi_{j+1}$ divisible by $\Delta$, shows that $\prod_{j=1}^{k}\left(\zeta_j\right)_{\left(\omega_j''\right)}$ is indivisible by $\Delta$, using Lemma~\ref{lproductindivisibleDelta}. An important point in this argument is that the index of the last generator in $\left(\zeta_{j-1}\right)_{\left(\omega_{j-1}''\right)}$ is equal to the index of the first generator in $\left(\zeta_j\right)_{\left(\omega_j''\right)}$ if $g_j$ is a non-standard positive half-twist and $\zeta_j$ is an Artin generator.

On the other hand, if $\zeta_j$ is an Artin generator when $g_j$ is the square of a non-standard positive half-twist, then it is possible that the index of $\left(\zeta_j\right)_{\left(\omega_j''\right)}$ is different from the indices of the last generator in $\left(\zeta_{j-1}\right)_{\left(\omega_{j-1}''\right)}$ and the first generator in $\left(\zeta_{j+1}\right)_{\left(\omega_{j+1}''\right)}$. If this is the case, then the product $\left(\zeta_{j-1}\right)_{\left(\omega_{j-1}''\right)}\left(\zeta_j\right)_{\left(\omega_j''\right)}\left(\zeta_{j+1}\right)_{\left(\omega_{j+1}''\right)}$ is divisible by $\Delta$, and thus the product $\prod_{j=1}^{k} \left(\zeta_j\right)_{\left(\omega_j''\right)}$ is also divisible by $\Delta$. We now adjust the proof of Theorem~\ref{tfactorizationpositivehalftwistsstandard} accordingly to handle this case. Indeed, we claim that $\prod_{j=1}^{k} \left(\zeta_j\right)_{\left(\omega_j''\right)}$ is not a power of $\Delta$ in any case, which contradicts the equation $\Delta^{\omega_k+2} = \xi_1\xi_2\cdots \xi_k$. We split the argument according to the number of squares of positive half-twists in the factorization ${\cal F} = \left(g_1,g_2,\dots,g_k\right)$. We have already established the claim (in fact, the stronger statement that $\prod_{j=1}^{k} \left(\zeta_j\right)_{\left(\omega_j''\right)}$ is indivisible by $\Delta$) if $\zeta_j$ is not an Artin generator when $g_j$ is the square of a non-standard positive half-twist. In particular, we will assume that $\zeta_j$ is an Artin generator when $g_j$ is the square of a non-standard positive half-twist, for some $1\leq j\leq k$.

Firstly, we consider the case where there is precisely one square of a positive half-twist in ${\cal F}$, and let $g_j$ be the square of a positive half-twist. The same proof as that of Theorem~\ref{tfactorizationpositivehalftwistsstandard} shows that $\prod_{j'=1}^{j-1} \left(\zeta_{j'}\right)_{\left(\omega_{j'}''\right)}$ and $\prod_{j'=j+1}^{k} \left(\zeta_{j'}\right)_{\left(\omega_{j'}''\right)}$ are indivisible by $\Delta$. Furthermore, as in the proof of Theorem~\ref{tfactorizationpositivehalftwistsstandard}, Proposition~\ref{pmaximaldivisordualcharacterization}~\textbf{(i)} implies that the index of the last generator of $\left(\zeta_{j-2}\right)_{\left(\omega_{j-2}''\right)}$ is the same as the index of the first generator of $\left(\zeta_{j-1}\right)_{\left(\omega_{j-1}''\right)}$, if $\chi_{j-2}'' = \rho_{j-1}\neq \overline{\rho_{j-1}}$. We have assumed that $\zeta_j$ is an Artin generator, and in particular, $\left(\prod_{j'=1}^{j-1} \left(\zeta_{j'}\right)_{\left(\omega_{j'}''\right)}\right)\left(\zeta_j\right)_{\left(\omega_j''\right)}$ is dual to $\prod_{j'=j+1}^{k} \left(\zeta_{j'}\right)_{\left(\omega_{j'}''\right)}$ (i.e., $\left(\prod_{j'=1}^{j-1} \left(\zeta_{j'}\right)_{\left(\omega_{j'}''\right)}\right)\left(\zeta_j\right)_{\left(\omega_j''\right)}$ is indivisible by $\Delta$).  Of course, $\zeta_j$ is not an Artin generator if $j=1$, since $\xi_1 = \zeta_1\xi_1''$ and $\xi_1''$ is a right-divisor of $\overline{\tau_1}$. Let us assume that the square of a positive half-twist $g_j$ is leftmost (i.e., $j$ is minimal) among all factorizations Hurwitz equivalent to ${\cal F}$ with the same complexity as ${\cal F}$. We have $j>1$. If $\zeta_j$ is an Artin generator, then Lemma~\ref{lsquareboundedcontraction} implies that $\chi_{j-1}'' = \overline{\rho_j}$ and $\xi_j'' = \overline{\tau_j}$. 

In particular, the fact that the index of the last generator in $\left(\zeta_{j-2}\right)_{\left(\omega_{j-2}''\right)}$ is equal to the index of the first generator in $\left(\zeta_{j-1}\right)_{\left(\omega_{j-1}''\right)}$ if $\chi_{j-2}''=\rho_{j-1}\neq \overline{\rho_{j-1}}$, the fact that the index of the last generator in $\left(\zeta_{j+1}\right)_{\left(\omega_{j+1}''\right)}$ is equal to the index of the first generator in $\left(\zeta_{j+2}\right)_{\left(\omega_{j+2}''\right)}$ if $\xi_{j+1}'' = \overline{\tau_{j+1}}$ (also a consequence of Proposition~\ref{pmaximaldivisordualcharacterization}), and the fact that $\left(\prod_{j'=1}^{j-1} \left(\zeta_{j'}\right)_{\left(\omega_{j'}''\right)}\right)\left(\zeta_j\right)_{\left(\omega_j''\right)}$ is dual to $\prod_{j'=j+1}^{k} \left(\zeta_{j'}\right)_{\left(\omega_{j'}''\right)}$, imply that the hypothesis of Lemma~\ref{ltwoHurwitzmovessquarepositivehalftwist} is satisfied for the triple $\left(g_{j-1},g_j,g_{j+1}\right)$. 

We deduce that there is a composition of Hurwitz moves that transfers the square of a positive half-twist $g_j$ in the factorization to the $\left(j-1\right)$th position and does not change the complexity of the factorization. Of course, this contradicts the assumption that the square of a positive half-twist $g_j$ is left-most, and the statement is established in the case where there is precisely one square of a positive half-twist in ${\cal F}$.

Secondly, we consider the case where there are two squares of positive half-twists in ${\cal F}$. Proposition~\ref{pnumberoffactorsinfactorization} implies that $k=4$. Furthermore, in this case, neither $\zeta_1$ nor $\zeta_4$ is an Artin generator, unless either $g_1$ or $g_4$ is a standard positive half-twist, respectively (since we are assuming for a contradiction that the conclusion is false). However, we have assumed that $\zeta_j$ is an Artin generator when $g_j$ is the square of a non-standard positive half-twist, for some $1\leq j\leq k$. In particular, the same argument as the previous case shows that if at least one of the squares of a positive half-twist is at an end of ${\cal F}$, then we can transfer the other square of a positive half-twist $g_j$ to the other end of ${\cal F}$ by a composition of Hurwitz moves that does not change the complexity of the factorization (using the fact that $\zeta_j$ is an Artin generator, by our assumption). However, in the case where both squares of positive half-twists in ${\cal F}$ are at the ends of ${\cal F}$, we have a contradiction since neither $\zeta_1$ nor $\zeta_4$ can be an Artin generator. 

Let us now assume that both squares of positive half-twists in ${\cal F}$ are in the middle of ${\cal F}$. If $g_j$ is the power of a non-standard positive half-twist for each $j\in \{1,2,3,4\}$ and if $\xi_j\xi_{j+1}$ is divisible by $\Delta$ for each $j\in \{1,2,3\}$, then Lemma~\ref{lcomplexityreducingglobalconjugation} implies that we can apply a complexity reducing global conjugation move to ${\cal F}$, which contradicts the assumption that ${\cal F}$ is a minimal complexity factorization. However, in all other cases, the proof of Theorem~\ref{tfactorizationpositivehalftwistsstandard} follows through, using our assumption at the beginning of the proof that the conclusion is false. Indeed, if $\xi_2\xi_3$ is indivisible by $\Delta$, then $\zeta_j$ is not an Artin generator for every $1\leq j\leq 4$, which is a contradiction. Furthermore, if $g_1$ is a standard positive half-twist, then Lemma~\ref{ladjacentstandardpositivehalftwist} implies that $\xi_1\xi_2$ is indivisible by $\Delta$, and if $g_4$ is a standard positive half-twist, then Lemma~\ref{ladjacentstandardpositivehalftwist} implies that $\xi_3\xi_4$ is indivisible by $\Delta$. Thus, using our assumption that the conclusion is false, we deduce that both $g_2$ and $g_3$ are squares of non-standard positive half-twists, and either $\xi_1\xi_2$ or $\xi_3\xi_4$ is indivisible by $\Delta$. 

If $\xi_1\xi_2$ is indivisible by $\Delta$, then $\zeta_2$ is not an Artin generator, and we deduce that $\zeta_3$ is an Artin generator. In this case, Lemma~\ref{ltwoHurwitzmovessquarepositivehalftwist} implies that we can transfer $g_2$ to the right end of ${\cal F}$ by a composition of two Hurwitz moves without changing the complexity of ${\cal F}$. However, we have already addressed the case where one of the squares of a positive half-twist is at the end of the factorization, and obtained a contradiction. A similar argument applies if $\xi_3\xi_4$ is indivisible by $\Delta$, and the statement is established in the case where there are precisely two squares of positive half-twists in ${\cal F}$.

Finally, we consider the case where there are three squares of positive half-twists in ${\cal F}$. In this case, Proposition~\ref{pnumberoffactorsinfactorization} implies that $k = 3$. Firstly, we will show that either $\xi_1\xi_2$ or $\xi_2\xi_3$ is indivisible by $\Delta$. Indeed, we assume for a contradiction that both $\xi_1\xi_2$ and $\xi_2\xi_3$ are divisible by $\Delta$. In this case, our assumption that the conclusion is false at the beginning of the proof implies that $g_j$ is the square of a non-standard positive half-twist for each $j\in \{1,2,3\}$. However, if $g_j$ is the square of a non-standard positive half-twist for each $j\in \{1,2,3\}$, then Lemma~\ref{lcomplexityreducingglobalconjugation} implies that we can apply a complexity reducing global conjugation move to ${\cal F}$, which contradicts the assumption that ${\cal F}$ is a minimal complexity factorization. We deduce that either $\xi_1\xi_2$ or $\xi_2\xi_3$ is indivisible by $\Delta$. If $\xi_1\xi_2$ is indivisible by $\Delta$, then $\xi_1\xi_2\xi_3 = \xi_1\zeta_2\left(\zeta_3\right)_{\left(\omega_3''\right)}\Delta^{\omega\left(\xi_2''\right)}$ is the Garside normal form of $\xi_1\xi_2\xi_3$, which contradicts the fact that $\xi_1\xi_2\xi_3$ is a power of $\Delta$. A similar argument applies if $\xi_2\xi_3$ is indivisible by $\Delta$, and the statement is established in the case where there are precisely three squares of positive half-twists in ${\cal F}$.

Of course, Proposition~\ref{pnumberoffactorsinfactorization} implies that the number of squares of positive half-twists in ${\cal F}$ is an element of $\{1,2,3\}$. Therefore, we have arrived at a contradiction in all cases, and the statement is established.
\end{proof}

We now establish that the conclusion in Theorem~\ref{texistencesquarestandardpositivehalftwist} is sufficient in order to uniquely constrain the Hurwitz equivalence class of a factorization of $\Delta^2$ into positive half-twists and squares of positive half-twists.

\begin{theorem}
\label{texistencesquarestandardpositivehalftwistimpliesstandard}
We adopt Setup~\ref{setfactorization} and assume that ${\cal F}$ is a factorization of $\Delta^2$ into positive half-twists and squares of positive half-twists. Furthermore, let us assume that for some $1\leq j\leq k$, the factor $g_j$ is the square of a standard positive half-twist and either $\xi_{j-1}\xi_j$ or $\xi_j\xi_{j+1}$ is divisible by $\Delta$. In this case, ${\cal F}$ is Hurwitz equivalent to the standard factorization (in Definition~\ref{dstandardfactorization}) with the same number of factors of each type as ${\cal F}$.
\end{theorem}
\begin{proof}
If no factor in ${\cal F}$ is the square of a positive half-twist, then the statement is a consequence of Theorem~\ref{tfactorizationpositivehalftwistsstandard}. Let us assume that there is at least one factor in ${\cal F}$ that is the square of a positive half-twist. If precisely one factor in ${\cal F}$ is the square of a positive half-twist, then we will show that ${\cal F}$ is Hurwitz equivalent to $\left(\sigma_1^2,\sigma_2,\sigma_1,\sigma_1,\sigma_2\right)$. If precisely two factors in ${\cal F}$ are the squares of positive half-twists, then we will show that ${\cal F}$ is Hurwitz equivalent to $\left(\sigma_1^2,\sigma_2,\sigma_1^2,\sigma_2\right)$. If precisely three factors in ${\cal F}$ are the squares of positive half-twists, then we will show that ${\cal F}$ is Hurwitz equivalent to $\left(\sigma_1^2,\sigma_1^{-1}\sigma_2^2\sigma_1,\sigma_2^2\right)$. Of course, Proposition~\ref{pnumberoffactorsinfactorization} implies that these are all the possibilities for the number of squares of positive half-twists in the factorization ${\cal F}$.

Let us assume, for a contradiction, that ${\cal F}$ is a factorization of minimal complexity satisfying the hypothesis but is not Hurwitz equivalent to a standard factorization. Lemma~\ref{lstandardfactorizationspowersHurwitzequivalent} implies that the complexity $c\left({\cal F}\right)>0$. We consider the factor $g_j$ in ${\cal F}$ that is the square of a standard positive half-twist, and where either $\xi_{j-1}\xi_j$ or $\xi_j\xi_{j+1}$ is divisible by $\Delta$. We can write $g_j = \sigma_1^2$ after applying a global conjugation move by a power of $\Delta$ to ${\cal F}$. Let us assume, without loss of generality, that $\xi_j\xi_{j+1}$ is divisible by $\Delta$, since the proof in the other case is analogous. In this case, $g_{j+1}$ is the power of a non-standard positive half-twist. We consider the cases where $g_{j+1}$ is a non-standard positive half-twist and where $g_{j+1}$ is the square of a non-standard positive half-twist separately.

Firstly, we consider the case where $g_{j+1}$ is a non-standard positive half-twist. In this case, Proposition~\ref{padjacentsquarestandardpositivehalftwist} implies that $g_{j+1}=\sigma_{1}^{-1}\sigma_{2}\sigma_1$ and the product $g_jg_{j+1} = \Delta$. In particular, Proposition~\ref{pfactorizationmoveGarsideelement} implies that we can move this pair $\left(g_j,g_{j+1}\right)$ of adjacent factors to the beginning of the factorization ${\cal F}$, by a finite sequence of Hurwitz moves that does not change the complexity of the factorization. Thus, we can assume that $j=1$ and ${\cal F} = \left(\sigma_1^2,\sigma_1^{-1}\sigma_2\sigma_1,\dots\right)$. We now obtain a factorization ${\cal F}'$ of $\Delta$ into positive half-twists and squares of positive half-twists by considering all but the first two factors in ${\cal F}$. We consider cases according to the numbers of factors of each type in ${\cal F}'$.

If there are three positive half-twists in ${\cal F}'$, then ${\cal F}' = \left(g_3,g_4,g_5\right)$ and Theorem~\ref{tfactorizationpositivehalftwistsstandard} implies that ${\cal F}'$ is Hurwitz equivalent to the standard factorization $\left(\sigma_1,\sigma_2,\sigma_1\right)$ of $\Delta$. We deduce that ${\cal F}$ is Hurwitz equivalent to the factorization $\left(\sigma_1^2,\sigma_1^{-1}\sigma_2\sigma_1,\sigma_1,\sigma_2,\sigma_1\right)$ of $\Delta^2$. The application of another Hurwitz move to the fourth and fifth factors in this latter factorization shows that ${\cal F}$ is Hurwitz equivalent to $\left(\sigma_1^2,\sigma_1^{-1}\sigma_2\sigma_1,\sigma_1,\sigma_1,\sigma_1^{-1}\sigma_2\sigma_1\right)$. Finally, the application of a global conjugation move by $\sigma_1$ shows that ${\cal F}$ is Hurwitz equivalent to the standard factorization $\left(\sigma_1^2,\sigma_2,\sigma_1,\sigma_1,\sigma_2\right)$, which is a contradiction.

If there is one square of a positive half-twist in ${\cal F}'$, then ${\cal F}' = \left(g_3,g_4\right)$ is a factorization of $\Delta$ into a square of a positive half-twist and a positive half-twist, according to Proposition~\ref{pnumberoffactorsinfactorization}. Furthermore, ${\cal F}'$ is a minimal complexity factorization of $\Delta$. Proposition~\ref{padjacentsquarestandardpositivehalftwist} now implies that ${\cal F}' = \left(\sigma_i^2,\sigma_i^{-1}\sigma_{i'}\sigma_i\right)$, where $i\neq i'\in \{1,2\}$, after possibly applying a Hurwitz move to transpose the factors in ${\cal F}'$ (depending on whether $g_3$ or $g_4$ is the square of a standard positive half-twist). In particular, ${\cal F} = \left(\sigma_1^2,\sigma_1^{-1}\sigma_2\sigma_1,\sigma_i^2,\sigma_i^{-1}\sigma_{i'}\sigma_i\right)$.

If $i=1$, then global conjugation of ${\cal F}$ by $\sigma_1$ shows that ${\cal F}$ is Hurwitz equivalent to the standard factorization $\left(\sigma_1^2,\sigma_2,\sigma_1^2,\sigma_2\right)$, which is a contradiction. If $i = 2$, then ${\cal F} = \left(\sigma_1^2,\sigma_1^{-1}\sigma_2\sigma_1,\sigma_2^2,\sigma_2^{-1}\sigma_1\sigma_2\right)$. We have the following finite sequence of Hurwitz moves applied to the factorization ${\cal F}$: \begin{align*} {\cal F} &\sim \left(\sigma_1^2,\sigma_1^{-1}\sigma_2\sigma_1,\sigma_2\sigma_1\sigma_2^{-1},\sigma_2^2\right) \\ &\sim \left(\sigma_1^2,\sigma_1^{-1}\sigma_2\sigma_1,\sigma_1^{-1}\sigma_2\sigma_1,\sigma_2^2\right) \\ &\sim \left(\sigma_1^2,\sigma_1^{-1}\sigma_2\sigma_1,\sigma_1^2,\sigma_1^{-1}\sigma_2\sigma_1\right). \end{align*} The first Hurwitz equivalence is by the application of a Hurwitz move to the third and fourth factors in ${\cal F}$. The second Hurwitz equivalence is by the application of a Hurwitz move to the second and third factors in the second factorization. The third Hurwitz equivalence is by the application of a Hurwitz move to the third and fourth factors in the third factorization. Finally, the application of a global conjugation move by $\sigma_1$ to the fourth factorization shows that ${\cal F}$ is Hurwitz equivalent to $\left(\sigma_1^2,\sigma_2,\sigma_1^2,\sigma_2\right)$, which is a contradiction. We have considered all possibilities for ${\cal F}$ if $g_{j+1}$ is a non-standard positive half-twist.

Secondly, we consider the case where $g_{j+1}$ is the square of a non-standard positive half-twist. In this case, Proposition~\ref{padjacentsquarestandardpositivehalftwist} implies that $g_{j+1} = \sigma_1^{-1}\sigma_2^2\sigma_1$. Of course, Proposition~\ref{pnumberoffactorsinfactorization} implies that there are either three squares of positive half-twists in ${\cal F}$, or two squares of positive half-twists in ${\cal F}$ ($g_j$ and $g_{j+1}$) and two positive half-twists in ${\cal F}$. We consider these subcases separately. If there are three squares of positive half-twists in ${\cal F}$, then the equation $g_jg_{j+1} = \sigma_1\sigma_2^2\sigma_1$ uniquely constrains the third factor since $\Delta^2=g_1g_2g_3$. Indeed, if $j=1$, then the third factor $g_3 = \sigma_2^2$ and ${\cal F} = \left(\sigma_1^2,\sigma_1^{-1}\sigma_2^2\sigma_1,\sigma_2^2\right)$ is a standard factorization. If $j=2$, then the third factor $g_1 = \sigma_2^2$ and ${\cal F} = \left(\sigma_2^2,\sigma_1^2,\sigma_1^{-1}\sigma_2^2\sigma_1\right)$. However, in this case, a Hurwitz move applied to the first and second factors in ${\cal F}$ shows that ${\cal F}$ is Hurwitz equivalent to $\left(\sigma_1^2,\sigma_1^{-2}\sigma_2^2\sigma_1^{2},\sigma_1^{-1}\sigma_2^2\sigma_1\right)$. If we apply a global conjugation move by $\sigma_1$ to the latter factorization, then we observe that ${\cal F}$ is Hurwitz equivalent to the standard factorization $\left(\sigma_1^2,\sigma_1^{-1}\sigma_2^2\sigma_1,\sigma_2^2\right)$, which is a contradiction.

On the other hand, if there are two squares of positive half-twists in ${\cal F}$, then ${\cal F}$ has four factors, where the two squares of positive half-twists are $g_j$ and $g_{j+1}$. We consider the subcases according to whether the two squares of positive half-twists in ${\cal F}$ are at an end of ${\cal F}$ or in the middle of ${\cal F}$, separately. If the two squares of positive half-twists are at an end of ${\cal F}$, then we can assume without loss of generality that $j=3$, in which case $g_3 = \sigma_1^2$ and $g_4 = \sigma_1^{-1}\sigma_2^2\sigma_1$ (the proof is similar if $j=1$). If $j=3$, and if $g_1$ and $g_2$ are standard positive half-twists, then $\xi_1 = \sigma_2 = \xi_2$ since $g_3g_4 = \sigma_1\sigma_2^2\sigma_1$ and $g_1g_2g_3g_4 = \Delta^2$. In this case, ${\cal F} = \left(\sigma_2,\sigma_2,\sigma_1^2,\sigma_1^{-1}\sigma_2^2\sigma_1\right)$ is Hurwitz equivalent to $\left(\sigma_2,\sigma_2\sigma_1^2\sigma_2^{-1},\sigma_2,\sigma_1^{-1}\sigma_2^2\sigma_1\right)$ by a Hurwitz move applied to the second and third factors. However, a global conjugation move by $\sigma_2^{-1}$ applied to the latter factorization shows that ${\cal F}$ is Hurwitz equivalent to $\left(\sigma_2,\sigma_1^2,\sigma_2,\sigma_1^2\right)$ since $\sigma_2^{-1}\sigma_1^{-1}\sigma_2^2\sigma_1\sigma_2 = \sigma_2^{-1}\sigma_1^{-1}\Delta\sigma_1 = \sigma_1^2$. Of course, this is a contradiction, and we conclude that at least one of $g_1$ or $g_2$ is a non-standard positive half-twist. In particular, $g_2\neq \sigma_2$, since otherwise the equation $g_1g_2g_3g_4 = \Delta^2$ would imply that $g_1 = \sigma_2$.

If $\xi_2\xi_3$ is divisible by $\Delta$, then we are in a situation we have already considered. Indeed, in this case $g_2$ and $g_3$ is an adjacent pair of factors where one factor is a non-standard positive half-twist ($g_2$) and the other factor is the square of a standard positive half-twist ($g_3$). In particular, we can assume that $\xi_2\xi_3$ is indivisible by $\Delta$. We have $\xi_3\xi_4 = \left(\sigma_1\sigma_2^2\sigma_1\right)_{\left(\omega\left(\tau_1\right)+\omega\left(\tau_2\right)\right)}\Delta$ (in the notation of Setup~\ref{setfactorization}). We claim that $\xi_2\left(\sigma_1\sigma_2^2\sigma_1\right)_{\omega\left(\tau_1\right)+\omega\left(\tau_2\right)}$ is indivisible by $\Delta$. If not, then $\xi_2$ ends with $\left(\sigma_2\right)_{\left(\omega\left(\tau_1\right)+\omega\left(\tau_2\right)\right)}$. Of course, $\xi_2\neq \left(\sigma_2\right)_{\left(\omega\left(\tau_1\right)+\omega\left(\tau_2\right)\right)}$, since otherwise $g_2 = \sigma_2$ (in the notation of Setup~\ref{setfactorization}), which we have ruled out. Also, $\xi_2$ does not end with $\left(\sigma_1\sigma_2\right)_{\left(\omega\left(\tau_1\right)+\omega\left(\tau_2\right)\right)}$ since $\xi_2\xi_3$ is indivisible by $\Delta$ (and $\xi_3 = \left(\sigma_1^2\right)_{\left(\omega\left(\tau_1\right)+\omega\left(\tau_2\right)\right)}$ in the notation of Setup~\ref{setfactorization}). We conclude that $\xi_2$ ends with $\left(\sigma_2\sigma_2\right)_{\left(\omega\left(\tau_1\right)+\omega\left(\tau_2\right)\right)}$, and Proposition~\ref{pcutclosureconjugate} implies that $\overline{\tau_2}$ also ends with $\left(\sigma_2\sigma_2\right)_{\left(\omega\left(\tau_1\right)+\omega\left(\tau_2\right)\right)}$. In particular, Lemma~\ref{lcutdivisorcharacterization} implies that the $\left(\sigma_2\right)_{\left(\omega\left(\tau_1\right)+\omega\left(\tau_2\right)\right)}$ at the end of $\overline{\tau_2}$ is a cut right-divisor of $\overline{\tau_2}$, and Proposition~\ref{pdivisorconjugatedual} implies that $\overline{\rho_2}$ begins with $\left(\sigma_1\sigma_2\right)_{\left(\omega\left(\tau_1\right)\right)}$. We deduce that the absolute value of the Garside power of $\left(\sigma_1\sigma_2^2\sigma_1\right)^{-1}g_2\left(\sigma_1\sigma_2^2\sigma_1\right)$ is strictly less than the absolute value of the Garside power of $g_2$, since $g_2 = \Delta^{-\omega\left(\tau_2\right)}\left(\xi_2\right)_{\left(\omega\left(\tau_1\right)+\omega\left(\tau_2\right)\right)}$ and $\omega\left(\sigma_1\sigma_2^2\sigma_1\right) = 2$. 

Of course, $g_3g_4 = \sigma_1\sigma_2^2\sigma_1$. We conclude that the composition of two Hurwitz moves $\left(g_1,g_2,g_3,g_4\right)\to \left(g_1,g_3,g_4,g_4^{-1}g_3^{-1}g_2g_3g_4\right)$ reduces the complexity of the factorization, which contradicts the hypothesis that ${\cal F}$ is a minimal complexity factorization. Thus, our claim is true and $\xi_2\left(\sigma_1\sigma_2^2\sigma_1\right)_{\left(\omega\left(\tau_1\right)+\omega\left(\tau_2\right)\right)}$ is indivisible by $\Delta$. 

Let $\xi_1''$ be the maximal right-divisor of $\xi_1$ that is left-dual to a left-divisor of $\xi_2$. Let $\chi_1''$ be the relevant left-divisor of $\xi_2$ and write $\xi_1 = \xi_1'\xi_1''$ and $\xi_2 = \chi_1''\chi_1'$. Lemma~\ref{lboundedcontraction} implies that $\xi_1''$ is a right-divisor of $\overline{\tau_1}$ and $\chi_1''$ is a left-divisor of $\overline{\rho_2}$. We conclude that the Garside normal form of $\xi_1\xi_2\xi_3\xi_4$ is $\xi_1\xi_2\xi_3\xi_4 = \xi_1'\left(\chi_1'\right)_{\left(\omega\left(\xi_1''\right)\right)}\left(\sigma_1\sigma_2^2\sigma_1\right)_{\left(\omega\left(\tau_1\right)+\omega\left(\tau_2\right)+\omega\left(\xi_1''\right)\right)}\Delta^{\omega\left(\xi_1''\right)+1}$, which contradicts the hypothesis that $\xi_1\xi_2\xi_3\xi_4$ is a power of $\Delta$. Thus, the two squares of positive half-twists are not at an end of ${\cal F}$. 

However, if the two squares of positive half-twists are in the middle of ${\cal F}$, then $j=2$, in which case $g_2 = \sigma_1^2$ and $g_3 = \sigma_1^{-1}\sigma_2^2\sigma_1$. We will arrive at a contradiction in this subcase by an argument that is only slightly different from the subcase $j=3$ that we have just considered. 

If $\xi_1\xi_2$ is divisible by $\Delta$, then we are in a situation that we have already considered. Indeed, in this case $g_1$ and $g_2$ is an adjacent pair of factors where one factor is a non-standard positive half-twist ($g_1$) and the other factor is the square of a standard positive half-twist ($g_2$). In particular, we can assume that $\xi_1\xi_2$ is indivisible by $\Delta$. In this case, $\xi_1\xi_2\xi_3 = \xi_1\left(\sigma_1\sigma_2^2\sigma_1\right)_{\left(\omega\left(\tau_1\right)\right)}\Delta$. Furthermore, $\xi_1\left(\sigma_1\sigma_2^2\sigma_1\right)_{\left(\omega\left(\tau_1\right)\right)}\in B_3^{+}$ is a positive braid indivisible by $\Delta$ by the same argument as in the subcase $j=3$ that we have just considered.

If $\xi_3\xi_4$ is indivisible by $\Delta$, then the Garside normal form of $\xi_1\xi_2\xi_3\xi_4$ is $\xi_1\left(\sigma_1\sigma_2^2\sigma_1\right)_{\left(\omega\left(\tau_1\right)\right)}\left(\xi_4\right)_{\left(1\right)}\Delta$, which contradicts the hypothesis that $\xi_1\xi_2\xi_3\xi_4$ is a power of $\Delta$. In particular, we can assume that $\xi_3\xi_4$ is divisible by $\Delta$, and Lemma~\ref{ladjacentstandardpositivehalftwist} implies that $g_4$ is not a standard positive half-twist, since ${\cal F}$ is a minimal complexity factorization. Let $\xi_3''$ be the maximal right-divisor of $\xi_3$ that is left-dual to a left-divisor of $\xi_4$. Let $\chi_3''$ be the relevant left-divisor of $\xi_4$ and write $\xi_3 = \xi_3'\xi_3''$ and $\xi_4 = \chi_3''\chi_3'$. Lemma~\ref{lsquareboundedcontraction} implies that $\xi_3''$ is a right-divisor of $\overline{\tau_3}$ and $\chi_3''$ is a left-divisor of $\overline{\rho_4}$. Lemma~\ref{lmaximaldivisorcutdivisor} implies that $\xi_3''$ is a cut right-divisor of $\xi_3$, and since $\xi_3''$ is not the identity braid, we deduce that $\xi_3'' = \left(\sigma_2\sigma_1\right)_{\left(\omega\left(\tau_1\right)+1\right)}$ (in the notation of Setup~\ref{setfactorization}). Furthermore, in this case, the Garside normal form of $\xi_1\xi_2\xi_3\xi_4$ is $\xi_1\left(\sigma_1\sigma_2\right)_{\left(\omega\left(\tau_1\right)\right)}\chi_4'\Delta^2$, which also contradicts the hypothesis that $\xi_1\xi_2\xi_3\xi_4$ is a power of $\Delta$. Thus, the two squares of positive half-twists are not in the middle of ${\cal F}$.

Therefore, we have arrived at a contradiction in all cases, and the statement is established.
\end{proof}

Let us summarize the results in this subsection. 

\begin{theorem}
\label{tmainclassificationfactorizationsquarespositivehalftwists}
If ${\cal F}$ is a factorization of $\Delta^2$ into positive half-twists and squares of positive half-twists, then ${\cal F}$ is Hurwitz equivalent to the standard factorization (in Definition~\ref{dstandardfactorization}) with the same number of factors of each type as ${\cal F}$.
\end{theorem}
\begin{proof}
The statement is an immediate consequence of Theorem~\ref{texistencesquarestandardpositivehalftwist} and Theorem~\ref{texistencesquarestandardpositivehalftwistimpliesstandard}.
\end{proof} 

We state the geometric reformulation of Theorem~\ref{tmainclassificationfactorizationsquarespositivehalftwists}.

\begin{theorem}
\label{tuniqueisotopyclassnumberofnodes}
If $C,C'\subseteq \mathbb{CP}^2$ are degree three nodal Hurwitz curves, then $C$ is isotopic to $C'$ if and only if the number of nodes in $C$ is equal to the number of nodes in $C'$.
\end{theorem}
\begin{proof}
The braid monodromy factorization of a degree three nodal Hurwitz curve is a factorization of $\Delta^2$ into positive half-twists and squares of positive half-twists. Furthermore, the Hurwitz equivalence class of the braid monodromy factorization of the curve completely determines the isotopy class of the curve. 

The number of nodes in a nodal Hurwitz curve is equal to the number of factors in the braid monodromy factorization of the curve that are squares of positive half-twists. In particular, Theorem~\ref{tmainclassificationfactorizationsquarespositivehalftwists} implies that the braid monodromy factorization of $C$ is Hurwitz equivalent to the braid monodromy factorization of $C'$ if and only if the number of nodes in $C$ is equal to the number of nodes in $C'$. Therefore, the statement is established.
\end{proof}

The next step in this subsection is to show that there is a unique isotopy class of degree three simple Hurwitz curves in $\mathbb{CP}^2$ with precisely one cusp and no other singularities. In Subsection~\ref{sssingularityconstraints}, we will show that a degree three simple Hurwitz curve with at least one cusp has precisely one cusp and no other singularities. In particular, this will complete the classification of isotopy classes of degree three cuspidal Hurwitz curves in $\mathbb{CP}^2$.

We will prove that there is a unique Hurwitz equivalence class of factorizations of $\Delta^2$ into one cube of a positive half-twist and three positive half-twists. Firstly, we establish a characterization of cubes of standard positive half-twists in a minimal complexity factorization.

\begin{proposition}
\label{padjacentcubestandardpositivehalftwist}
We adopt Setup~\ref{setfactorization} with $k=2$, where $g_1$ is a non-standard positive half-twist and $g_2$ is the cube of a standard positive half-twist. If ${\cal F}$ is a minimal complexity factorization of $g_1g_2$, then precisely one of the following conditions is satisfied:
\begin{description}
\item[(i)] If $\xi_1''$ is the maximal right-divisor of $\xi_1$ that is left-dual to a left-divisor of $\xi_2$, then $\xi_1''$ is a right-divisor of $\overline{\tau_1}$ and $\omega\left(\xi_1''\right)\leq 1$. 
\item[(ii)] The factorization ${\cal F} = \left(\sigma_{i_2'}\sigma_{i_1}\sigma_{i_2'}^{-1},\sigma_{i_2'}^3\right)$, where $i_2\neq i_2'\in \{1,2\}$, and $g_1g_2 = \sigma_{i_2}\Delta$.
\end{description}
\end{proposition}
\begin{proof}
Of course, Lemma~\ref{lmaximaldivisorcutdivisor} implies that $\xi_1''$ is a cut right-divisor of $\xi_1$. Let us assume that $\xi_1''$ is not a right-divisor of $\overline{\tau_1}$. In this case, Proposition~\ref{pomeganumbercutdivisor} implies that $\omega\left(\xi_1''\right)> \omega\left(\overline{\tau_1}\right)$. We have $\xi_2 = \sigma_{i_2}^{3}$ (in the notation of Setup~\ref{setfactorization}), and in particular, $\omega\left(\xi_1''\right)\leq \omega\left(\xi_2\right) = 3$. Proposition~\ref{pcutclosureconjugate} implies that $\xi_1 = \overline{\rho_1}\overline{\tau_1}$, where $\rho_1,\tau_1\in B_3^{+}$ are nonidentity positive braids such that $\rho_1$ is left-dual to $\tau_1$, since $g_1$ is a non-standard positive half-twist. We deduce that $\omega\left(\xi_1''\right)>1$ since $\omega\left(\overline{\tau_1}\right)\geq 1$. Furthermore, if $\omega\left(\overline{\tau_1}\right) = 1$, then $\overline{\tau_1} = \sigma_{i_2}\sigma_{i_2'}$, which implies that $\xi_1 = \sigma_{i_2'}\sigma_{i_2}\sigma_{i_2}\sigma_{i_2'}=\xi_1''$ and $\omega\left(\xi_1''\right) = 2$.

In particular, if $\omega\left(\xi_1''\right) = 3$, then $\omega\left(\overline{\tau_1}\right) = 2$. However, this implies that $\overline{\tau_1} = \sigma_{i_2'}\sigma_{i_2}\sigma_{i_2}\sigma_{i_2'}$, since $\xi_1''$ is left-dual to $\sigma_{i_2}^{3}$ in this case. Lemma~\ref{lcharacterizationdualpositivebraids} and Proposition~\ref{pdivisorconjugatedual} imply that $\rho_1$ begins with $\sigma_{i_2}^2$. However, we have $g_1 = \Delta^{-\omega\left(\tau_1\right)}\left(\rho_1\sigma_{i_1}\tau_1\right)_{\left(\omega\left(\tau_1\right)\right)} = \Delta^{-2}\rho_1\sigma_{i_1}\tau_1$ and $g_2 = \left(\sigma_{i_2}\right)_{\left(\omega\left(\tau_1\right)\right)}^{3} = \sigma_{i_2}^3$ (in the notation of Setup~\ref{setfactorization}). In particular, the absolute value of the Garside power of $g_2^{-1}g_1g_2$ is strictly less than the absolute value of the Garside power of $g_1$. We conclude that the Hurwitz move $\left(g_1,g_2\right)\to \left(g_2,g_2^{-1}g_1g_2\right)$ reduces the complexity of the factorization, which contradicts the hypothesis that ${\cal F}$ is a minimal complexity factorization.

The only possibility is $\omega\left(\xi_1''\right) = 2$. Furthermore, if $\omega\left(\xi_1''\right) = 2$, then $\omega\left(\overline{\tau_1}\right) = 1$, which implies that $\overline{\tau_1} = \sigma_{i_2}\sigma_{i_2'}$. However, in this case, $\xi_1 = \sigma_{i_2'}\sigma_{i_2}\sigma_{i_2}\sigma_{i_2'} = \xi_1''$, which implies that $g_1 = \Delta^{-1}\left(\xi_1\right)_{\left(1\right)} = \sigma_{i_2'}\sigma_{i_2}\sigma_{i_2'}^{-1}$ and $g_2 = \sigma_{i_2'}^{3}$. We conclude that either $\xi_1''$ is a right-divisor of $\overline{\tau_1}$ or \textbf{(ii)} holds. 

Finally, if $\xi_1''$ is a right-divisor of $\overline{\tau_1}$, and if $\omega\left(\xi_1''\right)>1$, then $\sigma_{i_2'}\sigma_{i_2}\sigma_{i_2}\sigma_{i_2'}$ is a right-divisor of $\overline{\tau_1}$. A similar argument to the second paragraph of this proof shows that the Hurwitz move $\left(g_1,g_2\right)\to \left(g_2,g_2^{-1}g_1g_2\right)$ reduces the complexity of the factorization, which contradicts the hypothesis that ${\cal F}$ is a minimal complexity factorization. Therefore, the statement is established.
\end{proof}

Proposition~\ref{padjacentcubestandardpositivehalftwist} locally constrains a minimal complexity factorization of $\Delta^2$ into one cube of a standard positive half-twist and three positive half-twists. The following technical statement will allow us to use this local constraint in order to globally constrain such a factorization.

\begin{proposition}
\label{pcubestandardpositivehalftwistHurwitzequivalenceclass}
We adopt Setup~\ref{setfactorization} with $k=4$. Let us assume that $g_j$ is a non-standard positive half-twist and $g_{j+1}$ is the cube of a standard positive half-twist, such that the product $g_jg_{j+1} = \left(\sigma_{i_{j+1}'}\right)_{\left(\omega_{j+1}\right)}\Delta$, where $i_{j+1}\neq i_{j+1}'\in \{1,2\}$ (i.e., Proposition~\ref{padjacentcubestandardpositivehalftwist}~\textbf{(ii)} is satisfied for the factorization $\left(g_j,g_{j+1}\right)$ of $g_jg_{j+1}$). If ${\cal F}$ is a minimal complexity factorization of $\Delta^2$, then ${\cal F}$ is Hurwitz equivalent to the standard factorization $\left(\sigma_1^3,\sigma_1^{-1}\sigma_2\sigma_1,\sigma_1,\sigma_2\right)$ (in Definition~\ref{dstandardfactorization}). 
\end{proposition}
\begin{proof}
We adopt Setup~\ref{setfactorization}. The hypothesis that $g_jg_{j+1} = \left(\sigma_{i_{j+1}'}\right)_{\left(\omega_{j+1}\right)}\Delta$ is equivalent to the statement that $g_j = \left(\sigma_{i_{j+1}}\sigma_{i_j}\sigma_{i_{j+1}}^{-1}\right)_{\left(\omega_{j+1}\right)}$ and $g_{j+1} = \left(\sigma_{i_{j+1}}^3\right)_{\left(\omega_{j+1}\right)}$ (as in Proposition~\ref{padjacentcubestandardpositivehalftwist}~\textbf{(ii)}). If we apply a global conjugation move to the factorization by a power of $\Delta$, then we can assume that $g_{j+1} = \sigma_1^3$, in which case $g_j = \sigma_1\sigma_2\sigma_1^{-1}$. If $j = 1$, then we will show that $g_3 = \sigma_2$ and $g_4 = \sigma_1$ are standard positive half-twists, and ${\cal F} = \left(\sigma_1\sigma_2\sigma_1^{-1},\sigma_1^3,\sigma_2,\sigma_1\right)$. 

Indeed, if $j=1$, then $g_1g_2 = \Delta\sigma_1$, and firstly we will show that $\sigma_2\xi_3$ is indivisible by $\Delta$. We have $g_3 = \left(\xi_3\right)_{\left(1\right)}$ in the notation of Setup~\ref{setfactorization}. If $\sigma_2\xi_3$ is divisible by $\Delta$, then $g_3$ is a non-standard positive half-twist, and the absolute value of the Garside power of $\left(g_1g_2\right)g_3\left(g_1g_2\right)^{-1}$ is one less than the absolute value of the Garside power of $g_3$. In particular, the composition of two Hurwitz moves $\left(g_1,g_2,g_3,g_4\right)\to \left(\left(g_1g_2\right)g_3\left(g_1g_2\right)^{-1},g_1,g_2,g_4\right)$ reduces the complexity of the factorization, which contradicts the hypothesis that ${\cal F}$ is a minimal complexity factorization. We deduce that $\sigma_2\xi_3$ is indivisible by $\Delta$.

We recall the equation $\xi_1\xi_2\xi_3\xi_4 = \Delta^{\omega_4+2}$ in the notation of Setup~\ref{setfactorization}, since $g_1g_2g_3g_4 = \Delta^2$. We also have $\xi_1\xi_2 = \Delta^2\sigma_2$ in the notation of Setup~\ref{setfactorization}. If either $g_3$ or $g_4$ is a standard positive half-twist, then Lemma~\ref{ladjacentstandardpositivehalftwist} implies that $\xi_3\xi_4$ is indivisible by $\Delta$, since ${\cal F}$ is a minimal complexity factorization. If $g_3=\sigma_2$, then $g_4 = \sigma_1$ is uniquely determined by the equation $g_1g_2g_3g_4 = \Delta^2$, and ${\cal F} = \left(\sigma_1\sigma_2\sigma_1^{-1},\sigma_1^3,\sigma_2,\sigma_1\right)$. If $g_3=\sigma_1$, then $\xi_3 = \sigma_2$ and $\xi_3\xi_4$ is indivisible by $\Delta$. We have $\xi_1\xi_2\xi_3\xi_4 = \Delta^2\sigma_2\sigma_2\xi_4$, and Lemma~\ref{ltripleproductindivisibleDelta} implies that $\sigma_2\sigma_2\xi_4$ is indivisible by $\Delta$. Of course, this contradicts the assumption that $\xi_1\xi_2\xi_3\xi_4$ is a power of $\Delta$. We conclude that $g_3$ is a non-standard positive half-twist. Similarly, if $g_4$ is a standard positive half-twist, then $\xi_1\xi_2\xi_3\xi_4 = \Delta^2\sigma_2\xi_3\xi_4$, and Lemma~\ref{ltripleproductindivisibleDelta} implies that $\sigma_2\xi_3\xi_4$ is indivisible by $\Delta$, unless ${\cal F} = \left(\sigma_1\sigma_2\sigma_1^{-1},\sigma_1^3,\sigma_2,\sigma_1\right)$. Of course, this contradicts the assumption that $\xi_1\xi_2\xi_3\xi_4$ is a power of $\Delta$. We conclude that $g_3$ and $g_4$ are both non-standard positive half-twists.

Let us define $\xi_3''$ to be the maximal right-divisor of $\xi_3$ that is left-dual to a left-divisor of $\xi_4$. Let $\chi_3''$ be the relevant left-divisor of $\xi_4$. We can write $\xi_3 = \xi_3'\xi_3''$ and $\xi_4 = \chi_3''\chi_3'$ for positive braids $\xi_3',\chi_3'\in B_3^{+}$. Lemma~\ref{lboundedcontraction} implies that $\xi_3''$ is a right-divisor of $\overline{\tau_3}$ and $\chi_3''$ is a left-divisor of $\overline{\rho_4}$, since $g_3$ and $g_4$ are non-standard positive half-twists. In this case, the Garside normal form of $\xi_1\xi_2\xi_3\xi_4$ is $\sigma_2\xi_3'\left(\chi_4'\right)_{\left(\omega\left(\xi_3''\right)\right)}\Delta^{2+\omega\left(\xi_3''\right)}$, since Proposition~\ref{pmaximaldivisordualindivisibleDelta} and Lemma~\ref{ltripleproductindivisibleDelta} imply that $\sigma_2\xi_3'\left(\chi_4'\right)_{\left(\omega\left(\xi_3''\right)\right)}$ is indivisible by $\Delta$. Of course, this contradicts the assumption that $\xi_1\xi_2\xi_3\xi_4$ is a power of $\Delta$.

We conclude that if $j = 1$, then ${\cal F} = \left(\sigma_1\sigma_2\sigma_1^{-1},\sigma_1^3,\sigma_2,\sigma_1\right)$. A similar argument shows that if $j = 3$, then ${\cal F} = \left(\sigma_2,\sigma_1,\sigma_1\sigma_2\sigma_1^{-1},\sigma_1^{3}\right)$. However, both factorizations are Hurwitz equivalent to the standard factorization $\left(\sigma_1^3,\sigma_1^{-1}\sigma_2\sigma_1,\sigma_1,\sigma_2\right)$. Indeed, the factorization $\left(\sigma_1\sigma_2\sigma_1^{-1},\sigma_1^3,\sigma_2,\sigma_1\right)$ is Hurwitz equivalent to the factorization $\left(\sigma_1^3,\sigma_1^{-2}\sigma_2\sigma_1^2,\sigma_1,\sigma_1^{-1}\sigma_2\sigma_1\right)$ by the composition of a Hurwitz move applied to the first two factors and a Hurwitz move applied to the second two factors. A global conjugation move by $\sigma_1$ applied to the latter factorization shows that if $j=1$, then ${\cal F}$ is Hurwitz equivalent to the standard factorization $\left(\sigma_1^3,\sigma_1^{-1}\sigma_2\sigma_1,\sigma_1,\sigma_2\right)$. 

A similar argument shows that if $j=3$, then ${\cal F} = \left(\sigma_2,\sigma_1,\sigma_1\sigma_2\sigma_1^{-1},\sigma_1^3\right)$ is Hurwitz equivalent to $\left(\sigma_1,\sigma_2,\sigma_1^3,\sigma_1^{-1}\sigma_2\sigma_1\right)$. The composition of two Hurwitz moves $\left(\sigma_1,\sigma_2,\sigma_1^3,\sigma_1^{-1}\sigma_2\sigma_1\right)\to \left(\left(\sigma_1\sigma_2\right)\sigma_1^3\left(\sigma_1\sigma_2\right)^{-1},\sigma_1,\sigma_2,\sigma_1^{-1}\sigma_2\sigma_1\right)$ followed by the composition of two Hurwitz moves $\left(\sigma_2^3,\sigma_1,\sigma_2,\sigma_1^{-1}\sigma_2\sigma_1\right)\to \left(\sigma_2^3,\left(\sigma_1\sigma_2\right)\sigma_1^{-1}\sigma_2\sigma_1\left(\sigma_1\sigma_2\right)^{-1},\sigma_1,\sigma_2\right)$, shows that ${\cal F}$ is Hurwitz equivalent to $\left(\sigma_2^3,\left(\sigma_1\sigma_2\right)\sigma_1^{-1}\sigma_2\sigma_1\left(\sigma_1\sigma_2\right)^{-1},\sigma_1,\sigma_2\right)$. If we apply a Hurwitz move to the last two factors in the latter factorization, then we deduce that ${\cal F}$ is Hurwitz equivalent to $\left(\sigma_2^3,\left(\sigma_1\sigma_2\right)\sigma_1^{-1}\sigma_2\sigma_1\left(\sigma_1\sigma_2\right)^{-1},\sigma_2,\sigma_2^{-1}\sigma_1\sigma_2\right)$. Subsequently, if we apply a global conjugation move by $\left(\sigma_1\sigma_2\right)^{-1}$ to the latter factorization, then we conclude that ${\cal F}$ is Hurwitz equivalent to the standard factorization $\left(\sigma_1^3,\sigma_1^{-1}\sigma_2\sigma_1,\sigma_1,\sigma_2\right)$. 

Finally, we address the case $j=2$. We recall that $g_1 = \Delta^{-\omega\left(\tau_1\right)}\left(\xi_1\right)_{\left(\omega\left(\tau_1\right)\right)}$ in the notation of Setup~\ref{setfactorization}. If $\xi_1\left(\sigma_2\right)_{\left(\omega\left(\tau_1\right)\right)}$ is divisible by $\Delta$, then the composition of two Hurwitz moves $\left(g_1,g_2,g_3,g_4\right)\to \left(g_2,g_3,\left(g_2g_3\right)^{-1}g_1\left(g_2g_3\right),g_4\right)$ reduces the complexity of the factorization since $g_2g_3 = \sigma_2\Delta$. Of course, this contradicts the hypothesis that ${\cal F}$ is a minimal complexity factorization. In particular, $\xi_1\left(\sigma_2\right)_{\left(\omega\left(\tau_1\right)\right)}$ is indivisible by $\Delta$, and a similar argument shows that $\left(\sigma_2\right)_{\left(\omega\left(\tau_1\right)\right)}\xi_4$ is indivisible by $\Delta$, since $g_2g_3 = \Delta\sigma_1$, and $g_4 = \Delta^{-\omega\left(\tau_4\right)}\left(\xi_4\right)_{\left(\omega\left(\tau_1\right)+1+\omega\left(\tau_4\right)\right)}$ in the notation of Setup~\ref{setfactorization}.

If $g_1 = \sigma_1$, then the equation $g_1g_2g_3g_4 = \Delta^2$ uniquely determines $g_4 = \sigma_2$ and ${\cal F} = \left(\sigma_1,\sigma_1\sigma_2\sigma_1^{-1},\sigma_1^{3},\sigma_2\right)$. In this case, if we apply a global conjugation move by $\sigma_1^{-1}$ to ${\cal F}$, then we deduce that ${\cal F}$ is Hurwitz equivalent to $\left(\sigma_1,\sigma_2,\sigma_1^3,\sigma_1^{-1}\sigma_2\sigma_1\right)$, and we have already shown that this latter factorization is Hurwitz equivalent to the standard factorization $\left(\sigma_1^3,\sigma_1^{-1}\sigma_2\sigma_1,\sigma_1,\sigma_2\right)$ earlier in the proof. We have $\xi_2\xi_3 = \Delta^2\left(\sigma_2\right)_{\left(\omega\left(\tau_1\right)\right)}$ in the notation of Setup~\ref{setfactorization}. If $g_1 = \sigma_2$, then $\xi_1\xi_2\xi_3\xi_4 = \Delta^2\left(\sigma_2\sigma_2\right)_{\left(\omega\left(\tau_1\right)\right)}\xi_4$. In particular, Lemma~\ref{ltripleproductindivisibleDelta} implies that $\left(\sigma_2\right)_{\left(\omega\left(\tau_1\right)\right)}\xi_4$ is divisible by $\Delta$, since $\xi_1\xi_2\xi_3\xi_4$ is a power of $\Delta$, and this is a contradiction. Thus, we can assume that $g_1$ is not a standard positive half-twist, and a similar argument shows that $g_4$ is not a standard positive half-twist.

We observe that the composition of two Hurwitz moves $\left(g_1,g_2,g_3,g_4\right)\to \left(g_2,g_3,\left(g_2g_3\right)^{-1}g_1\left(g_2g_3\right),g_4\right)$ increases the complexity of the factorization by at most one, since $g_2g_3 = \sigma_2\Delta$. However, the assumption that $\xi_1\xi_2\xi_3\xi_4$ is a power of $\Delta$ implies that $\xi_1\left(\sigma_2\right)_{\left(\omega\left(\tau_1\right)\right)}$ is left-dual to $\xi_4$. Of course, $\xi_1$ is not a right-divisor of $\overline{\tau_1}$, and this implies that that the conclusion in the first statement of Lemma~\ref{lboundedcontraction} is not satisfied for the third and fourth factors in $\left(g_2,g_3,\left(g_2g_3\right)^{-1}g_1\left(g_2g_3\right),g_4\right)$. The contrapositive of Lemma~\ref{lboundedcontraction} implies that we can apply a complexity reducing Hurwitz move to the factorization $\left(\left(g_2g_3\right)^{-1}g_1\left(g_2g_3\right),g_4\right)$. We conclude that ${\cal F}$ is Hurwitz equivalent to a minimal complexity factorization that satisfies the hypothesis with $j=1$. However, we have already handled this case, and we deduce that if $j=2$, then ${\cal F}$ is Hurwitz equivalent to the standard factorization $\left(\sigma_1^3,\sigma_1^{-1}\sigma_2\sigma_1,\sigma_1,\sigma_2\right)$.

We have considered all possibilities for $j\in \{1,2,3\}$. Therefore, the statement is established. 
\end{proof}

The following statement is an analogue of Lemma~\ref{lminimalcomplexitypositivehalftwists} for a pair of adjacent factors in a minimal complexity factorization, where one factor is the cube of a non-standard positive half-twist and the other factor is a non-standard positive half-twist. 

\begin{lemma}
\label{lminimalcomplexitycubepositivehalftwists}
We adopt Setup~\ref{setfactorization} with $k=2$ and where $g_1$ is the cube of a non-standard positive half-twist and $g_2$ is a non-standard positive half-twist. If ${\cal F} = \left(g_1,g_2\right)$ is a minimal complexity factorization of $g_1g_2$, then the following conditions are satisfied:
\begin{description}
\item[(i)] The maximal right-divisor $\xi_1''$ of $\xi_1$ that is left-dual to a left-divisor of $\overline{\rho_2}$ is a right-divisor of $\sigma_{i_1}\overline{\tau_1}$. If $\xi_1''=\sigma_{i_1}\overline{\tau_1}$, then the Hurwitz move $\left(g_1,g_2\right)\to \left(g_1g_2g_1^{-1},g_1\right)$ does not increase the complexity of the factorization.
\item[(ii)] The maximal left-divisor $\xi_2''$ of $\xi_2$ that is right-dual to a right-divisor of $\overline{\tau_1}$ is a left-divisor of $\overline{\rho_2}$.
\end{description}
\end{lemma}
\begin{proof}
The statement \textbf{(ii)} is a special case of Lemma~\ref{lminimalcomplexitypositivehalftwists}~\textbf{(ii)}. We prove \textbf{(i)} using Lemma~\ref{lHurwitzmovecomplexitychange}. Firstly, Lemma~\ref{lmaximaldivisorcutdivisor} implies that $\xi_1''$ is a cut right-divisor of $\xi_1$. Proposition~\ref{pcutclosureconjugate} implies that $\xi_1 = \overline{\rho_1}\sigma_{i_1}^2\overline{\tau_1}$, and Lemma~\ref{lcutdivisorcharacterization} implies that $\sigma_{i_1}\overline{\tau_1}$ is also a cut right-divisor of $\xi_1$. Furthermore, Proposition~\ref{ppdb} implies that $\omega\left(\sigma_{i_1}\overline{\tau_1}\right) = \omega\left(\tau_1\right) + 1$, since $\overline{\tau_1}$ is a cut right-divisor of $\sigma_{i_1}\overline{\tau_1}$ according to Proposition~\ref{pcutdivisorsubdivisor}.

Proposition~\ref{pomeganumbercutdivisor} now implies that $\xi_1''$ is a right-divisor of $\sigma_{i_1}\overline{\tau_1}$ if and only if $\omega\left(\xi_1''\right)\leq \omega\left(\sigma_{i_1}\overline{\tau_1}\right) = \omega\left(\tau_1\right) + 1$. However, if $\omega\left(\xi_1''\right)>\omega\left(\tau_1\right)+1$, then Lemma~\ref{lHurwitzmovecomplexitychange} with $e_1 = 3$ implies that the Hurwitz move $\left(g_1,g_2\right)\to \left(g_1g_2g_1^{-1},g_1\right)$ reduces the complexity of the factorization. Of course, this would contradict the hypothesis that ${\cal F}$ is a minimal complexity factorization. We conclude that $\xi_1''$ is a right-divisor of $\sigma_{i_1}\overline{\tau_1}$. 

Finally, if $\xi_1'' = \sigma_{i_1}\overline{\tau_1}$, and if $\xi_1''$ is left-dual to the left-divisor $\rho_2''$ of $\overline{\rho_2}$, then $\rho_2''$ is a cut left-divisor of $\overline{\rho_2}$. Indeed, Lemma~\ref{lcharacterizationdualpositivebraids} implies that $\rho_2''$ ends with a product of distinct generators, and Lemma~\ref{lcutdivisorcharacterization} implies that $\rho_2''$ is a cut left-divisor of $\overline{\rho_2}$. In particular, $\epsilon\leq 0$ in the context of Lemma~\ref{lHurwitzmovecomplexitychange}. Lemma~\ref{lHurwitzmovecomplexitychange} with $e_1 = 3$ now implies that the Hurwitz move $\left(g_1,g_2\right)\to \left(g_1g_2g_1^{-1},g_1\right)$ does not increase the complexity of the factorization.

Therefore, the statement is established. 
\end{proof}

We now establish an analogue of Lemma~\ref{lboundedcontraction}, where one factor is the cube of a non-standard positive half-twist and the other factor is a non-standard positive half-twist.

\begin{lemma}
\label{lcubeboundedcontraction}
We adopt Setup~\ref{setfactorization} with $k = 2$. Let us assume that $g_1$ is the cube of a non-standard positive half-twist, $g_2$ is a non-standard positive half-twist, and ${\cal F} = \left(g_1,g_2\right)$ is a minimal complexity factorization of $g_1g_2$. 

In this case, the maximal right-divisor $\xi_1''$ of $\xi_1$ left-dual to a left-divisor of $\xi_2$ is a right-divisor of $\sigma_{i_1}\overline{\tau_1}$. Furthermore, the maximal left-divisor $\xi_2''$ of $\xi_2$ right-dual to a right-divisor of $\xi_1$ is a left-divisor of the cut closure $\overline{\rho_2}$. However, $\sigma_{i_1}\overline{\tau_1}$ is not left-dual to $\overline{\rho_2}$.

Finally, if $\xi_1'' = \sigma_{i_1}\overline{\tau_1}$, then the Hurwitz move $\left(g_1,g_2\right)\to \left(g_1g_2g_1^{-1},g_1\right)$ does not increase the complexity of the factorization. 
\end{lemma}
\begin{proof}
We will establish the first statement that $\xi_1''$ is a right-divisor of $\sigma_{i_1}\overline{\tau_1}$, based on Lemma~\ref{lminimalcomplexitycubepositivehalftwists}. The proof that $\xi_2''$ is a left-divisor of the cut closure $\overline{\rho_2}$ is analogous, also based on Lemma~\ref{lminimalcomplexitycubepositivehalftwists}. We use a similar argument to the proof of Lemma~\ref{lboundedcontraction} (which is based on Lemma~\ref{lminimalcomplexitypositivehalftwists} rather than Lemma~\ref{lminimalcomplexitycubepositivehalftwists}), but with an extra step to show that $\sigma_{i_1}\overline{\tau_1}$ is not left-dual to $\overline{\rho_2}$. 

Firstly, Lemma~\ref{lmaximaldivisorcutdivisor} implies that $\xi_1''$ is a cut right-divisor of $\xi_1$. Lemma~\ref{lcutdivisorcharacterization} implies that $\sigma_{i_1}\overline{\tau_1}$ is a cut right-divisor of $\xi_1$, since Proposition~\ref{pcutclosureconjugate} implies that $\xi_1 = \overline{\rho_1}\sigma_{i_1}\sigma_{i_1}\overline{\tau_1}$. Let us assume, for a contradiction, that $\xi_1''$ is not a right-divisor of $\sigma_{i_1}\overline{\tau_1}$, or equivalently, $\sigma_{i_1}\overline{\tau_1}$ is a proper right-divisor of $\xi_1''$. In this case, Proposition~\ref{pcutdivisorsubdivisor} implies that $\sigma_{i_1}\overline{\tau_1}$ is a proper cut right-divisor of $\xi_1''$.

We observe that the left-divisor $\chi_1''$ of $\xi_2$ that is right-dual to $\xi_1''$ cannot be a left-divisor of $\overline{\rho_2}$. Indeed, if $\chi_1''$ were a left-divisor of $\overline{\rho_2}$, then this would contradict the assumption that ${\cal F}$ is a minimal complexity factorization, according to Lemma~\ref{lminimalcomplexitycubepositivehalftwists}~\textbf{(i)}. We deduce that $\overline{\rho_2}$ is a proper left-divisor of $\chi_1''$.

Let $\rho_2''$ be the left-divisor of $\chi_1''$ that is right-dual to $\sigma_{i_1}\overline{\tau_1}$. Firstly, Lemma~\ref{lcharacterizationdualpositivebraids} implies that $\rho_2''$ ends with a product of two distinct generators, since Proposition~\ref{pcutclosureconjugate} implies that $\overline{\tau_1}$ begins with $\sigma_{i_1}\sigma_{i_1'}$ for $i_1\neq i_1'\in \{1,2\}$. We deduce that $\rho_2''$ is a cut left-divisor of $\xi_2$ using Lemma~\ref{lcutdivisorcharacterization}. 

We claim that $\rho_2''$ is a left-divisor of $\overline{\rho_2}$. Proposition~\ref{pomeganumbercutdivisor} implies that $\rho_2''$ is a left-divisor of $\overline{\rho_2}$ if and only if $\omega\left(\rho_2''\right)\leq \omega\left(\overline{\rho_2}\right)$. However, if $\omega\left(\rho_2''\right)>\omega\left(\overline{\rho_2}\right)$, then $\omega\left(\rho_2''\right)>\omega\left(\overline{\rho_2}\right)+1$. Indeed, Proposition~\ref{pcutclosureconjugate} implies that $\overline{\tau_1}$ begins with the product of two distinct generators $\sigma_{i_1}\sigma_{i_1'}$ for $i_1\neq i_1'\in \{1,2\}$, and if $\omega\left(\rho_2''\right) = \omega\left(\overline{\rho_2}\right) + 1$, then Proposition~\ref{pcutclosureconjugate} implies that $\rho_2''$ ends with $\sigma_{i_2'}\sigma_{i_2}\sigma_{i_2}\sigma_{i_2'}$ for $i_2\neq i_2'\in \{1,2\}$. However, this contradicts Lemma~\ref{lcharacterizationdualpositivebraids}, since $\sigma_{i_1}\overline{\tau_1}$ is left-dual to $\rho_2''$. On the other hand, if $\omega\left(\rho_2''\right)>\omega\left(\overline{\rho_2}\right)+1$, then $\overline{\tau_1}$ is left-dual to a left-divisor of $\xi_2$ of which $\overline{\rho_2}$ is a proper left-divisor. Of course, this contradicts Lemma~\ref{lminimalcomplexitycubepositivehalftwists}~\textbf{(ii)}.

In fact, we claim that $\rho_2'' = \overline{\rho_2}$. We have assumed that $\sigma_{i_1}\overline{\tau_1}$ is a proper right-divisor of $\xi_1''$, and we will use this to show that $\omega\left(\rho_2''\right) = \omega\left(\overline{\rho_2}\right)$. We have already shown that $\rho_2''$ is a left-divisor of $\overline{\rho_2}$. If $\omega\left(\rho_2''\right)<\omega\left(\overline{\rho_2}\right)$, then $\overline{\rho_2}$ is right-dual to a right-divisor of $\xi_1''$, of which $\sigma_{i_1}\overline{\tau_1}$ is a proper right-divisor. However, this contradicts the assumption that ${\cal F}$ is a minimal complexity factorization, according to Lemma~\ref{lminimalcomplexitycubepositivehalftwists}~\textbf{(i)}. We conclude that $\omega\left(\rho_2''\right) = \omega\left(\overline{\rho_2}\right)$, and Proposition~\ref{pomeganumbercutdivisor} implies that $\rho_2'' = \overline{\rho_2}$. We conclude that $\sigma_{i_1}\overline{\tau_1}$ is left-dual to $\overline{\rho_2}$.

We will derive a contradiction from the statement that $\sigma_{i_1}\overline{\tau_1}$ is left-dual to $\overline{\rho_2}$. Proposition~\ref{pcutclosureconjugate} implies that $\xi_1 = \overline{\rho_1}\sigma_{i_1}^2\overline{\tau_1}$ and $\xi_2 = \overline{\rho_2}\overline{\tau_2}$. Lemma~\ref{lcharacterizationdualpositivebraids} implies that $\overline{\rho_2}$ ends with $\sigma_{i_2'}\sigma_{i_2'}\sigma_{i_2}$, since $\sigma_{i_1}\overline{\tau_1}$ is left-dual to $\overline{\rho_2}$. Lemma~\ref{lcharacterizationdualpositivebraids} and Proposition~\ref{pcutclosureconjugate} now imply that $\overline{\tau_2}$ begins with $\sigma_{i_2}\sigma_{i_2'}\sigma_{i_2'}\sigma_{i_2}$, since $g_2$ is a non-standard positive half-twist. 

Let $\rho_1'\in B_3^{+}$ be the positive braid such that $\rho_1'\sigma_{i_1} = \overline{\rho_1}$. Lemma~\ref{lcharacterizationdualpositivebraids} implies that $\rho_1'$ is left-dual to $\overline{\tau_1}$, since $\rho_1$ is left-dual to $\tau_1$. Let $\rho_2'\in B_3^{+}$ be the positive braid such that $\overline{\rho_2} = \rho_2'\sigma_{i_2'}\sigma_{i_2}$. Lemma~\ref{lcharacterizationdualpositivebraids} implies that $\overline{\tau_1}$ is left-dual to $\rho_2'$, since $\sigma_{i_1}\overline{\tau_1}$ is left-dual to $\overline{\rho_2}$. Proposition~\ref{ppdb} implies that $\omega\left(\overline{\rho_2}\right) = \omega\left(\rho_2'\right) + 1$. The pseudo-symmetry of duals (Proposition~\ref{ppseudosymmetryduality}) implies that $\left(\rho_2'\right)_{\left(\omega\left(\tau_1\right)\right)}$ is left-dual to $\overline{\tau_1}$, and the uniqueness of duals (Proposition~\ref{puniquenessofduals}) now implies that $\left(\rho_2'\right)_{\left(\omega\left(\tau_1\right)\right)} = \rho_1'$. 

In the notation of Setup~\ref{setfactorization}, we have $g_1 = \Delta^{-\omega\left(\tau_1\right)}\left(\rho_1\sigma_{i_1}^3\tau_1\right)_{\left(\omega\left(\tau_1\right)\right)}$ and $g_2 = \Delta^{-\omega\left(\tau_2\right)}\left(\rho_2\sigma_{i_2}\tau_2\right)_{\left(\omega\left(\tau_1\right) + \omega\left(\tau_2\right)\right)}$. Proposition~\ref{pcutclosureconjugate} implies that $\xi_1 = \rho_1'\sigma_{i_1}\sigma_{i_1}\sigma_{i_1}\overline{\tau_1}$, where $\sigma_{i_1}\sigma_{i_1}\sigma_{i_1}\overline{\tau_1}$ is left-dual to the left-divisor $\overline{\rho_2}\sigma_{i_2}\sigma_{i_2'}\sigma_{i_2'}\sigma_{i_2}$ of $\xi_2$. Furthermore, the pseudo-symmetry of duals (Proposition~\ref{ppseudosymmetryduality}) and Proposition~\ref{ppdb} imply that $\omega\left(\sigma_{i_1}\sigma_{i_1}\sigma_{i_1}\overline{\tau_1}\right) = \omega\left(\overline{\rho_2}\sigma_{i_2}\sigma_{i_2'}\sigma_{i_2'}\sigma_{i_2}\right) = \omega\left(\rho_2\right) + 2$ and $\omega\left(\rho_1'\right) = \omega\left(\rho_2'\right) = \omega\left(\rho_2\right) - 1$. Proposition~\ref{pGarsidepowerinversepowerpositivehalftwist} implies that the Garside power of $g_2^{-1}$ is equal to the Garside power of $g_2$, which is $-\omega\left(\rho_2\right)$. In particular, the absolute value of the Garside power of $g_2^{-1}g_1g_2$ is strictly less than the absolute value of the Garside power of $g_1$. We conclude that the Hurwitz move $\left(g_1,g_2\right)\to \left(g_2,g_2^{-1}g_1g_2\right)$ decreases the complexity of the factorization, which contradicts the hypothesis that ${\cal F}$ is a minimal complexity factorization.

The final statement is a consequence of Lemma~\ref{lminimalcomplexitycubepositivehalftwists}~\textbf{(i)}.
\end{proof}

Finally, we establish an analogue of Lemma~\ref{ltwoHurwitzmovessquarepositivehalftwist} for a triple of factors in a factorization, where the middle factor is the cube of a positive half-twist, and the other factors are non-standard positive half-twists. 

\begin{lemma}
\label{ltwoHurwitzmovescubepositivehalftwist}
We adopt Setup~\ref{setfactorization} with $k = 3$. Let us assume that $g_j$ is a non-standard positive half-twist for each $j\in \{1,3\}$, and $g_2$ is the cube of a positive half-twist. If $j\in \{1,2\}$, then let $\xi_j''$ be the maximal right-divisor of $\xi_j$ that is left-dual to a left-divisor of $\xi_{j+1}$. Let us denote the relevant left-divisor of $\xi_{j+1}$ by $\chi_j''$. Let us write $\xi_1 = \xi_1'\xi_1''$ and $\xi_3 = \chi_2''\xi_3'$ for positive braids $\xi_1',\xi_3'\in B_3^{+}$.

Let us assume that $\xi_1''$ is a right-divisor of $\overline{\tau_1}$ and $\chi_2''$ is a left-divisor of $\overline{\rho_3}$. If $g_2$ is the cube of a non-standard positive half-twist, then we assume that either $\chi_1''=\overline{\rho_2}\sigma_{i_2}$ and $\xi_2'' = \overline{\tau_2}$, or $\chi_1'' = \overline{\rho_2}$ and $\xi_2'' = \sigma_{i_2}\overline{\tau_2}$. If $g_2$ is the cube of a standard positive half-twist, then we assume that $\chi_1'' = \sigma_{i_2}$ and $\xi_2'' = \sigma_{i_2}$. (We have $\xi_2 = \chi_1''\sigma_{i_2}\xi_2''$ in both cases.) Let $\xi_1'''$ be the maximal right-divisor of $\xi_1'$ that is left-dual to a left-divisor of $\left(\sigma_{i_2}\right)_{\left(\omega\left(\tau_2\right)\right)+1}\xi_3'$. Let us denote the relevant left-divisor of $\left(\sigma_{i_2}\right)_{\left(\omega\left(\tau_2\right)+1\right)}\xi_3'$ by $\left(\sigma_{i_2}\right)_{\left(\omega\left(\tau_2\right)+1\right)}\xi_3'''$. 

If either the right-divisor $\xi_1'''\xi_1''$ of $\xi_1$ is not a right-divisor of $\overline{\tau_1}$ or the left-divisor $\chi_2''\xi_3'''$ is not a left-divisor of $\overline{\rho_3}$, then the following statements are true:
\begin{description}
\item[(i)] If $g_2$ is the cube of a standard positive half-twist, then the composition of the Hurwitz move $\left(g_1,g_2,g_3\right)\to \left(g_2,g_2^{-1}g_1g_2,g_3\right)$ followed by a Hurwitz move applied to the second and third factors, does not increase the complexity of the factorization. If $g_2$ is the cube of a non-standard positive half-twist with $\chi_1'' = \overline{\rho_2}\sigma_{i_2}$ and $\xi_2'' = \overline{\tau_2}$, then the composition of the Hurwitz move $\left(g_1,g_2,g_3\right)\to \left(g_2,g_2^{-1}g_1g_2,g_3\right)$ followed by a Hurwitz move applied to the second and third factors, decreases the complexity of the factorization.
\item[(ii)] If $g_2$ is the cube of a standard positive half-twist, then the composition of the Hurwitz move $\left(g_1,g_2,g_3\right)\to \left(g_1,g_2g_3g_2^{-1},g_2\right)$ followed by a Hurwitz move applied to the first and second factors, does not increase the complexity of the factorization. If $g_2$ is the cube of a non-standard positive half-twist with $\chi_1'' = \overline{\rho_2}$ and $\xi_2'' = \sigma_{i_2}\overline{\tau_2}$, then the composition of the Hurwitz move $\left(g_1,g_2,g_3\right)\to \left(g_1,g_2g_3g_2^{-1},g_2\right)$ followed by a Hurwitz move applied to the first and second factors, decreases the complexity of the factorization.
\end{description}
\end{lemma}
\begin{proof}
We will assume that the right-divisor $\xi_1'''\xi_1''$ of $\xi_1$ is not a right-divisor of $\overline{\tau_1}$ and prove \textbf{(i)}. The proofs of the other cases are similar. If $g_2$ is the cube of a non-standard positive half-twist, then (a symmetric version of) Lemma~\ref{lHurwitzmovecomplexitychange} with $e_2 = 3$ implies that the Hurwitz move $\left(g_1,g_2,g_3\right)\to \left(g_2,g_2^{-1}g_1g_2,g_3\right)$ does not increase the complexity of the factorization, since $\xi_1''$ is a cut right-divisor of $\overline{\tau_1}$ and $\chi_1''=\overline{\rho_2}\sigma_{i_2}$. If $g_2$ is the cube of a standard positive half-twist, then the Hurwitz move $\left(g_1,g_2,g_3\right)\to \left(g_2,g_2^{-1}g_1g_2,g_3\right)$ increases the complexity of the factorization by at most one. 

However, the assumption that $\xi_1'''\xi_1''$ is not a right-divisor of $\overline{\tau_1}$ implies that the conclusion in the first statement of Lemma~\ref{lboundedcontraction} is not satisfied for the second and third factors in $\left(g_2,g_2^{-1}g_1g_2,g_3\right)$. Indeed, this is a consequence of (a symmetric version of) Lemma~\ref{lHurwitzmoveGarsidenormalform}, which determines the Garside normal form of $g_2^{-1}g_1g_2$. The contrapositive of Lemma~\ref{lboundedcontraction} implies that we can apply a complexity reducing Hurwitz move to the factorization $\left(g_2^{-1}g_1g_2,g_3\right)$. Therefore, the statement is established.
\end{proof}

We are now prepared to classify factorizations of $\Delta^2$ into powers of positive half-twists, where one of the factors is the cube of a positive half-twist, and the other factors are positive half-twists.

\begin{theorem}
\label{tmainclassificationfactorizationscubepositivehalftwist}
If ${\cal F}$ is a factorization of $\Delta^2$ into three positive half-twists and one cube of a positive half-twist, then ${\cal F}$ is Hurwitz equivalent to the standard factorization $\left(\sigma_1^3,\sigma_1^{-1}\sigma_2\sigma_1,\sigma_1,\sigma_2\right)$ (in Definition~\ref{dstandardfactorization}).
\end{theorem}
\begin{proof}
We adopt Setup~\ref{setfactorization}. Let ${\cal F} = \left(g_1,g_2,g_3,g_4\right)$ be a factorization of $\Delta^2$ into three positive half-twists and one cube of a positive half-twist. We will show that ${\cal F}$ is Hurwitz equivalent to the standard factorization $\left(\sigma_1^3,\sigma_1^{-1}\sigma_2\sigma_1,\sigma_1,\sigma_2\right)$. Let us assume, for a contradiction, that ${\cal F}$ is a factorization of minimal complexity that is not Hurwitz equivalent to the standard factorization $\left(\sigma_1^3,\sigma_1^{-1}\sigma_2\sigma_1,\sigma_1,\sigma_2\right)$. Lemma~\ref{lstandardfactorizationspowersHurwitzequivalent} implies that the complexity $c\left({\cal F}\right)>0$.

We will derive a contradiction and the proof is very similar to that of Theorem~\ref{tfactorizationpositivehalftwistsstandard}, with adaptations that we will explain. We can assume that the factorization ${\cal F}$ satisfies the conditions in Lemma~\ref{lreorderingfactorssamecomplexity}, after possibly replacing ${\cal F}$ with another minimal complexity factorization (as in the proof of Theorem~\ref{tfactorizationpositivehalftwistsstandard}). Furthermore, if $1\leq j\leq 4$, then we adopt the notation in the proof of Theorem~\ref{tfactorizationpositivehalftwistsstandard} for $\xi_j''$, $\chi_j''$, and $\zeta_j$. Indeed, if $1\leq j\leq 4$, then we define $\xi_j''$ to be the maximal right-divisor of $\xi_j$ that is left-dual to a left-divisor of $\xi_{j+1}$. We denote the relevant left-divisor of $\xi_{j+1}$ by $\chi_{j}''$. We also define $\zeta_j\in B_3^{+}$ to be the positive braid such that $\xi_j = \chi_{j-1}''\zeta_j\xi_j''$. Finally, we define $J\subseteq \{1,2,3,4\}$ to be the set of all $1\leq j\leq 4$ such that $g_j$ is a standard positive half-twist and let $J' = \{1,2,3,4\}\setminus J$ be the complement of $J$.

Let $g_{k'}$ be the factor in ${\cal F}$ that is the cube of a positive half-twist, where $1\leq k'\leq 4$. We will establish constraints on $\chi_{j-1}''$ and $\xi_j''$, which will imply that $\zeta_j$ is not the identity braid for each $1\leq j\leq 4$, after possibly replacing ${\cal F}$ with a Hurwitz equivalent minimal complexity factorization. Firstly, we consider the case where $g_{k'}$ is the cube of a non-standard positive half-twist. Lemma~\ref{lcubeboundedcontraction} implies that $\xi_{k'}''$ is a right-divisor of $\sigma_{i_{k'}}\overline{\tau_{k'}}$. If $\xi_{k'}'' = \sigma_{i_{k'}}\overline{\tau_{k'}}$ and if $k'<4$, then Lemma~\ref{lcubeboundedcontraction} implies that the Hurwitz move $\left(\dots,g_{k'},g_{k'+1},\dots\right)\to \left(\dots,g_{k'}g_{k'+1}g_{k'}^{-1},g_{k'},\dots\right)$ does not change the complexity of the factorization. In particular, after possibly replacing ${\cal F}$ by a Hurwitz equivalent factorization, we can assume that $\xi_{k'}''$ is a proper right-divisor of $\sigma_{i_{k'}}\overline{\tau_{k'}}$. Lemma~\ref{lcubeboundedcontraction} also implies that $\chi_{k'-1}''$ is a left-divisor of $\overline{\rho_{k'}}\sigma_{i_{k'}}$. If $j\neq k'$ and $j\in J'$, then Lemma~\ref{lboundedcontraction} and Lemma~\ref{lcubeboundedcontraction} imply that $\xi_j''$ is a right-divisor of $\overline{\tau_j}$ and $\chi_{j-1}''$ is a proper left-divisor of $\overline{\rho_j}$. If $j\in J$, then Lemma~\ref{ladjacentstandardpositivehalftwist} implies that $\xi_j''$ and $\chi_{j-1}''$ are both the identity braid. If $g_{k'}$ is the cube of a non-standard positive half-twist, then we conclude that $\zeta_j$ is not the identity braid for each $1\leq j\leq 4$, based on the equation $\xi_j = \chi_{j-1}''\zeta_j\xi_j''$. Indeed, Proposition~\ref{pcutclosureconjugate} implies that $\xi_{k'} = \overline{\rho_{k'}}\sigma_{i_{k'}}^2\overline{\tau_{k'}}$, and $\xi_j = \overline{\rho_j}\overline{\tau_j}$ if $j\neq k'$ and $j\in J'$.

On the other hand, if $g_{k'}$ is the cube of a standard positive half-twist, and if either $\xi_{k'-1}''$ is not a right-divisor of $\overline{\tau_{k'-1}}$ (equivalently, $\chi_{k'-1}''$ is not a left-divisor of $\sigma_{i_{k'}}$ by Proposition~\ref{padjacentcubestandardpositivehalftwist}) or $\chi_{k'}''$ is not a left-divisor of $\overline{\rho_{k'+1}}$ (equivalently, $\xi_{k'}''$ is not a right-divisor of $\sigma_{i_{k'}}$ by Proposition~\ref{padjacentcubestandardpositivehalftwist}), then Proposition~\ref{padjacentcubestandardpositivehalftwist} implies that the hypothesis of Proposition~\ref{pcubestandardpositivehalftwistHurwitzequivalenceclass} is satisfied for the factorization ${\cal F}$. 

In this case, Proposition~\ref{pcubestandardpositivehalftwistHurwitzequivalenceclass} implies that ${\cal F}$ is Hurwitz equivalent to the standard factorization $\left(\sigma_1^3,\sigma_1^{-1}\sigma_2\sigma_1,\sigma_1,\sigma_2\right)$, which is a contradiction. If $j\neq k'$ and $j\in J'$, then Lemma~\ref{lboundedcontraction} and this reasoning imply that $\xi_j''$ is a right-divisor of $\overline{\tau_j}$ and $\chi_{j-1}''$ is a proper left-divisor of $\overline{\rho_j}$. Furthermore, if $g_{k'}$ is the cube of a standard positive half-twist, then $\chi_{k'-1}''$ is a left-divisor of $\sigma_{i_{k'}}$ and $\xi_{k'}''$ is a right-divisor of $\sigma_{i_{k'}}$. Of course, if $j\in J$, then Lemma~\ref{ladjacentstandardpositivehalftwist} implies that $\xi_j''$ and $\chi_{j-1}''$ are both the identity braid. We conclude that $\zeta_j$ is not the identity braid for each $1\leq j\leq 4$ in all cases.

We have $\Delta^{\omega_{4}+2} = \xi_1\xi_2\xi_3\xi_4$ in the notation of Setup~\ref{setfactorization}. Let $\omega_j'' = \sum_{j'=1}^{j-1} \omega\left(\xi_{j'}''\right)$. We have the equation (as in the proof of Theorem~\ref{tfactorizationpositivehalftwistsstandard}) \[\xi_1\xi_2\xi_3\xi_4 = \left[\prod_{j=1}^{4} \left(\zeta_j\right)_{\left(\omega_j''\right)}\right]\Delta^{\omega_4''}\] but the positive braid in brackets is not necessarily indivisible by $\Delta$ (unlike in the proof of Theorem~\ref{tfactorizationpositivehalftwistsstandard}). 

Proposition~\ref{pmaximaldivisordualindivisibleDelta} and the maximality of $\xi_j''$ imply that $\left(\zeta_{j-1}\right)_{\left(\omega_{j-1}''\right)}\left(\zeta_j\right)_{\left(\omega_j''\right)}$ is indivisible by $\Delta$ for each $1\leq j\leq 4$. Lemma~\ref{lproductindivisibleDelta} implies that $\prod_{j=1}^{4} \left(\zeta_j\right)_{\left(\omega_j''\right)}$ is indivisible by $\Delta$ if $\zeta_j$ is not an Artin generator for each $1\leq j\leq 4$. We will analyze the possible cases where $\zeta_j$ is an Artin generator, and we will show that the product $\prod_{j=1}^{4} \left(\zeta_j\right)_{\left(\omega_j''\right)}$ is not a power of $\Delta$ in every case, which contradicts the fact that $\xi_1\xi_2\xi_3\xi_4$ is a power of $\Delta$.

Firstly, if $k'\neq j$, $\{j,j+1\}\subseteq J'$ (i.e., neither $g_j$ nor $g_{j+1}$ are standard positive half-twists), and if $\zeta_j$ is an Artin generator, then the index of the generator $\left(\zeta_j\right)_{\left(\omega_j''\right)}$ is equal to the index of the first generator in $\left(\zeta_{j+1}\right)_{\left(\omega_{j+1}''\right)}$ (this is a consequence of Proposition~\ref{pmaximaldivisordualcharacterization}~\textbf{(i)}, using the same reasoning as in the proof of Theorem~\ref{tfactorizationpositivehalftwistsstandard}). Furthermore, if $g_{k'}$ is the cube of a standard positive half-twist, then $g_{k'}$ is not followed (resp. preceded) by two standard positive half-twists $g_{k'+1}$ and $g_{k'+2}$ (resp. $g_{k'-1}$ and $g_{k'-2}$) such that $\xi_{k'}\xi_{k'+1}\xi_{k'+2}$ (resp. $\xi_{k'-2}\xi_{k'-1}\xi_{k'}$) is divisible by $\Delta$. Indeed, if this is not true, then the fact that $c\left({\cal F}\right)>0$ implies that there is a non-standard positive half-twist in the factorization. Thus, after possibly applying a composition of two Hurwitz moves that does not change the complexity of the factorization, we can assume that $g_{k'}$ is adjacent to a non-standard positive half-twist. However, in this case, the hypothesis of Proposition~\ref{pcubestandardpositivehalftwistHurwitzequivalenceclass} is satisfied for the factorization ${\cal F}$. Proposition~\ref{pcubestandardpositivehalftwistHurwitzequivalenceclass} implies that ${\cal F}$ is Hurwitz equivalent to the standard factorization $\left(\sigma_1^3,\sigma_1^{-1}\sigma_2\sigma_1,\sigma_1,\sigma_2\right)$, which is a contradiction.

We now use the statements in the previous paragraph to establish a contradiction. If $k'\neq j+1$, then Lemma~\ref{ladjacentstandardpositivehalftwist},  Lemma~\ref{ltwoadjacentstandardpositivehalftwists}, Lemma~\ref{ltripleproductindivisibleDelta}, and the assumption that ${\cal F}$ satisfies the conditions in Lemma~\ref{lreorderingfactorssamecomplexity} imply that $\left(\zeta_j\right)_{\left(\omega_j''\right)}\left(\zeta_{j+1}\right)_{\left(\omega_{j+1}''\right)}\left(\zeta_{j+2}\right)_{\left(\omega_{j+2}''\right)}$ is indivisible by $\Delta$. Lemma~\ref{lproductindivisibleDelta} now implies that $\prod_{j'=1}^{k'} \left(\zeta_{j'}\right)_{\left(\omega_{j'}\right)}$ and $\prod_{j'=k'+1}^{4} \left(\zeta_{j'}\right)_{\left(\omega_{j'}\right)}$ are indivisible by $\Delta$. Furthermore, if $\zeta_{k'}$ is not an Artin generator, then Lemma~\ref{lproductindivisibleDelta} implies that $\prod_{j'=1}^{4} \left(\zeta_{j'}\right)_{\left(\omega_j''\right)}$ is indivisible by $\Delta$, which is a contradiction. Let us assume that $\zeta_{k'}$ is an Artin generator. If $g_{k'}$ is the cube of a standard positive half-twist, then let us also assume that $g_{k'}$ is rightmost among minimal complexity factorizations that are Hurwitz equivalent to ${\cal F}$. The fact that $\xi_1\xi_2\xi_3\xi_4$ is a power of $\Delta$ implies that $\prod_{j'=1}^{k'} \left(\zeta_{j'}\right)_{\left(\omega_{j'}\right)}$ is left-dual to $\prod_{j'=k+1}^{4} \left(\zeta_{j'}\right)_{\left(\omega_{j'}\right)}$. However, if $g_{k'}$ is the cube of a standard positive half-twist, then Lemma~\ref{ltwoHurwitzmovescubepositivehalftwist} contradicts the assumption that $g_{k'}$ is rightmost, and if $g_{k'}$ is the cube of a non-standard positive half-twist, then Lemma~\ref{ltwoHurwitzmovescubepositivehalftwist} contradicts the assumption that ${\cal F}$ is a minimal complexity factorization. 

Therefore, the statement is established.
\end{proof}

We state the geometric reformulation of Theorem~\ref{tmainclassificationfactorizationscubepositivehalftwist}.

\begin{theorem}
\label{tuniqueisotopyclasscusp}
Let $C,C'\subseteq \mathbb{CP}^2$ be degree three simple Hurwitz curves. If $C$ and $C'$ each have a single cusp and no other singularities, then $C$ is isotopic to $C'$.
\end{theorem}
\begin{proof}
The braid monodromy factorization of a degree three Hurwitz curve with a single cusp and no other singularities is a factorization of $\Delta^2$ into three positive half-twists and one cube of a positive half-twist. Furthermore, the Hurwitz equivalence class of the braid monodromy factorization of the curve completely determines the isotopy class of the curve.

Theorem~\ref{tmainclassificationfactorizationscubepositivehalftwist} implies that the braid monodromy factorization of $C$ is Hurwitz equivalent to the braid monodromy factorization of $C'$. Therefore, the statement is established.
\end{proof}

We conclude this subsection by showing that there is a unique isotopy class of degree three simple Hurwitz curves in $\mathbb{CP}^2$ with precisely one tacnode and no other singularities. In Subsection~\ref{sssingularityconstraints}, we will show that a degree three simple Hurwitz curve with at least one tacnode has precisely one tacnode and no other singularities. In particular, this will complete the classification of degree three simple Hurwitz curves in $\mathbb{CP}^2$ with a tacnode.

We will prove that there is a unique Hurwitz equivalence class of factorizations of $\Delta^2$ into the fourth power of a positive half-twist and two positive half-twists. Firstly, we establish a technical statement concerning minimal complexity factorizations of $\Delta^2$ into three powers of positive half-twists.

\begin{lemma}
\label{lthreefactorsconstraint}
We adopt Setup~\ref{setfactorization} with $k=3$. Let us assume that ${\cal F}$ is a factorization of $\Delta^2$ that does not consist of three squares of positive half-twists. Firstly, if ${\cal F}$ is a minimal complexity factorization, then we can assume that either $\xi_1\xi_2$ or $\xi_2\xi_3$ is indivisible by $\Delta$, after possibly replacing ${\cal F}$ by a Hurwitz equivalent factorization of the same complexity. 

If $1\leq j\leq 3$, then we define $\delta_j = 1$ if $g_j$ is the power of a non-standard positive half-twist and $\delta_j = 0$ if $g_j$ is the power of a standard positive half-twist. If $\xi_1\xi_2$ is indivisible by $\Delta$, then we define $\epsilon\in \{0,1\}$ such that $\omega\left(\xi_1\xi_2\right) = \omega\left(\xi_1\right)+\omega\left(\xi_2\right) - \epsilon$ (Proposition~\ref{pomegaalmostadditive}). If $\xi_1\xi_2$ is indivisible by $\Delta$, then \[2 = \delta_1+\delta_2+\delta_3 + \epsilon\] and \[\omega\left(\rho_3\right) = 2 + \omega\left(\rho_1\right)+\omega\left(\rho_2\right) + \delta_3 - e_3.\]
\end{lemma}
\begin{proof}
We adopt Setup~\ref{setfactorization}. Firstly, we prove that either $\xi_1\xi_2$ or $\xi_2\xi_3$ is indivisible by $\Delta$, after possibly replacing the factorization ${\cal F}$ with a Hurwitz equivalent factorization of the same complexity. Let us assume that both $\xi_1\xi_2$ and $\xi_2\xi_3$ are divisible by $\Delta$. If each factor in ${\cal F}$ is the power of a non-standard positive half-twist, then Lemma~\ref{lcomplexityreducingglobalconjugation} implies that there is a global conjugation move that reduces the complexity of the factorization ${\cal F}$, which contradicts the minimality of the complexity $c\left({\cal F}\right)$. In particular, we can assume that at least one factor in ${\cal F}$ is the power of a standard positive half-twist. 

Firstly, Lemma~\ref{lcomplexityzerofactorization} implies that all three factors in ${\cal F}$ cannot be powers of standard positive half-twists. Let us assume that precisely one factor in ${\cal F}$ is the power of a standard positive half-twist. In this case, either the factor that is the power of a standard positive half-twist is at one end of ${\cal F}$ or $g_2$ is the power of a standard positive half-twist. If the factor is at one end of ${\cal F}$, then we will show that there is a global conjugation move that reduces the complexity of the factorization ${\cal F}$, which contradicts the minimality of the complexity $c\left({\cal F}\right)$. Indeed, we can assume without loss of generality that $g_1$ is the power of a standard positive half-twist, and $g_1 = \sigma_{i_1}^{e_1}$ (in the notation of Setup~\ref{setfactorization}). Let $\tau'$ be the cut left-divisor of $\xi_2$ such that $\omega\left(\tau'\right) = 1$. Firstly, Proposition~\ref{pdualbraidproduct} implies that $\left(\tau'\right)^{-1}g_1\tau'$ is the power of a standard positive half-twist, since $\xi_1\xi_2$ is divisible by $\Delta$. Proposition~\ref{pGarsidepowerreduction} implies that the absolute value of the Garside power of $\left(\tau'\right)^{-1}g_2\tau'$ is strictly less than the absolute value of the Garside power of $g_2$, since $\xi_1\xi_2$ is divisible by $\Delta$. Proposition~\ref{pGarsidepowerreductionconsistency} implies that the absolute value of the Garside power of $\left(\tau'\right)^{-1}g_3\tau'$ is at most the absolute value of the Garside power of $g_3$, since $\xi_2\xi_3$ is divisible by $\Delta$. We conclude that the complexity $c\left(\left(\tau'\right)^{-1}{\cal F}\tau'\right)<c\left({\cal F}\right)$ in this case.

On the other hand, let us assume that $g_2$ is the power of a standard positive half-twist, and $g_2 = \left(\sigma_{i_2}^{e_2}\right)_{\left(\omega\left(\tau_1\right)\right)}$ (in the notation of Setup~\ref{setfactorization}). We will show that there is a global conjugation move that replaces ${\cal F}$ with a factorization of the same complexity for which both the second and third factors are powers of standard positive half-twists (in which case, the product of the second and third factors is indivisible by $\Delta$). Proposition~\ref{pGarsidepowerreduction} implies that the absolute value of the Garside power of $\left(\sigma_{i_2}\right)_{\left(\omega\left(\tau_1\right)\right)}g_3\left(\sigma_{i_2}\right)_{\left(\omega\left(\tau_1\right)\right)}^{-1}$ is strictly less than the absolute value of the Garside power of $g_3$, since $\xi_2\xi_3$ is divisible by $\Delta$. Of course, the absolute value of the Garside power of $\left(\sigma_{i_2}\right)_{\left(\omega\left(\tau_1\right)\right)}g_1\left(\sigma_{i_2}\right)_{\left(\omega\left(\tau_1\right)\right)}^{-1}$ is at most one more than the absolute value of the Garside power of $g_1$ in any case. We conclude that the complexity $c\left(\left(\sigma_{i_2}\right)_{\left(\omega\left(\tau_1\right)\right)}{\cal F}\left(\sigma_{i_2}\right)_{\left(\omega\left(\tau_1\right)\right)}^{-1}\right)\leq c\left({\cal F}\right)$. If we apply a composition of such global conjugation moves, then we ultimately obtain a factorization for which both the second and third factors are powers of standard positive half-twists. Indeed, each such move does not change the Garside power of $g_2$ and reduces the absolute value of the Garside power of $g_3$. We have thus established the first statement if there is precisely one factor in ${\cal F}$ that is the power of a standard positive half-twist.

If precisely two factors in ${\cal F}$ are powers of standard positive half-twists, then they cannot be adjacent factors since we have assumed that both $\xi_1\xi_2$ and $\xi_2\xi_3$ are divisible by $\Delta$. In this case, $g_1$ and $g_3$ are powers of standard positive half-twists. Lemma~\ref{ladjacentstandardpositivehalftwist} implies that neither $g_1$ nor $g_3$ is a standard positive half-twist, since the complexity of ${\cal F}$ is minimal and both $\xi_1\xi_2$ and $\xi_2\xi_3$ are divisible by $\Delta$. In particular, either $e_1 = 2$ or $e_3 = 2$, since $6 = e_1+e_2+e_3$ according to Proposition~\ref{pnumberoffactorsinfactorization}. Let us assume without loss of generality that $e_1 = 2$, in which case $g_1 = \sigma_{i_1}^2$. Proposition~\ref{padjacentsquarestandardpositivehalftwist} implies that $g_2 = \sigma_{i_1}^{-1}\sigma_{i_2'}^{e_2}\sigma_{i_1}$, where $i_2\neq i_2'\in \{1,2\}$, since the complexity of ${\cal F}$ is minimal and $\xi_1\xi_2$ is divisible by $\Delta$. We can write $g_3 = \sigma_{i_3'}^{e_3}$, where $i_3\neq i_3'\in \{1,2\}$, since we have assumed that $g_3$ is the power of a standard positive half-twist. 

The equation $\Delta^2 = g_1g_2g_3$ can now be rewritten as $\Delta^2 = \sigma_{i_1}\sigma_{i_2'}^{e_2}\sigma_{i_1}\sigma_{i_3'}^{e_3}$. We have $i_2'=i_3'$ since the product $\xi_2\xi_3$ is divisible by $\Delta$. In particular, we obtain the equation $\Delta = \sigma_{i_1'}\sigma_{i_2}^{e_2-1}\sigma_{i_3'}^{e_3-1}$. However, this is a contradiction since $e_2\neq 2$ (due to the hypothesis that ${\cal F}$ does not consist of three squares of positive half-twists). We have thus established the first statement if there are precisely two factors in ${\cal F}$ that are the powers of standard positive half-twists. We have now established in all cases that either $\xi_1\xi_2$ or $\xi_2\xi_3$ is indivisible by $\Delta$, after possibly replacing ${\cal F}$ by a Hurwitz equivalent factorization of the same complexity.

We now establish the formulas. If $\xi_1\xi_2$ is indivisible by $\Delta$, then $\xi_1\xi_2$ is dual to $\xi_3$, since the product $\xi_1\xi_2\xi_3$ is a power of $\Delta$. Proposition~\ref{ppseudosymmetryduality} implies that $\omega\left(\xi_1\xi_2\right) = \omega\left(\xi_3\right)$ and Proposition~\ref{pomegaalmostadditive} implies that $\omega\left(\xi_1\xi_2\right) = \omega\left(\xi_1\right) + \omega\left(\xi_2\right) - \epsilon$ for some $\epsilon\in \{0,1\}$. Furthermore, Proposition~\ref{pomeganumberpowerpositivehalftwist} implies that $\omega\left(\xi_j\right) = 2\omega\left(\tau_j\right) + e_j-\delta_j$ for each $1\leq j\leq 3$. Finally, $\xi_1\xi_2\xi_3 = \Delta^{2 + \omega_3}$ (in the notation of Setup~\ref{setfactorization}) since ${\cal F}$ is a factorization of $\Delta^2$. We deduce that $\omega\left(\xi_3\right) = 2 + \omega\left(\tau_1\right)+\omega\left(\tau_2\right)+\omega\left(\tau_3\right)$.
The equation $6 = e_1+e_2+e_3$ now implies that \[2 = \delta_1+\delta_2+\delta_3 +\epsilon\] and \[\omega\left(\rho_3\right) = 2 + \omega\left(\rho_1\right) + \omega\left(\rho_2\right) + \delta_3 - e_3.\] Therefore, the statement is established.
\end{proof}

The following technical statement is an addendum to Lemma~\ref{lthreefactorsconstraint}.

\begin{lemma}
\label{lthreefactorsreduction}
We adopt Setup~\ref{setfactorization} with $k = 3$. Let us assume that ${\cal F}$ is a minimal complexity factorization of $\Delta^2$ that does not consist of three squares of positive half-twists (i.e., the hypothesis of Lemma~\ref{lthreefactorsconstraint} is satisfied) and $\xi_1\xi_2$ is indivisible by $\Delta$. If there are two powers of non-standard positive half-twists in ${\cal F}$, then they must be $g_1$ and $g_2$.
\end{lemma}
\begin{proof}
We will show that $g_3$ and precisely one of $g_j$ for $j\in \{1,2\}$ cannot simultaneously be powers of non-standard positive half-twists. Let us assume, for a contradiction, that this is false. We will consider the case where $g_2$ and $g_3$ are powers of non-standard positive half-twists since the proof in the other case is similar. We will show that we can apply a complexity reducing global conjugation move to ${\cal F}$, which contradicts the hypothesis that ${\cal F}$ is a minimal complexity factorization.

Firstly, Lemma~\ref{lthreefactorsconstraint} implies that $\epsilon = 0$. In particular, Lemma~\ref{lcutdivisorcharacterization} implies that the index of the last generator of $\xi_1$ is equal to the index of the first generator of $\xi_2$. Let $\tau'$ be the cut left-divisor of $\xi_2$ with $\omega\left(\tau'\right) = 1$. We consider cases according to whether $\tau'$ is a single generator or the product of two distinct generators. In general, Proposition~\ref{pGarsidepowerreduction} implies that the absolute value of the Garside power of $\left(\tau'\right)^{-1}g_2\tau'$ is strictly less than the absolute value of the Garside power of $g_2$. Furthermore, the hypothesis that $\xi_1\xi_2$ is dual to $\xi_3$ and Lemma~\ref{ltripleproductindivisibleDelta} imply that $\xi_2\xi_3$ is divisible by $\Delta$, since $g_2$ is the power of a non-standard positive half-twist. 

If $\tau'$ is a single generator, then $\left(\tau'\right)^{-1}g_1\tau' = g_1$ is the power of a standard positive half-twist. Proposition~\ref{pGarsidepowerreductionconsistency} implies that the absolute value of the Garside power of $\left(\tau'\right)^{-1}g_3\tau'$ is at most the absolute value of the Garside power of $g_3$, since $\xi_2\xi_3$ is divisible by $\Delta$. We deduce that $c\left(\left(\tau'\right)^{-1}{\cal F}\tau'\right)<c\left({\cal F}\right)$ if $\tau'$ is a single generator.

If $\tau'$ is a product of two distinct generators, then the absolute value of the Garside power of $\left(\tau'\right)^{-1}g_1\tau'$ is at most one more than the absolute value of the Garside power of $g_1$. Proposition~\ref{pdualbraidproduct}, Lemma~\ref{lcharacterizationdualpositivebraids} and Proposition~\ref{pGarsidepowerreduction} imply that the absolute value of the Garside power of $\left(\tau'\right)^{-1}g_3\left(\tau'\right)$ is strictly less than the absolute value of the Garside power of $g_3$. We deduce that $c\left(\left(\tau'\right)^{-1}{\cal F}\tau'\right)<c\left({\cal F}\right)$ if $\tau'$ is a product of two distinct generators. 

Therefore, the statement is established. 
\end{proof}

We are now prepared to classify factorizations of $\Delta^2$ into powers of positive half-twists, where one of the factors is the fourth power of a positive half-twist, and the other factors are positive half-twists.

\begin{theorem}
\label{tmainclassificationfactorizationsfourthpowerpositivehalftwist}
If ${\cal F}$ is a factorization of $\Delta^2$ into two positive half-twists and one fourth power of a positive half-twist, then ${\cal F}$ is Hurwitz equivalent to the standard factorization $\left(\sigma_1^4,\sigma_1^{-1}\sigma_2\sigma_1,\sigma_2^{-1}\sigma_1\sigma_2\right)$ (in Definition~\ref{dstandardfactorization}).
\end{theorem}
\begin{proof}
We adopt Setup~\ref{setfactorization}. Let us assume, for a contradiction, that ${\cal F}$ is a factorization of minimal complexity that is not Hurwitz equivalent to the standard factorization $\left(\sigma_1^4,\sigma_1^{-1}\sigma_2\sigma_1,\sigma_2^{-1}\sigma_1\sigma_2\right)$. Lemma~\ref{lcomplexityzerofactorization} implies that the complexity $c\left({\cal F}\right)>0$. Lemma~\ref{lthreefactorsconstraint} implies that we can assume that either $\xi_1\xi_2$ or $\xi_2\xi_3$ is indivisible by $\Delta$, after possibly replacing ${\cal F}$ by a Hurwitz equivalent minimal complexity factorization. We assume without loss of generality that $\xi_1\xi_2$ is indivisible by $\Delta$.

We adopt notation from Lemma~\ref{lthreefactorsconstraint}. The formulas imply that either exactly two of the factors in ${\cal F}$ are powers of non-standard positive half-twists and $\epsilon = 0$, or exactly one of the factors in ${\cal F}$ is the power of a non-standard positive half-twist and $\epsilon = 1$. 

Firstly, we assume that exactly two of the factors in ${\cal F}$ are powers of non-standard positive half-twists and $\epsilon = 0$. If $g_1$ and $g_2$ are powers of non-standard positive half-twists, then Lemma~\ref{lthreefactorsconstraint} implies that $e_3 = 2 + \omega\left(\rho_1\right)+\omega\left(\rho_2\right)$. In particular, $e_3 = 4$ and $\omega\left(\rho_1\right) = 1 = \omega\left(\rho_2\right)$. We deduce that ${\cal F} = \left(\sigma_1^{-1}\sigma_2\sigma_1,\sigma_2^{-1}\sigma_1\sigma_2,\sigma_1^4\right)$ after possibly applying a global conjugation move by $\Delta$ to ${\cal F}$, using Lemma~\ref{lcharacterizationdualpositivebraids} since $\xi_1\xi_2$ is dual to $\xi_3$. In this case, ${\cal F}$ is Hurwitz equivalent to the standard factorization by a sequence of two Hurwitz moves. Lemma~\ref{lthreefactorsreduction} implies that exactly one of $g_1$ and $g_2$ cannot be the power of a non-standard positive half-twist.

Secondly, we assume that exactly one of the factors in ${\cal F}$ is the power of a non-standard positive half-twist and $\epsilon = 1$. If $g_j$ is the power of a non-standard positive half-twist for $j\in \{1,2\}$, then Lemma~\ref{lthreefactorsconstraint} implies that $\omega\left(\rho_j\right) = e_3 - 2$. In particular $e_3 = 4$. We can assume that $g_3 = \sigma_1^4$ after possibly applying a global conjugation move by $\Delta$ to ${\cal F}$. We deduce that if $j = 1$, then ${\cal F} = \left(\sigma_1^{-2}\sigma_2\sigma_1^2, \sigma_2,\sigma_1^4\right)$, and if $j = 2$, then ${\cal F} = \left(\sigma_2,\sigma_1^{2}\sigma_2\sigma_1^{-2},\sigma_1^4\right)$, using Lemma~\ref{lcharacterizationdualpositivebraids} since $\xi_1\xi_2$ is dual to $\xi_3$. In either case, ${\cal F}$ is Hurwitz equivalent to the standard factorization. 

If $j = 3$, then Lemma~\ref{lthreefactorsconstraint} implies that $0<\omega\left(\rho_3\right) = 3 - e_3$. In particular, $e_3 = 1$ and $\omega\left(\rho_3\right) = 2$. If $e_1 = 4$, then we can assume that $g_1 = \sigma_1^4$ after possibly applying a global conjugation move by $\Delta$ to ${\cal F}$. We deduce that ${\cal F} = \left(\sigma_1^4,\sigma_2,\sigma_1^{2}\sigma_2\sigma_1^{-2}\right)$ using Lemma~\ref{lcharacterizationdualpositivebraids} since $\xi_1\xi_2$ is dual to $\xi_3$. If $e_2 = 4$, then we can assume that $g_2 = \sigma_1^4$ after possibly applying a global conjugation move by $\Delta$ to ${\cal F}$. We deduce that ${\cal F} = \left(\sigma_2,\sigma_1^4,\sigma_1^{-2}\sigma_2\sigma_1^2\right)$ using Lemma~\ref{lcharacterizationdualpositivebraids} since $\xi_1\xi_2$ is dual to $\xi_3$. In each case, ${\cal F}$ is Hurwitz equivalent to the standard factorization. 

Therefore, the statement is established.
\end{proof}

We state the geometric reformulation of Theorem~\ref{tmainclassificationfactorizationsfourthpowerpositivehalftwist}.

\begin{theorem}
\label{tuniqueisotopyclasstacnode}
Let $C,C'\subseteq \mathbb{CP}^2$ be degree three simple Hurwitz curves. If $C$ and $C'$ each have a single tacnode and no other singularities, then $C$ is isotopic to $C'$.
\end{theorem}
\begin{proof}
The braid monodromy factorization of a degree three Hurwitz curve with a single tacnode and no other singularities is a factorization of $\Delta^2$ into two positive half-twists and one fourth power of a positive half-twist. Furthermore, the Hurwitz equivalence class of the braid monodromy factorization of the curve completely determines the isotopy class of the curve.

Theorem~\ref{tmainclassificationfactorizationsfourthpowerpositivehalftwist} implies that the braid monodromy factorization of $C$ is Hurwitz equivalent to the braid monodromy factorization of $C'$. Therefore, the statement is established.
\end{proof}

We summarize the results in this subsection in the following statement.

\begin{theorem}
\label{tuniqueisotopyclassnodalcuspidal}
Let $C,C'\subseteq \mathbb{CP}^2$ be degree three simple Hurwitz curves with the same number of $A_n$-singularities for each positive integer $n\geq 1$. Let us consider one of the following constraints on the numbers of $A_n$-singularities for the curves $C$ and $C'$.
\begin{description}
\item[(i)] The curves are nodal (i.e., the only singularities are nodal singularities).
\item[(ii)] The curves each have a single singularity and it is a cuspidal singularity.
\item[(iii)] The curves each have a single singularity and it is a tacnodal singularity.
\end{description}
If one of the constraints is satisfied, then $C$ is isotopic to $C'$.
\end{theorem}
\begin{proof}
The statement is a consequence of Theorem~\ref{tuniqueisotopyclassnumberofnodes}, Theorem~\ref{tuniqueisotopyclasscusp}, and Theorem~\ref{tuniqueisotopyclasstacnode}.
\end{proof}

\subsection{A complete classification of the singularities of degree three simple Hurwitz curves in $\mathbb{CP}^2$}
\label{sssingularityconstraints}

The goal of this subsection is to establish a complete set of constraints on the numbers of $A_n$-singularities of a degree three simple Hurwitz curve $C\subseteq \mathbb{CP}^2$. The singularity formula (Lemma~\ref{lsingularityformula}) implies that there are at most two $A_3$-singularities (cuspidal singularities) in $C$ and at most one $A_n$-singularity in $C$ for $n\geq 4$. 

In this subsection, we will strengthen the singularity formula to show that there is at most one $A_3$-singularity in $C$ and no $A_n$-singularities in $C$ for $n\geq 5$. We will show that an $A_3$-singularity and an $A_2$-singularity (a nodal singularity) cannot simultaneously occur in a degree three simple Hurwitz curve $C$. We will also show that an $A_4$-singularity (a tacnodal singularity) and an $A_2$-singularity cannot simultaneously occur in a degree three simple Hurwitz curve $C$. The combination of these statements is Theorem~\ref{tmainsingularities} in the Introduction.

Theorem~\ref{tmainsingularities} is equivalent to the algebraic statement that there is no factorization of $\Delta^2$ into powers of positive half-twists in the braid group $B_3$ with at most three factors, unless either the three factors are all squares of positive half-twists, or the three factors consist of one fourth power of a positive half-twist and two positive half-twists. We will use our theory of duality of positive braids and conjugacy classes in $B_3$ from Section~\ref{sB3background} in order to establish this algebraic statement. The characterization of the possible singularities of a degree three simple Hurwitz curve $C\subseteq \mathbb{CP}^2$ will complete our classification of isotopy classes of degree three simple Hurwitz curves in $\mathbb{CP}^2$. Indeed, Theorem~\ref{tmainsingularities} shows that the combination of Theorem~\ref{tuniqueisotopyclasssmooth} and Theorem~\ref{tuniqueisotopyclassnodalcuspidal} is equivalent to Theorem~\ref{tmainHurwitzcurve} in the Introduction.

We are now prepared to establish the first constraint on the number of factors in a factorization of $\Delta^2$ into powers of positive half-twists. 

\begin{lemma}
\label{lnottwofactors}
A factorization of $\Delta^2$ into powers of positive half-twists has at least three factors.
\end{lemma}
\begin{proof}
Let us assume, for a contradiction, that there exists a factorization of $\Delta^2$ into a product of at most two powers of positive half-twists. Firstly, the factorization must have precisely two factors. Indeed, $\Delta^2$ itself is not the power of a positive half-twist, since $\Delta^2$ is central and constitutes its own conjugacy class.

Let us assume, for a contradiction, that ${\cal F} = \left(g_1,g_2\right)$ is a factorization of $\Delta^2$ into a product of two powers of positive half-twists, with minimal complexity. We adopt Setup~\ref{setfactorization} with $k=2$. Firstly, Lemma~\ref{lcomplexityzerofactorization} implies that the complexity $c\left({\cal F}\right)>0$. Indeed, either $e_1\geq 3$ or $e_2\geq 3$ since $e_1 + e_2 = 6$ (in the notation of Setup~\ref{setfactorization}). We will show that we can always apply a complexity reducing global conjugation move to the factorization ${\cal F}$, and this contradiction of the minimality of the complexity $c\left({\cal F}\right)$ will complete the proof.

Of course, $\xi_1$ is dual to $\xi_2$ (in the notation of Setup~\ref{setfactorization}). If $g_1$ and $g_2$ are both powers of non-standard positive half-twists, then Lemma~\ref{lcomplexityreducingglobalconjugation} implies that we can apply a complexity reducing global conjugation move to the factorization ${\cal F}$. We conclude that at least one of the factors is the power of a standard positive half-twist. 

Let us assume without loss of generality that $g_1$ is the power of a standard positive half-twist, and write $g_1 = \sigma_{i_1}^{e_1}$. Proposition~\ref{pdualbraidproduct} implies that $\rho_2$ begins with $\sigma_{i_1'}\sigma_{i_1}$ where $i_1\neq i_1'\in \{1,2\}$. Lemma~\ref{lcharacterizationdualpositivebraids} implies that $\tau_2$ ends with $\left(\sigma_{i_1}\right)_{\left(\omega\left(\tau_2\right)\right)}$ and $\left(\sigma_{i_1}\right)_{\left(\omega\left(\tau_2\right)\right)}$ is a cut right-divisor of $\tau_2$. Proposition~\ref{pGarsidepowerreduction} now implies that the absolute value of the Garside power of $\sigma_{i_1}g_2\sigma_{i_1}^{-1}$ is strictly less than the absolute value of the Garside power of $g_2$. Of course, $\sigma_{i_1}g_1\sigma_{i_1}^{-1} = g_1$. We conclude that the complexity $c\left(\sigma_{i_1}{\cal F}\sigma_{i_1}^{-1}\right)<c\left({\cal F}\right)$. Therefore, the statement is established.
\end{proof}

We state the equivalent geometric formulation of Lemma~\ref{lnottwofactors}.

\begin{theorem}
\label{tnottwosingularities}
A degree three simple Hurwitz curve $C\subseteq \mathbb{CP}^2$ has at least three $A_n$-singularities.
\end{theorem}
\begin{proof}
The braid monodromy factorization of a degree three simple Hurwitz curve $C\subseteq \mathbb{CP}^2$ is a factorization ${\cal F}$ of $\Delta^2$ into a product of $k$ powers of positive half-twists, where $k$ is the number of $A_n$-singularities in $C$. Lemma~\ref{lnottwofactors} implies that $k\geq 3$. Therefore, the statement is established.
\end{proof}

We now constrain factorizations of $\Delta^2$ into three powers of positive half-twists.

\begin{lemma}
\label{lnotthreefactors}
A factorization of $\Delta^2$ into powers of positive half-twists cannot have precisely three factors, where one factor is the cube of a positive half-twist. 
\end{lemma}
\begin{proof}
Let us assume, for a contradiction, that there is a factorization of $\Delta^2$ into three powers of positive half-twists where one factor is the cube of a positive half-twist. Let ${\cal F} = \left(g_1,g_2,g_3\right)$ be such a factorization of $\Delta^2$ with minimal complexity. Proposition~\ref{pnumberoffactorsinfactorization} implies that ${\cal F}$ is a factorization of $\Delta^2$ into one cube of a positive half-twist, one square of a positive half-twist, and one positive half-twist. We will show that such a factorization of $\Delta^2$ into powers of positive half-twists does not exist.  Lemma~\ref{lthreefactorsconstraint} implies that we can assume that either $\xi_1\xi_2$ or $\xi_2\xi_3$ is indivisible by $\Delta$, after possibly replacing ${\cal F}$ by a Hurwitz equivalent minimal complexity factorization. We assume without loss of generality that $\xi_1\xi_2$ is indivisible by $\Delta$.

We adopt notation from Lemma~\ref{lthreefactorsconstraint}. The formulas in Lemma~\ref{lthreefactorsconstraint} imply that either exactly two of the factors in ${\cal F}$ are powers of non-standard positive half-twists and $\epsilon = 0$, or exactly one of the factors in ${\cal F}$ is the power of a non-standard positive half-twist and $\epsilon = 1$. 

Firstly, if two of the factors in ${\cal F}$ are powers of non-standard positive half-twists, then they cannot be $g_1$ and $g_2$. Indeed, otherwise Lemma~\ref{lthreefactorsconstraint} implies that $e_3 = 2+\omega\left(\rho_1\right)+\omega\left(\rho_2\right)\geq 4$, which is a contradiction. Lemma~\ref{lthreefactorsreduction} now implies that there cannot be exactly two factors in ${\cal F}$ that are powers of non-standard positive half-twists.

Secondly, we assume that exactly one of the factors in ${\cal F}$ is the power of a non-standard positive half-twist and $\epsilon = 1$. If $g_j$ is the power of a non-standard positive half-twist for $j\in \{1,2\}$, then Lemma~\ref{lthreefactorsconstraint} implies that $0<\omega\left(\rho_j\right) = e_3 - 2$. In particular, $e_3 = 3$. However, in this case $\xi_1\xi_2$ cannot be dual to $\xi_3$ by Lemma~\ref{lcharacterizationdualpositivebraids} and Proposition~\ref{pcutclosureconjugate}. If $g_3$ is the power of a non-standard positive half-twist, then Lemma~\ref{lthreefactorsconstraint} implies that $0<\omega\left(\rho_3\right) = 3 - e_3$, which is a contradiction.

Therefore, the statement is established.
\end{proof}

We state the equivalent geometric formulation of Lemma~\ref{lnotthreefactors}.

\begin{theorem}
\label{tnotthreesingularities}
A degree three simple Hurwitz curve $C\subseteq \mathbb{CP}^2$ does not have precisely three $A_n$-singularities, unless either all three $A_n$-singularities are $A_2$-singularities (nodal singularities), or one $A_n$-singularity is an $A_4$-singularity and the other two $A_n$-singularities are $A_1$-singularities.
\end{theorem}
\begin{proof}
If a degree three simple Hurwitz curve $C\subseteq \mathbb{CP}^2$ has precisely three $A_n$-singularities, then the braid monodromy factorization of $C$ is a factorization ${\cal F}$ of $\Delta^2$ into a product of three powers of positive half-twists. If the conclusion is not satisfied, then ${\cal F}$ is a factorization of $\Delta^2$ into three powers of positive half-twists, where one factor is the cube of a positive half-twist. Lemma~\ref{lnotthreefactors} states that such a factorization of $\Delta^2$ does not exist. Therefore, the statement is established. 
\end{proof}

We can interpret the combination of Theorem~\ref{tnottwosingularities} and Theorem~\ref{tnotthreesingularities} as a complete set of constraints on the possible singularities of a degree three simple Hurwitz curve $C\subseteq \mathbb{CP}^2$. The following statement together with the singularity formula (Lemma~\ref{lsingularityformula}) is equivalent to Theorem~\ref{tmainsingularities} in the Introduction.

\begin{theorem}
\label{tsingularityconstraints}
We have the following constraints on the singularities of a degree three simple Hurwitz curve $C\subseteq \mathbb{CP}^2$:
\begin{description}
\item[(i)] $C$ has at most one $A_3$-singularity (cuspidal singularity).
\item[(ii)] $C$ does not simultaneously have an $A_2$-singularity (nodal singularity) and an $A_3$-singularity (cuspidal singularity).
\item[(iii)] $C$ does not simultaneously have an $A_2$-singularity (nodal singularity) and an $A_4$-singularity (tacnodal singularity).
\item[(iv)] $C$ does not have an $A_n$-singularity for $n\geq 5$.
\end{description}
\end{theorem}
\begin{proof}
The statements are a consequence of Theorem~\ref{tnottwosingularities}, Theorem~\ref{tnotthreesingularities} and the singularity formula (Lemma~\ref{lsingularityformula}).
\end{proof}

If a degree three simple Hurwitz curve $C\subseteq \mathbb{CP}^2$ has three $A_2$-singularities (nodal singularities), one $A_4$-singularity (tacnodal singularity) and two $A_1$-singularities, or four or more $A_n$-singularities, then we have established that the isotopy class of $C$ is uniquely determined by the numbers of $A_n$-singularities in $C$ for each positive integer $n\geq 1$ (Theorem~\ref{tuniqueisotopyclasssmooth} in Subsection~\ref{ssclassificationfactorizationspositivehalftwists} and Theorem~\ref{tuniqueisotopyclassnodalcuspidal} in Subsection~\ref{ssclassificationfactorizationspowerspositivehalftwists}). We summarize the main results in this paper in the following statement.

\begin{theorem}
\label{tfinal}
Let $C\subseteq \mathbb{CP}^2$ be a degree three simple Hurwitz curve. If $\nu_n$ is the number of $A_n$-singularities in $C$ for each positive integer $n\geq 1$, then the tuple $\left(\nu_n\right)_{n=1}^{6}$ is a standard tuple (Definition~\ref{dstandardfactorization}). Conversely, if $\left(\nu_n\right)_{n=1}^{6}$ is a standard tuple, then there is a unique degree three simple Hurwitz curve $C\subseteq \mathbb{CP}^2$ up to isotopy such that the number of $A_n$-singularities in $C$ is equal to $\nu_n$ for each positive integer $n\geq 1$.
\end{theorem}
\begin{proof}
The first statement is equivalent to Theorem~\ref{tmainsingularities}. The second statement is equivalent to the combination of Theorem~\ref{tuniqueisotopyclasssmooth} and Theorem~\ref{tuniqueisotopyclassnodalcuspidal}.
\end{proof}

We remark that Theorem~\ref{tfinal} is equivalent to the combination of Theorem~\ref{tmainHurwitzcurve} and Theorem~\ref{tmainsingularities} in the Introduction.

\section{The classification of genus one Lefschetz fibrations over $\mathbb{S}^2$}
\label{sLefschetzfibrations}
In this section, we establish a complete classification of genus one Lefschetz fibrations over $\mathbb{S}^2$. Indeed, we establish that the number of singular fibers in a genus one Lefschetz fibration over $\mathbb{S}^2$ is a complete invariant of the isomorphism class of the Lefschetz fibration, and that this number is divisible by $12$. We remark that this has already been established by Moishezon-Livne~\cite{moishezongenusonelefschetzfibrations}, but we give a short independent proof of the classification based on the results in this paper.

Indeed, we recall that the classification of genus one Lefschetz fibrations over $\mathbb{S}^2$ is equivalent to the classification of factorizations of the identity into positive Dehn twists in the modular group $\text{Mod}\left(\mathbb{T}^2\right)\cong \text{SL}_2\left(\mathbb{Z}\right)$, where $\mathbb{T}^2$ is a genus one complex curve. The main observation is that there is a close relationship between the braid group $B_3$ and the modular group $\text{SL}_2\left(\mathbb{Z}\right)$, and we will use this relationship and the results in this paper in order to classify genus one Lefschetz fibrations over $\mathbb{S}^2$. 

Firstly, we recall that \[\text{SL}_2\left(\mathbb{Z}\right) = \left\{\begin{bmatrix} a & b \\ c & d \end{bmatrix}: a,b,c,d\in \mathbb{Z}\text{ and } ad-bc = 1\right\}\] is the group of $2\times 2$ integer matrices with determinant $+1$. The center of $\text{SL}_2\left(\mathbb{Z}\right)$ is the cyclic group $\left\langle -I \right\rangle$ of order $2$ generated by $-I$, where $I$ denotes the $2\times 2$ identity matrix in $\text{SL}_2\left(\mathbb{Z}\right)$. We also recall that \[\text{PSL}_2\left(\mathbb{Z}\right) = \text{SL}_2\left(\mathbb{Z}\right)/\left\langle - I \right\rangle.\] Let \[S = \begin{bmatrix} 1 & 1 \\ 0 & 1 \end{bmatrix},\text{ } T = \begin{bmatrix} 1 & 0 \\ -1 & 1 \end{bmatrix}.\] We view $S$ and $T$ as elements of $\text{SL}_2\left(\mathbb{Z}\right)$ or $\text{PSL}_2\left(\mathbb{Z}\right)$, depending on context. A \textit{positive Dehn twist} in $\text{SL}_2\left(\mathbb{Z}\right)$ is an element of the conjugacy class of $S$ (or $T$) in $\text{SL}_2\left(\mathbb{Z}\right)$. 

We recall the following classical result. 

\begin{proposition}
\label{propfpPSL_2}
We have an isomorphism of groups \[\text{SL}_2\left(\mathbb{Z}\right) = \left\langle S,T : STS = TST,\text{ }\left(ST\right)^6 = I \right\rangle.\]
\end{proposition}

Note that Proposition~\ref{propfpPSL_2} implies that $\text{PSL}_2\left(\mathbb{Z}\right)$ admits the finite presentation \[\text{PSL}_2\left(\mathbb{Z}\right) = \left\langle S,T : STS = TST,\text{ }\left(ST\right)^3 = I \right\rangle.\] Proposition~\ref{propfpPSL_2} also implies that $\text{PSL}_2\left(\mathbb{Z}\right)$ is isomorphic to the free product $\mathbb{Z}/2\mathbb{Z}\ast \mathbb{Z}/3\mathbb{Z}$. Indeed, $STS$ generates a copy of $\mathbb{Z}/2\mathbb{Z}$ in $\text{PSL}_2\left(\mathbb{Z}\right)$ and $ST$ generates a copy of $\mathbb{Z}/3\mathbb{Z}$ in $\text{PSL}_2\left(\mathbb{Z}\right)$. Furthermore, $\text{PSL}_2\left(\mathbb{Z}\right)$ is generated by $STS$ and $ST$ with these relations. A traditional proof of these finite presentations is based on the ping-pong lemma, which gives a criterion for a group to be a free product of subgroups in terms of the actions of the group and its subgroups on a set (see II.B of~\cite{topicsggtpierre}). We remark that the approach to classifying genus one Lefschetz fibrations over $\mathbb{S}^2$ in~\cite{moishezongenusonelefschetzfibrations} is based on the isomorphism $\text{PSL}_2\left(\mathbb{Z}\right)\cong \mathbb{Z}/2\mathbb{Z}\ast \mathbb{Z}/3\mathbb{Z}$ and the short exact sequence \[1\to \left\langle \pm I \right\rangle \to \text{SL}_2\left(\mathbb{Z}\right)\to \text{PSL}_2\left(\mathbb{Z}\right)\to 1.\]

However, in this paper, we will use our results on the braid group $B_3$ in order to classify genus one Lefschetz fibrations over $\mathbb{S}^2$. Indeed, let us define $\phi:B_3\to \text{SL}_2\left(\mathbb{Z}\right)$ in terms of generators in the Artin presentation of $B_3$ by the rules $\phi\left(\sigma_1\right) = S$ and $\phi\left(\sigma_2\right) = T$. Proposition~\ref{propfpPSL_2} implies that $\phi$ is a well-defined surjective homomorphism.

\begin{proposition}
\label{propB_3PSL_2}
The sequence of maps \[1\to \left\langle \Delta^4 \right\rangle \to B_3\to \text{SL}_2\left(\mathbb{Z}\right)\to 1\] is a short exact sequence. Furthermore, every positive Dehn twist in $\text{SL}_2\left(\mathbb{Z}\right)$ is the image of a positive half-twist in $B_3$.
\end{proposition}
\begin{proof}
The first statement follows from Proposition~\ref{propfpPSL_2}, since $\phi\left(\Delta\right) = STS$. Indeed, $\Delta^4$ is central in $B_3$, and thus the subgroup of $B_3$ normally generated by $\Delta^4$ equals the subgroup of $B_3$ generated by $\Delta^4$. The second statement follows from the surjectivity of $\phi:B_3\to \text{SL}_2\left(\mathbb{Z}\right)$, which implies that the image of the conjugacy class of $\sigma_1$ in $B_3$ is equal to the conjugacy class of $\phi\left(\sigma_1\right) = S$ in $\text{SL}_2\left(\mathbb{Z}\right)$. 
\end{proof}

Finally, we recall that if $k\geq 0$ is a nonnegative integer, then there is a unique Hurwitz equivalence class of factorizations of $\Delta^{k}$ into positive half-twists in the braid group $B_3$ (Theorem~\ref{tfactorizationpositivehalftwistsstandard}). The following statement is an application of Theorem~\ref{tfactorizationpositivehalftwistsstandard}, Proposition~\ref{propB_3PSL_2}, and the functoriality of Hurwitz equivalence (Proposition~\ref{pHurwitzequivalencefunctorial}).

\begin{theorem}
The number of singular fibers is a complete invariant of the isomorphism class of a genus one Lefschetz fibration over $\mathbb{S}^2$. Furthermore, the number of singular fibers in a genus one Lefschetz fibration over $\mathbb{S}^2$ is divisible by $12$. 
\end{theorem}
\begin{proof}
The braid monodromy factorization of a genus one Lefschetz fibration over $\mathbb{S}^2$ is a factorization of the identity into positive Dehn twists in the group $\text{Mod}\left(\mathbb{T}^2\right)\cong \text{SL}_2\left(\mathbb{Z}\right)$, where the number of factors is equal to the number of singular fibers in the Lefschetz fibration. Furthermore, the isomorphism class of a genus one Lefschetz fibration over $\mathbb{S}^2$ is completely determined by the Hurwitz equivalence class of its braid monodromy factorization. Let ${\cal F}$ be a factorization of the identity in $\text{SL}_2\left(\mathbb{Z}\right)$ into positive Dehn twists. We will show that ${\cal F}$ is Hurwitz equivalent to the standard factorization $I\equiv \left(S,T,S,\dots,S,T,S\right)$ of the identity $I\in \text{SL}_2\left(\mathbb{Z}\right)$ with $3k$ factors, where $k$ is a positive integer divisible by $4$.

Indeed, Proposition~\ref{propB_3PSL_2} implies that ${\cal F}$ lifts to a factorization ${\cal F}'$ of $\Delta^{k}$ into positive half-twists in the braid group $B_3$, for some positive integer $k$ divisible by $4$. However, Theorem~\ref{tfactorizationpositivehalftwistsstandard} and Lemma~\ref{lstandardfactorizationsHurwitzequivalent} imply that ${\cal F}'$ is Hurwitz equivalent to the standard factorization $\Delta^{k}\equiv \left(\sigma_1,\sigma_2,\sigma_1,\dots,\sigma_1,\sigma_2,\sigma_1\right)$ into $3k$ standard positive half-twists in $B_3$. The image of the standard factorization of $\Delta^k$ in $\text{SL}_2\left(\mathbb{Z}\right)$ is the factorization $I\equiv \left(S,T,S,\dots,S,T,S\right)$ with $3k$ factors. The functoriality of Hurwitz equivalence (Proposition~\ref{pHurwitzequivalencefunctorial}) now implies that ${\cal F}$ is Hurwitz equivalent to the standard factorization $I\equiv \left(S,T,S,\dots,S,T,S\right)$. Therefore, the statement is established. 
\end{proof}

\bibliography{References}
\bibliographystyle{plain}
\end{document}